\documentclass[letterarticle,11pt]{article}
\usepackage[margin=1in]{geometry}
\usepackage{color, xcolor, graphicx, appendix}
\usepackage[bookmarks, colorlinks=true, plainpages = false, citecolor= darkblue, linkcolor = chicago-maroon, anchorcolor = red, urlcolor = chicago-maroon]{hyperref}
\definecolor{chicago-maroon}{RGB}{128,0,0}
\usepackage{url}
\usepackage{amsmath, amsfonts, amsthm, amssymb, bm, bbm, verbatim, dsfont, mathtools, mathrsfs}
\definecolor{darkblue}{rgb}{0.0, 0.0, 0.55}
\usepackage{etoolbox}
\usepackage{almendra}
\usepackage{authblk}
\usepackage{array}
\usepackage{multirow}
\usepackage{multicol}

\usepackage{float}
\usepackage{pgf,tikz}
\usetikzlibrary{arrows,shapes.arrows,shapes.geometric,shapes.multipart,fit,automata,decorations.pathmorphing,positioning}
\tikzset{
	>=stealth',
	true/.style={
		rectangle,
		draw=black, very thick,
		text width=6.5em,
		minimum height=2em,
		text centered,
		fill=gray, opacity = 0.5},
	punkt/.style={
		rectangle,
		rounded corners,
		draw=black, very thick,
		text width=6.5em,
		minimum height=2em,
		text centered},
	est/.style={
		circle,
		draw=black, very thick,
		text centered},
	shade/.style={
		circle,
		draw=black, very thick, fill=gray!50,
		text centered},
	weight/.style={
		circle,
		draw=black, very thick,
		text width=6.5em,
		minimum height=2em,
		text centered},
	pil/.style={
		->,
		thick,
		shorten <=2pt,
		shorten >=2pt,},
	double/.style={
		<->,
		thick,
		shorten <=2pt,
		shorten >=2pt,},
	dash/.style={
		dashed,
		thick,
		shorten <=2pt,
		shorten >=2pt,},
	dashdouble/.style={
		<->,
		dashed,
		thick,
		shorten <=2pt,
		shorten >=2pt,}
}

\usepackage{tikz-cd}
\makeatletter 
\newcolumntype{C}[1]{>{\centering\arraybackslash}p{#1}}

\usepackage{epstopdf}
\usepackage{fullpage}
\usepackage{algorithm}
\usepackage{algorithmicx}
\usepackage{subcaption}

\usepackage[nameinlink]{cleveref}
\usepackage[round]{natbib}

\usepackage[T1]{fontenc}    
\usepackage{url}            
\usepackage{booktabs}     
\usepackage{nicefrac}
\usepackage{xcolor}
\usepackage{pdfpages}
\usepackage{scalerel}
\usepackage{dashrule}
\usepackage{lmodern}
\usepackage{fancyhdr}
\usepackage{tabu}
\usepackage{enumitem}

\def\IIFF{\mathbb{IF}}
\def\IF{\mathsf{IF}}

\def\var{\mathrm{var}}
\def\cov{\mathrm{cov}}

\def\Holder{\text{H\"{o}lder}}

\def\op{\mathrm{op}}

\renewcommand{\[}{\left[}

\renewcommand{\hat}{\widehat}
\renewcommand{\tilde}{\widetilde}

\usepackage{xr,xspace}
\usepackage{todonotes}

\usepackage{caption,subcaption,soul}
\usepackage{algpseudocode}
\makeatletter
\usepackage{romanbar}

\theoremstyle{plain}
\newtheorem{theorem}{Theorem}
\newtheorem{lemma}{Lemma}
\newtheorem{proposition}{Proposition}
\newtheorem{corollary}{Corollary}

\theoremstyle{definition}

\newtheorem{example}{Example}

\newtheorem{assumption}{Assumption}

\newtheorem{problem}{Problem}

\newtheorem{remark}{Remark}

\newtheorem*{remark*}{Remark}

\usepackage{algorithm}
\usepackage{algpseudocode}

\usepackage{xspace, prettyref}

\newcommand{\diff}{{\mathrm d}}

\newrefformat{eq}{(\ref{#1})}
\newrefformat{chap}{Chapter~\ref{#1}}
\newrefformat{sec}{Section~\ref{#1}}
\newrefformat{alg}{Algorithm~\ref{#1}}
\newrefformat{fig}{Fig.~\ref{#1}}
\newrefformat{tab}{Table~\ref{#1}}
\newrefformat{rmk}{Remark~\ref{#1}}
\newrefformat{clm}{Claim~\ref{#1}}
\newrefformat{def}{Definition~\ref{#1}}
\newrefformat{cor}{Corollary~\ref{#1}}
\newrefformat{lmm}{Lemma~\ref{#1}}
\newrefformat{prop}{Proposition~\ref{#1}}
\newrefformat{prob}{Problem~\ref{#1}}
\newrefformat{app}{Appendix~\ref{#1}}
\newrefformat{hyp}{Hypothesis~\ref{#1}}
\newrefformat{thm}{Theorem~\ref{#1}}

\newcommand{\sfc}{{\mathsf{c}}}

\newcommand{\sfp}{{\mathsf{p}}}

\newcommand{\sfD}{{\mathsf{D}}}

\newcommand{\sfR}{{\mathsf{R}}}

\newcommand{\calA}{{\mathcal{A}}}
\newcommand{\calB}{{\mathcal{B}}}
\newcommand{\calC}{{\mathcal{C}}}
\newcommand{\calD}{{\mathcal{D}}}
\newcommand{\calE}{{\mathcal{E}}}

\newcommand{\calG}{{\mathcal{G}}}

\newcommand{\calI}{{\mathcal{I}}}

\newcommand{\calK}{{\mathcal{K}}}

\newcommand{\calM}{{\mathcal{M}}}
\newcommand{\calN}{{\mathcal{N}}}
\newcommand{\calO}{{\mathcal{O}}}
\newcommand{\calP}{{\mathcal{P}}}

\newcommand{\calR}{{\mathcal{R}}}
\newcommand{\calS}{{\mathcal{S}}}
\newcommand{\calT}{{\mathcal{T}}}
\newcommand{\calU}{{\mathcal{U}}}
\newcommand{\calV}{{\mathcal{V}}}

\newcommand{\calX}{{\mathcal{X}}}
\newcommand{\calY}{{\mathcal{Y}}}

\renewcommand{\tilde}{\widetilde}
\renewcommand{\hat}{\widehat}

\newcommand{\bfi}{{\mathbf{i}}}

\newcommand{\bbB}{{\mathbb{B}}}

\newcommand{\bbE}{{\mathbb{E}}}

\newcommand{\bbN}{{\mathbb{N}}}

\newcommand{\bbP}{{\mathbb{P}}}

\newcommand{\bbR}{{\mathbb{R}}}

\newcommand{\bbU}{{\mathbb{U}}}

\usepackage{tcolorbox}

\makeatletter
\def\ubar#1{\underline{\sbox\tw@{$#1$}\dp\tw@\z@\box\tw@}}
\makeatother

\def\leftarrowCirc{\hbox{$\leftarrow$}\kern-1.5pt\hbox{$\circ$}}
\def\Circrightarrow{\hbox{$\circ$}\kern-1.5pt\hbox{$\rightarrow$}}
\def\Circleftarrow{\hbox{$\circ$}\kern-1.5pt\hbox{$\leftarrow$}}
\def\rightarrowCirc{\hbox{$\rightarrow$}\kern-1.5pt\hbox{$\circ$}}

\def\u{\mathrm{u}}

\newcommand\ind{\mathrm{ind}}

\usepackage{namedtensor}

\usepackage{orcidlink}
\usepackage[normalem]{ulem}
\usepackage{setspace}
\usepackage{bbold}
\usepackage{bibunits}

\makeatletter
\newcounter{mybibunit}
\AtBeginEnvironment{bibunit}{\stepcounter{mybibunit}}
\renewcommand{\hyper@natlinkstart}[1]{%
  \Hy@backout{#1}%
  \hyper@linkstart{cite}{cite.\themybibunit @#1}%
  \def\hyper@nat@current{#1}%
}
\renewcommand{\hyper@natlinkbreak}[2]{%
  \hyper@linkend#1\hyper@linkstart{cite}{cite.\themybibunit @#2}%
}
\renewcommand{\hyper@natanchorstart}[1]{%
  \Hy@raisedlink{\hyper@anchorstart{cite.\themybibunit @#1}}%
}
\makeatother

\usepackage{placeins}
\usetikzlibrary{arrows.meta,positioning,calc}

\def\Mobius{M\"{o}bius}

\def\nuis{\mathrm{nuis}}

\def\frakT{\mathfrak{T}}
\def\frakr{\mathfrak{r}}
\def\frakm{\mathfrak{m}}
\def\frakC{\mathfrak{C}}
\def\nt{\textrm{non-tree}}

\def\mytitle{Stabilized Higher-Order Influence Functions: \\
Statistical Theory of a Class of Bilinear Forms}

\begin{document}

\title{\mytitle}

\author[1]{Na Liu\thanks{\href{ln101213@sjtu.edu.cn}{
ln101213@sjtu.edu.cn}}}

\author[2]{Chang Li\thanks{\href{xue6bq@virginia.edu}{xue6bq@virginia.edu}}}

\author[3]{Yujia Gu\thanks{\href{guyj@mail.tsinghua.edu.cn}{guyj@mail.tsinghua.edu.cn}}}

\author[1,4]{Lin Liu\orcidlink{0000-0002-9883-7962}\thanks{\href{linliu@sjtu.edu.cn}{linliu@sjtu.edu.cn} Yujia Gu and Lin Liu are co-corresponding authors that are alphabetically ordered. This manuscript improves, and therefore, supersedes our previous draft \citet{liu2023hoif}.}}

\affil[1]{School of Mathematical Sciences, Shanghai Jiao Tong University}

\affil[2]{Department of Statistics, University of Virginia}

\affil[3]{Department of Statistics and Data Science, Tsinghua University}

\affil[4]{Institute of Natural Sciences, MOE--LSC, CMA--Shanghai, SJTU--Yale Joint Center for Biostatistics and Data Science, Shanghai Jiao Tong University}
    
\date{\today}
    
\maketitle

\thispagestyle{empty}

\vspace{-1em}
    
\begin{abstract}
Higher-order influence functions, introduced in a series of articles \citep{robins2008higher, van2014higher, robins2016technical, liu2017semiparametric}, are a unified framework for constructing rate-optimal point estimates of a class of statistical functionals under various complexity-reducing assumptions on the posited statistical model that generates the observed data. Although higher-order (influence functions) estimators are theoretically appealing, they have very limited practical uptake compared to their first-order counterparts. The original higher-order estimators proposed in \citet{robins2008higher} and \citet{robins2017minimax} involve nonparametric density estimation of multi-dimensional covariates, a highly nontrivial statistical and computational problem on its own. The density estimator is, in turn, used in the evaluation of the inverse population Gram matrix $\Omega$ of a set of $k$-dimensional basis transformations of covariates. There, $k$ is allowed to be as large as $o (n^{2})$. To partially address this potential shortcoming, \citet{liu2017semiparametric} restrict $k$ to $o (n)$ and instead estimates $\Omega$ directly using the inverse sample Gram matrix estimator, but computed from an independent sample often obtained by sample-splitting. \citet{liu2017semiparametric} refer to this alternative estimator as the empirical higher-order estimator. Although the empirical higher-order estimator bypasses density estimation, it suffers from numerical instability due to potentially inverting a large-dimensional sample Gram matrix. In this article, for a class of bilinear forms/functionals that often appear in substantive fields such as economics, epidemiology, and clinical medicine, we propose a new stabilized higher-order estimator without sample splitting, which exhibits more stable finite-sample performance compared to the empirical higher-order estimator, and more importantly, we prove that this new class of higher-order estimators enjoys similar statistical guarantees to those of \citet{liu2017semiparametric}.
\end{abstract}
    

\onehalfspacing

\allowdisplaybreaks

\newpage

\setcounter{page}{1}

\section{Introduction}
\label{sec:intro}

One of the unique features of modern statistics, which distinguishes itself from other related areas such as machine learning or AI, is the enormous interest in learning about smooth (statistical) functionals of the possibly infinite-dimensional probabilistic model that generates the observed data, instead of the model itself \citep{bickel1988estimating, ritov1990achieving, van1991differentiable, robins2008higher}. In this article, a functional is a mapping $\psi: \calP \to \bbR$, from the underlying statistical model, denoted by $\calP$, to the reals $\bbR$. A statistical model $\calP$ contains all possible observed-data-generating probability distributions, posited by a statistician. 

A functional $\psi$ is said to be smooth in the sense of \citet{van1991differentiable}, that is, the pathwise derivative of $\psi (\bbP)$, along any parametric submodel $\{\bbP_{t}: \bbP_{0} = \bbP\} \subseteq \calP$, allows the following representation: $\left. \frac{\diff}{\diff t} \right|_{t = 0} \psi (\bbP_{t}) = \bbE \{\IF_{\psi} \cdot g (O)\}$, where $g$ is the score function associated with the parametric submodel $\bbP_{t}$, and $\IF_{\psi} \equiv \IF_{\psi, \bbP}$ is the (first-order) efficient influence function (IF) (or canonical gradient) of $\psi$ locally at $\bbP \in \calP$ \citep{fisher2021visually, hines2022demystifying}. It is required that $\bbE (\IF_{\psi}) = 0$. Examples of smooth functionals abound: in causal inference, common target parameters of interest, such as the average treatment effect, the average treatment effect on the treated, and the quantile treatment effect, are all smooth functionals under standard causal identification conditions (consistency, positivity, and ignorability) \citep{robins1994estimation, hahn1998role, van2006targeted}; in (conditional) independence testing, dependence measures such as the generalized covariance measure \citep{shah2020hardness, niu2024reconciling} and $f$-divergence \citep{kandasamy2015nonparametric}, are also smooth functionals. This article tackles the problem of constructing ``good'' estimators for smooth functionals, abbreviated as the functional estimation problem.

A natural attempt to estimate $\psi$ is to start with the ``plug-in'' estimator $\hat{\psi}_{0} = \psi (\hat{\bbP})$, where $\hat{\bbP}$ is some estimator of $\bbP$. However, a common theme in the functional estimation literature tells us that the plug-in estimator $\hat{\psi}_{0}$ has a sub-optimal convergence rate in many settings \citep{robins2009semiparametric, balakrishnan2026fundamental}. The sub-optimality of the plug-in estimator is often resulting from its large bias. A popular (and almost dominating) paradigm in the current statistics literature is to use the IF of $\psi$, $\IF_{\psi}$, to de-bias the plug-in estimator $\hat{\psi}_{0}$ \citep{robins1994estimation, van2006targeted, chernozhukov2018double, ray2020semiparametric, breunig2025double}. We refer to these debiased estimators based solely on $\IF_{\psi}$ as first-order estimators \citep{liu2026asymptotic}, which include popular methods in applications such as \emph{double machine learning/Neyman orthogonal scores} \citep{chernozhukov2018double} and \emph{targeted maximum likelihood estimation (TMLE)} \citep{van2006targeted}. In many settings, however, first-order estimators are still sub-optimal in terms of convergence rates \citep{liu2024assumption, bonvini2024doubly, liu2023root}. To resolve the potential sub-optimality of $\hat{\psi}_{1}$, building upon von Mises functional expansions and higher-order scores \citep{pfanzagl1983asymptotic, small1989projection, waterman1996projected, bobkov2024fisher}, \citet{robins2008higher, robins2016technical} develop a general framework called higher-order influence functions (HOIFs) that generalize the concept of IF from first-order to higher-orders, for constructing (nearly) rate-optimal estimators in various settings. We also refer to \citet{bonhomme2026higher} for related development in higher-order Neyman orthogonal scores and to \citet{van2021higher} for related development in higher-order TMLE (HOTMLE). TMLE-related methodologies generally enjoy favorable finite sample performance. The HOIF framework has also been used to construct estimators in infinite-dimensional problems \citep{kennedy2024minimax, bonvini2026fast} and to understand certain irregular estimators \citep{bonvini2024doubly}.

One key insight of \citet{robins2008higher, robins2016technical} is to find an approximation of the target functional $\psi$ by a particular bilinear form $\tilde{\psi}_{k} = \mu^{\top} \Sigma^{-1} \eta$, where $\Sigma = \bbE (X X^{\top})$ is the $k \times k$ population Gram matrix of some random vector $X$, and $\mu$ and $\eta$ are two $k$-dimensional vectors that can be written respectively as $\mu = \bbE (X A)$ and $\eta = \bbE (X Y)$ for some random variables $A$ and $Y$ (see Section~\ref{sec:setup} for details). Once this step is accomplished, HOIFs offer a unified scheme of constructing rate-optimal estimators of the bilinear form $\tilde{\psi}_{k}$, and the resulting estimators are higher-order $U$-statistics. Fortunately, many of the aforementioned examples of smooth functionals indeed admit such a bilinear form approximation; again, see Section~\ref{sec:setup} for concrete examples (Examples~\ref{eg:quad}--\ref{eg:gcm}). As will be clear in Section~\ref{sec:setup}, in this article, we will directly take the bilinear form $\tilde{\psi}_{k}$ as the target parameter $\psi$ without worrying about the bias due to this bilinear approximation. The HOIF estimators proposed in \citet{robins2008higher, robins2016technical} allow the dimension $k$ to be as large as of order $o (n^{2})$, but require a nonparametric density estimation step when estimating $\Sigma$ from data. Given the difficulty of nonparametric density estimation even in moderate dimensions, the original HOIF estimators have not been routinely deployed in practice.

When the dimension $k$ is of order $o (n)$ so $\Sigma^{-1}$ can be consistently estimated by the inverse of the sample Gram matrix $\hat{\Sigma}^{-1}$, \citet{liu2017semiparametric} proposed the so-called empirical HOIF estimators, simply estimating $\Sigma^{-1}$ by $\hat{\Sigma}^{-1}$ from a separate sample independent of the main sample used to estimate $\psi$. To our knowledge, the empirical HOIF estimator remains the only $\sqrt{n}$-consistent and asymptotically normal ($\sqrt{n}$-CAN) estimator of $\psi$ when $k = o (n)$, without imposing any assumption on the covariate density. \citet{zhang2026higher} extend both versions of HOIF estimators to parameters defined implicitly via $Z$/$M$-estimation problems, such as quantile treatment effects and expected shortfalls. More recently, \citet{newey2018cross} initiated the research program on constructing estimators motivated by but much simpler than HOIFs, with follow-up work in various directions \citep{kennedy2023towards, mcgrath2026nuisance, mcclean2026double}. Finally, we also mention in passing that similar bias correction ideas have also been independently developed in the econometric and general mathematical statistics literature \citep{newey2004twicing, cattaneo2018kernel, cattaneo2018inference, cattaneo2019two, breunig2024adaptive, cavaliere2024bootstrap, koltchinskii2022bootstrap}.

Although empirical HOIF estimators neither estimate nor impose any complexity-reducing assumptions on the density of $X$, inverting the sample Gram matrix $\hat{\Sigma}$ may easily lead to numerical instability when $k$ is relatively large compared to $n$. This potential instability has been documented in the simulation studies conducted in \citet{liu2020nearly, liu2017semiparametric, liu2024assumption, zhang2026higher}, being a primary reason for the limited practical uptake of empirical HOIF estimators. However, it is less well known that \citet{liu2020nearly} also proposed alternative empirical HOIF estimators (at orders $2$ and $3$, in retrospect) that still estimate the population Gram matrix $\Sigma$ by its sample analog $\hat{\Sigma}$ but from the \emph{same} sample used to compute the final $U$-statistic estimator. Since sample splitting is not used, \citet{liu2020nearly} did not prove that this new alternative HOIF estimator works in theory; interestingly, for the same reason, these alternative HOIF estimators exhibit much improved finite-sample performance compared to the original ones proposed in \citet{liu2017semiparametric}, in particular in terms of their numerical stability, even allowing practitioners to choose $k$ very close to $n$ (see Remark~\ref{rem:stability} for further explanations). For completeness, this is demonstrated in Figure~\ref{fig:sim} in Section~\ref{sec:theory}, which displays the results of a simple simulation study, the setup of which is in Appendix~\ref{app:sim}.

\subsection{Our contributions}
\label{sec:contributions}

The main contribution of this article is to offer theoretical guarantees for the aforementioned alternative HOIF estimators, which we refer to as stabilized HOIF estimators. The main technical difficulty arises from the dependence of the $U$-statistic kernel on the entire sample through $\hat{\Sigma}^{-1}$ when sample splitting is not employed. To overcome this challenge, we have to deviate from the analysis strategy for the original empirical HOIF estimators taken in \citet{liu2017semiparametric} and instead perform a more meticulous analysis that involves various complex expansions and nontrivial counting.  We obtain results similar to those for the empirical HOIF estimators of \citet{liu2017semiparametric}, in the sense that the new HOIF estimators are also $\sqrt{n}$-CAN for the bilinear forms $\psi$, as long as $k = o (n)$ without any further complexity-reducing assumptions on the density of $X$.

Specifically, we bring in tools from enumerative combinatorics and graph theory \citep{shpitser2011efficient, stanley2011enumerative} to prove the bias and variance bounds for this new class of stabilized HOIF estimators. These tools were recently exploited in \citet{chen2026computing} to design efficient algorithms for exactly computing higher-order $U$-statistics. In addition, \citet{schafer2026mobius} also uses these tools to give a new combinatorial interpretation of the iterative bootstrap procedure. However, to our knowledge, these tools have not been used to establish statistical properties of estimators that involve $U$-statistics. 


\subsection{Notation}
\label{sec:notation}

Throughout the article, $\bbU_{n,j}$ denotes the $j$-th order $U$-statistic operator: for any $h: \calO_{1} \times \cdots \times \calO_{j} \to \bbR$, $\bbU_{n, j} \{h (O_{1}, \cdots, O_{j})\} \coloneqq \frac{(n - j)!}{n!} \sum_{1 \leq i_{1} \neq \cdots \neq i_{j} \leq n} h (O_{i_{1}}, \cdots, O_{i_{j}})$. We reserve $\Sigma$ and $\hat{\Sigma}$ for the population and sample Gram matrices of $X$, and write $\Omega := \Sigma^{-1}$ and $\hat{\Omega} := \hat{\Sigma}^{-1}$ for their inverses whenever they exist ($\hat{\Sigma}$ invertible almost surely under our assumptions). The identity matrix is denoted by $I$. For a random variable $W$ and $p \geq 1$, $\Vert W \Vert_{p} \coloneqq \{\bbE (|W|^{p})\}^{1/p}$ denotes the $L^{p}(\bbP)$-norm of $W$. For any sample-index subset $S \subseteq [n]$, we
write $O_{S} \coloneqq \{O_{i} : i \in S\}$, and given any positive integer $\ell$, we let $[\ell] \coloneqq \{1, \cdots, \ell\}$. We write $\{i_1, \cdots, i_k\}$ as a set including elements $i_1, \cdots, i_k$ and $(i_1, \cdots, i_k)$ as an ordered tuple, in which all elements are distinct and assigned a particular ordering. 


\subsection{Organizations}
\label{sec:organization}

The remainder of this article is structured as follows. Section~\ref{sec:setup} sets the stage by describing the problem setting, stating the regularity assumptions, and providing a brief review of the empirical HOIF estimator of \citet{liu2017semiparametric}. In Section~\ref{sec:HOIF}, we present the main result of this article, in which we first introduce the new stabilized HOIF estimators and then characterize their bias, variance, and asymptotic distribution. The theoretical results are all encapsulated in Theorem~\ref{thm:main}, the main theorem in our article. Section~\ref{sec:proof} provides a proof sketch of Theorem~\ref{thm:main}, with technical details deferred to the Appendix. 
Section~\ref{sec:conclusions} concludes the article with a discussion of future topics.
    
\section{Problem Setting and A Brief Review of Existing HOIF Estimators}
\label{sec:setup}

Let $O \coloneqq (X, A, Y)$ denote a triple of the observed random vector, where $X \in \calX \subseteq \bbR^{k}$ is a $k$-dimensional vector, $A \in \calA \subset \bbR$ and $Y \in \calY \subseteq \bbR$ denote some outcomes of interest. We assume access to $n$ i.i.d. observations $\calD \coloneqq (O_{1}, \cdots, O_{n})$, drawn from a common data-generating distribution $\bbP \in \calP$, where $\calP$ denotes the statistical model restricted by the following regularity conditions.

\begin{assumption}
\label{as:cov}
The distribution of $X$ satisfies the following:
\begin{align}
& \bbE (X^{\top} X) = O (k), \\
& \Vert X^{\top} X \Vert_{\infty} = O (k).
\end{align}
Furthermore, the eigenvalues of $\Sigma$ are strictly bounded away from $0$ and $\infty$, and $\hat{\Sigma}$ is invertible with probability $1$.
\end{assumption}

In addition, in this article, we restrict to the case $k = o (n)$. But we will state the more precise condition on $k$ in the statement of related theoretical claims. We also need to impose the following $L_{\infty}$-stability assumption on the projection on the span of $X$, as commonly done in previous work on HOIFs \citep{robins2008higher, robins2016technical, robins2017minimax, liu2017semiparametric, liu2024assumption}.

\begin{assumption}
\label{as:kernel-stability}
For every bounded measurable function $h: \calX \to \bbR$, define the following integral operator: $(\Pi h) (x) \coloneqq x^{\top} \Omega \bbE \{X h(X)\}, \ \forall \ x \in \calX$. We assume that $\Pi$ is uniformly bounded as an operator on $L_\infty (\calX)$: there exists a strictly bounded constant $C_{\Pi} < \infty$, independent of $k$ and $n$, such that
\begin{equation}
\label{eq:kernel-linfty-stability}
\|\Pi h\|_{\infty} \leq C_\Pi \|h\|_{\infty}.
\end{equation}
\end{assumption}

Finally, for convenience, we further impose the following condition on $A$ and $Y$.
\begin{assumption}
\label{as:outcomes}
Both $A$ and $Y$ are bounded almost surely.
\end{assumption}

\begin{remark}
\label{rem:regularity assumptions}
The above assumptions are made for technical convenience. For example, if we relax Assumption~\ref{as:outcomes} from boundedness to light-tailed assumptions, we need to further develop exponential and moment inequalities for higher-order $U$-statistics with unbounded kernels, which is an important research topic in applied probability on its own \citep{chakrabortty2025tail}.
\end{remark}

To focus on the essential structure, we consider the following functional as the target parameter:
\begin{equation}
\label{target}
\psi \equiv \psi (\bbP) \coloneqq \mu^{\top} \Omega \eta, \quad \text{where } \mu \coloneqq \bbE (X A), \eta \coloneqq \bbE (X Y), \Sigma \coloneqq \bbE (X X^{\top}) \text{ and } \Omega \coloneqq \Sigma^{-1}.
\end{equation}
Although $\psi$ takes a very simple bilinear form, it encapsulates many substantively important smooth functionals that appear in the literature. We use several examples to demonstrate the ubiquity of $\psi$ in statistical applications.

\begin{example}[Quadratic functional of a density]
\label{eg:quad}
Suppose that $Y \sim \bbP$ with $p$ being the probability density function of $\bbP$, the target functional is $\psi = \int_{\calY} p (y)^{2} \diff y$, and $p$ can be represented as a linear combination of $\bar{\phi}$, assumed to be orthonormal with respect to the Lebesgue measure over $\calY$. Thus, there exists $\eta \in \bbR^{k}$ such that $p (\cdot) = \eta^{\top} \bar{\phi} (\cdot)$. We identify $X \coloneqq \bar{\phi} (Y)$ and $A \equiv Y$. Then given $O = (X, A, Y)$, $\psi = \eta^{\top} \eta$ with $\Sigma = I$. This quadratic functional of a density is one of the most well-studied smooth functionals in the statistics literature \citep{bickel1988estimating}.
\end{example}

\begin{example}[Signal-to-noise ratio]
\label{eg:SNR}
Suppose that $(X, Y) \sim \bbP$, and the target functional is $\psi = \bbE \{b (X)^{2}\}$ where $b (\cdot) \coloneqq \bbE (Y \mid X = \cdot)$. We further assume that $b (\cdot) = \beta^{\top} (\cdot)$ for some $\beta \in \bbR^{k}$. Then $\psi = \beta^{\top} \Sigma \beta = \eta^{\top} \Omega \eta$, with $\beta = \Omega \eta$. Similar parameters have been extensively studied in the context of high-dimensional (generalized) linear models \citep{verzelen2018adaptive}.
\end{example}

\begin{example}[Treatment-specific counterfactual mean]
\label{eg:ate}
Suppose that $(Z, A, Y) \sim \bbP$ constitutes the observed data of an unconfounded observational study, in which $A$ is the binary treatment variable, $Y$ is an outcome of interest, and $Z$ is the baseline covariates that contain all confounders between $A$ and $Y$. The target parameter is the treatment-specific counterfactual mean $\psi = \bbE Y (1) = \bbE \{A a (Z) Y\} = \bbE \{b (Z)\}$, where $a (\cdot) \coloneqq \bbE^{-1} (A \mid Z = \cdot)$ and $b (\cdot) \coloneqq \bbE (Y \mid Z = \cdot, A = 1)$. Let $X = A \bar{\phi} (Z)$. As shown in \citet{robins2007comment, liu2017semiparametric, bruns2026augmented}, if we posit that $a (\cdot) \equiv \alpha^{\top} \bar{\phi} (\cdot)$ and $b (\cdot) \equiv \beta^{\top} \bar{\phi} (\cdot)$, where $\alpha, \beta \in \bbR^{k}$, then $\psi = \alpha^{\top} \Sigma \beta = \mu^{\top} \Omega \eta$, where $\alpha = \Omega \mu$ and $\beta = \Omega \eta$ with $\mu = \bbE (X a (Z))$ and $\eta = \bbE (X A Y)$. For implicitly defined parameters such as the quantile treatment effect and the $\alpha$-expected shortfall, \citet{zhang2026higher} also showed how to represent the estimating equation of the parameter of interest in this bilinear form.
\end{example}

\begin{example}[Generalized covariance measure]
\label{eg:gcm}
When testing the conditional independence between $A$ and $Y$ given $Z$, \citet{shah2020hardness} proposed to construct test statistics based on the generalized covariance measure $\tau = \bbE \{(A - a (Z)) (Y - b (Z))\}$, where $a (\cdot) \coloneqq \bbE (A \mid Z = \cdot)$ and $b (\cdot) \coloneqq \bbE (Y \mid Z = \cdot)$. To estimate $\tau$, the most difficult component is $\psi = \bbE \{a (Z) b (Z)\}$. In \citet{liu2020nearly}, it was shown that if both $a (\cdot) = \alpha^{\top} \bar{\phi} (\cdot)$ and $b (\cdot) = \beta^{\top} \bar{\phi} (\cdot)$ are linear combinations of $\bar{\phi}$, then by identifying $X = \bar{\phi} (Z)$, $\psi = \alpha^{\top} \Sigma \beta = \mu^{\top} \Omega \eta$, once we set $\alpha = \Omega \mu$ and $\beta = \Omega \eta$.
\end{example}


More related examples can also be found in \citet{robins2008higher, rotnitzky2021characterization, chernozhukov2022locally}. For all of the above examples, when $\Omega$ is known (referred to as the oracle case in \citet{liu2020nearly}), $\psi$ can be unbiasedly estimated by its oracle second-order influence function, which is the following second-order $U$-statistic:
\begin{equation}
\label{oracle}
\hat{\psi}_{2, k} (\Omega) \coloneqq \hat{\IIFF}_{2, 2, k} (\Omega) = \bbU_{n, 2} \{\hat{\IF}_{2, 2, k} (\Omega)\}.
\end{equation}
In contrast to the settings of \citet{robins2008higher} and \citet{liu2017semiparametric}, we consider a slightly more simplified setting in which the first-order estimator $\hat{\psi}_{1} = 0$; otherwise $\hat{\psi}_{2, k} (\Omega) = \hat{\psi}_{1} + \hat{\IIFF}_{2, 2, k} (\Omega)$. 

When $\Omega$ is unknown, one can construct the so-called empirical HOIF estimators taking the following form \citep{liu2017semiparametric}: $\hat{\psi}_{m, k} (\hat{\Omega}_{\nuis}) \coloneqq \sum_{j = 2}^{m} \hat{\IIFF}_{j, j, k} (\hat{\Omega}_{\nuis})$, where
\begin{align}
& \hat{\IIFF}_{2, 2, k} (\tilde{\Omega}) \coloneqq \bbU_{n, 2} \{\hat{\IF}_{2, 2, k} (\tilde{\Omega})\} \text{ with } \hat{\IF}_{2, 2, k} (\tilde{\Omega}) \coloneqq A_{1} X_{1}^{\top} \tilde{\Omega} X_{2} Y_{2}, \text{ and for $j = 3, 4, \cdots$} \label{empHOIF} \\
& \hat{\IIFF}_{j, j, k} (\tilde{\Omega}) \coloneqq (-1)^{j} \bbU_{n, j} \{\hat{\IF}_{j, j, k} (\tilde{\Omega})\} \text{ with } \hat{\IF}_{j, j, k} (\tilde{\Omega}) \coloneqq A_{1} X_{1}^{\top} \tilde{\Omega} \Big\{ \prod_{s = 3}^{j} (X_{s} X_{s}^{\top} - \tilde{\Sigma}) \tilde{\Omega} \Big\} X_{2} Y_{2}. \nonumber
\end{align}
Here, $\tilde{\Sigma}$ and $\tilde{\Omega}$ denote, respectively, some generic estimators of $\Sigma$ and $\Omega$; $\hat{\Omega}_{\nuis} = \hat{\Sigma}_{\nuis}^{-1}$, with $\hat{\Sigma}_{\nuis}$ the sample Gram matrix estimator computed from a separate sample $\calD_{\nuis}$ independent of $\calD$.

\begin{remark}
\label{rem:notation}
We choose the above notation convention to strictly follow earlier works on HOIFs \citep{robins2008higher, robins2016technical, liu2017semiparametric, liu2024assumption}. For example, \citet{robins2008higher} reserves the notation $\hat{\IIFF}_{j, k} (\tilde{\Omega})$ for $\hat{\IIFF}_{j, k} (\tilde{\Omega}) \coloneqq \sum_{l = 2}^{j} \hat{\IIFF}_{l, l, k} (\tilde{\Omega})$. We also choose to use $\hat{\IF}$ and $\hat{\IIFF}$ instead of $\IF$ and $\IIFF$ throughout to keep the notation more aligned with the scenario in which all $A, Y, X$ may in fact depend on some first-step nuisance estimates.
\end{remark}

In particular, \citet{liu2017semiparametric} established the following results on $\hat{\psi}_{m, k} (\hat{\Omega}_{\nuis})$. Here, we only provide the simplified version of their results and \citet{liu2017semiparametric} in fact provide more comprehensive characterizations of both the bias and variance bounds of $\hat{\psi}_{m, k} (\hat{\Omega}_{\nuis})$.

\begin{proposition}
\label{prop:emp HOIF}
Under Assumptions~\ref{as:cov}--\ref{as:outcomes}, the following results hold.
\begin{enumerate}[label = (\arabic*)]
\item The bias of $\hat{\psi}_{m, k} (\hat{\Omega}_{\nuis})$ can be bounded as follows: $|\bbE \{\hat{\psi}_{m, k} (\hat{\Omega}_{\nuis}) - \psi\}| \lesssim \Vert A \Vert_{2} \cdot \Vert Y \Vert_{2} \cdot \Big( \frac{k}{n} \Big)^{m / 2}$;

\item The variance of $\hat{\psi}_{m, k} (\hat{\Omega}_{\nuis})$ satisfies $\var \{\hat{\psi}_{m, k} (\hat{\Omega}_{\nuis})\} \lesssim \frac{1}{n} + \frac{k}{n^{2}}$ if $k \lesssim \frac{n}{\log^{3} n}$ and $m \asymp \log n$;

\item Under the same additional conditions in (2), let $\nu_{\nuis}^{2} \coloneqq \lim_{n \rightarrow \infty} n \var \{\hat{\psi}_{m, k} (\hat{\Omega}_{\nuis})\}$. $\hat{\psi}_{m, k} (\hat{\Omega}_{\nuis})$ is $\sqrt{n}$-CAN, that is: $\sqrt{n} \{\hat{\psi}_{m, k} (\hat{\Omega}_{\nuis}) - \psi\} \rightsquigarrow_{\bbP} \calN (0, \nu^{2}_{\nuis})$.
\end{enumerate}
\end{proposition}

\section{The Stabilized HOIF Estimators and Related Statistical Theory}
\label{sec:HOIF}

\subsection{The stabilized HOIF estimators and statistical guarantees}
\label{sec:theory}

As alluded to in the Introduction, although the empirical HOIF estimator $\hat{\psi}_{m, k} (\hat{\Omega}_{\nuis})$ dispenses with the need of a (nonparametric) density estimator $\hat{g}$ of $g$, it can be numerically unstable when the dimension $k$ is large compared to the sample size $n$. As demonstrated in simulation studies shown in recent work \citep{liu2017semiparametric, zhang2026higher}, the finite-sample performance of $\hat{\psi}_{m, k} (\hat{\Omega}_{\nuis})$ indeed degrades as the condition number $\rho = \rho (n) \coloneqq k / n$ increases with $k$. 

To resolve the numerical instability of $\hat{\psi}_{m, k} (\hat{\Omega}_{\nuis})$, we instead construct the following HOIF estimator:
\begin{equation}
\label{new HOIF}
\hat{\psi}_{m, k} (\hat{\Omega}) \coloneqq \sum_{j = 2}^{m} \hat{\IIFF}_{j, j, k} (\hat{\Omega}).
\end{equation}
As mentioned, the 2nd- and 3rd-order versions of $\hat{\psi}_{m, k} (\hat{\Omega})$ have appeared in the previous work of the last author of this article \citep{liu2020nearly}, but there was no theoretical proof. The sole difference between our new HOIF estimator $\hat{\psi}_{m, k} (\hat{\Omega})$ and the empirical HOIF estimator $\hat{\psi}_{m, k} (\hat{\Omega}_{\nuis})$ is that we now estimate $\Omega = \Sigma^{-1}$ by the inverse sample Gram matrix estimator $\hat{\Omega}$ not from another independent sample $\calD_{\nuis}$, but from the same sample $\calD$ used to construct the HOIF estimator. Due to the correlation induced by $\hat{\Omega}$, it is more challenging to analyze the statistical properties of $\hat{\psi}_{m, k} (\hat{\Omega})$, compared to $\hat{\psi}_{m, k} (\hat{\Omega}_{\nuis})$ in \citet{liu2017semiparametric}. Overcoming this technical challenge to obtain theoretical guarantees parallel to those in Proposition~\ref{prop:emp HOIF} is our main contribution.

\begin{remark}
\label{rem:stability}
We explain why $\hat{\psi}_{m, k} (\hat{\Omega})$ has improved stability compared to $\hat{\psi}_{m, k} (\hat{\Omega}_{\nuis})$. Intuitively, since $\hat{\Sigma}$ contains the same sample $\calD$ and enters $\hat{\psi}_{m, k} (\hat{\Omega})$ as a ``denominator'', it exhibits a self-normalization phenomenon not shared by $\hat{\psi}_{m, k} (\hat{\Omega}_{\nuis})$, as $\hat{\Omega}_{\nuis} = \hat{\Sigma}_{\nuis}^{-1}$ is computed from a different sample. We refer readers to Section~S4.3 of \citet{liu2020nearly} for further explanations.
\end{remark}

\begin{remark}
\label{rem:comp}
\citet{chen2026computing} develop an algorithm for the exact computation of $\hat{\psi}_{m, k} (\hat{\Omega})$. In particular, they showed that the exact time complexity of computing $\hat{\psi}_{m, k} (\hat{\Omega})$ is $O (n^{\kappa})$, where $\kappa$ is the treewidth of an undirected graph associated with the $U$-statistic kernel of $\hat{\psi}_{m, k} (\hat{\Omega})$. 
\end{remark}

Next, we present Theorem~\ref{thm:main}, the main and most advanced result of this article.
\begin{theorem}
\label{thm:main}
Under Assumptions~\ref{as:cov}--\ref{as:outcomes}, the following results hold.
\begin{enumerate}[label = (\arabic*)]
\item Let $J = \lceil C_{0} \log n \rceil$ for a sufficiently large fixed
constant $C_{0}$, and suppose that
$m \asymp\log n$ and $k \lesssim n / \log^{3}n$. Then
\begin{equation*}
\hat{\psi}_{m, k} (\hat{\Omega})
= \chi_{m, k} (\Omega)
+ o_{\bbP} (n^{- 1 / 2}),
\end{equation*}
where $|\bbE \{\chi_{m, k} (\Omega) - \psi\}| \lesssim (\Vert A \Vert_{2} \cdot \Vert Y \Vert_{2} + \Vert A \Vert_{\infty} \cdot \Vert Y \Vert_{2} + \Vert A \Vert_{2} \cdot \Vert Y \Vert_{\infty}) \Big( \dfrac{k m}{n} \Big)^{ \lceil \dfrac{m - 1}{4} \rceil \vee 1}$ and $\var \{\chi_{m, k} (\Omega)\} \lesssim \dfrac{1}{n} + \dfrac{k}{n^{2}}$.

\item If $m \lesssim \log n$ and $k \lesssim \frac{n}{\log^{3} n}$, let $\nu^{2} \coloneqq \lim_{n \rightarrow \infty} n \var \{\hat{\psi}_{m, k} (\hat{\Omega})\}$. $\hat{\psi}_{m, k} (\hat{\Omega})$ is $\sqrt{n}$-CAN, that is:
\begin{align*}
\sqrt{n} \{\hat{\psi}_{m, k} (\hat{\Omega}) - \psi\} \rightsquigarrow_{\bbP} \calN (0, \nu^{2}).
\end{align*}
\end{enumerate}
\end{theorem}

In Section~\ref{sec:proof} below, we will provide a proof sketch of the above theorem, to illustrate the main steps. The details of the proof are delegated to the Appendix.

\begin{remark}
\label{rem:clt}
In fact, once the bias of $\hat{\psi}_{m, k} (\hat{\Omega})$ can be shown to be $o (n^{-1 / 2})$, it is straightforward to establish the $\sqrt{n}$-CAN of $\hat{\psi}_{m, k} (\hat{\Omega})$ because $\hat{\psi}_{2, k} (\Omega)$ is an unbiased and $\sqrt{n}$-CAN estimator of $\psi$, following \citet{bhattacharya1992class}; see \citet{liu2020nearly} for a proof and \citet{chakrabortty2025tail} for some recent related progress on the probability theory side.
\end{remark}

To demonstrate the improved finite-sample performance of $\hat{\psi}_{m, k} (\hat{\Omega})$ compared to $\hat{\psi}_{m, k} (\hat{\Omega}_{\nuis})$, a simple simulation study is conducted, with the setup described in Appendix~\ref{app:sim}. Specifically, Figure~\ref{fig:sim} compares the performance between $\hat{\psi}_{m, k} (\hat{\Omega})$ and $\hat{\psi}_{m, k} (\hat{\Omega}_{\nuis})$ when $m = 3$, varying $\rho = k / n$. All summary statistics are computed based on 250 Monte Carlo runs. It is evident that the performance of $\hat{\psi}_{m, k} (\hat{\Omega}_{\nuis})$ starts to break down as $\rho$ increases, whereas $\hat{\psi}_{m, k} (\hat{\Omega})$ maintains a very stable performance even when $\rho$ is near $1$. In particular, based on Figure~\ref{fig:sim}(a), the RMSEs of $\hat{\psi}_{m, k} (\hat{\Omega})$ track those of $\hat{\psi}_{m, k} (I)$ quite well even when $\rho$ is as large as $0.7$. In a follow-up paper, we will report numerical results from a set of more comprehensive simulation studies.

\begin{figure}[htbp]
\centering
\includegraphics[width=0.9\linewidth]{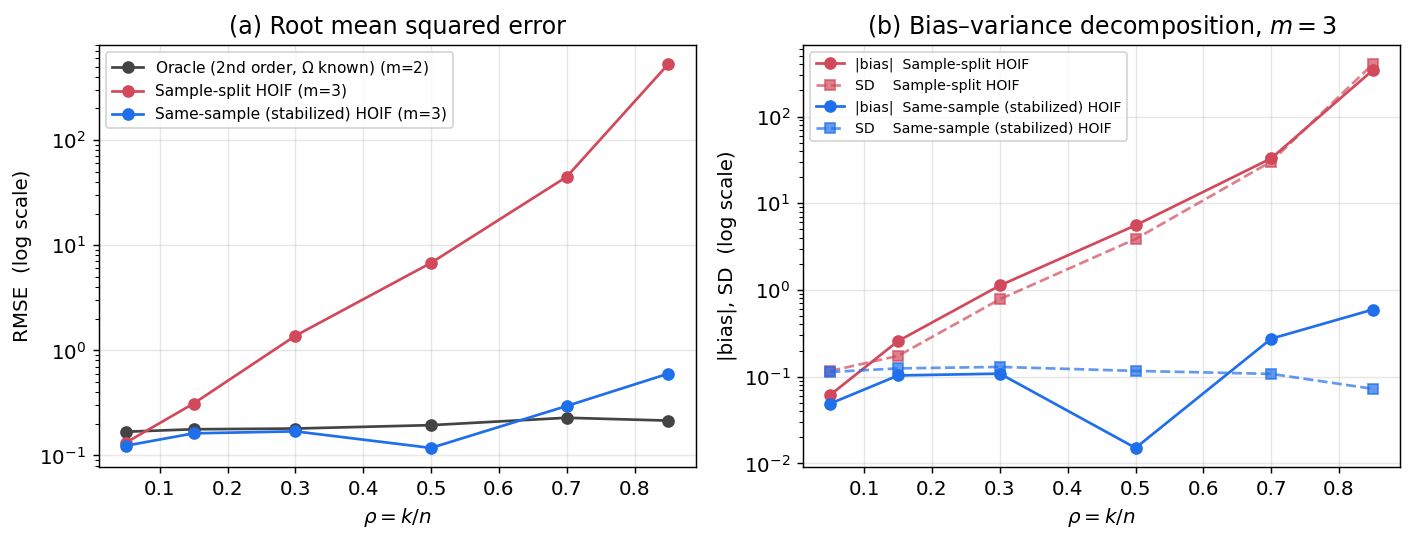}
\caption{Finite-sample comparison between the sample-split empirical HOIF estimator
$\hat{\psi}_{m, k} (\hat{\Omega}_{\nuis})$ and the same-sample stabilized HOIF estimator $\hat{\psi}_{m, k} (\hat{\Omega})$ at order $m = 3$.
Panel (a) reports the RMSE on a logarithmic scale as
$\rho = k / n$ varies. Panel (b) decomposes the error into absolute bias and standard deviation. The sample-split estimator becomes unstable as $\rho$ increases, whereas the stabilized estimator remains numerically stable.}
\label{fig:sim}
\end{figure}

\subsection{The \Mobius{} inversion decomposition}
\label{sec:mobius}

Before proving our main theorem, we record an (interesting) observation regarding $\hat{\IIFF}_{j, j, k} (\hat{\Omega})$.

\begin{lemma}
\label{lem:empirical centering}
Write $H_{i} \coloneqq (X_{i} X_{i}^{\top} - \hat{\Sigma}) \hat{\Omega} = X_{i} X_{i}^{\top} \hat{\Omega} - I$ for $i \in [n]$ (note that $H_{i}$'s appear repeatedly in the $U$-statistic kernel of $\hat{\IIFF}_{j, j, k} (\hat{\Omega})$). Then the following elementary identity holds.
\begin{equation}
\label{empirical centering}
\sum_{i = 1}^{n} H_{i} \equiv 0.
\end{equation}
\end{lemma}

With Lemma~\ref{lem:empirical centering}, by exploiting a classical tool in enumerative combinatorics, \emph{M\"obius inversion} on partition lattices \citep{stanley2011enumerative}, we can then decompose $\hat{\IIFF}_{j, j, k} (\hat{\Omega})$ into a finite sum of lower-order $U$-statistics, which will not only be useful in the proof of Theorem~\ref{thm:main} to be presented in Section~\ref{sec:proof}, but also shed some light on more detailed bias reduction mechanisms of each $\hat{\IIFF}_{j, j, k} (\hat{\Omega})$ for $j \in [m]$. 

Before presenting this \emph{\Mobius{} inversion decomposition}, we introduce some additional notation. Fix any $j \geq 3$. Let $\iota = j - 2$ and $\calR_{i_{1} i_{2}} \coloneqq H_{i_{1}} + H_{i_{2}}$. Let $\bbB_{\iota}$ consist of all finite collections $\calB = \{B_{1}, \cdots, B_{r}\}$ of pairwise disjoint subsets of $[\iota]$ such that $|B_{\nu}| \ge 2$ for every $\nu \in [r]$. The collection $\calB$ is allowed to be empty and is not required to cover $[\iota]$. For any $\calB \in \bbB_{\iota}$, order its elements according to their smallest elements and define
\begin{align}
& K_{\calB} (i_{1}, i_{2}; a_{1}, \cdots, a_{r}) \coloneqq A_{i_{1}} X_{i_{1}}^{\top} \hat{\Omega} \Big\{ \prod_{l = 1}^{\iota}
G_{\calB, l}^{i_{1} i_{2}} (a_{1}, \cdots, a_{r}) \Big\} X_{i_{2}} Y_{i_{2}}, \label{kernel} \\
& \text{ where } G_{\calB, l}^{i_{1} i_{2}} (a_{1}, \cdots, a_{r}) \coloneqq
\begin{cases}
H_{a_{\nu}}, & l \in \bigcup_{\nu = 1}^{r} B_{\nu},\\
\calR_{i_{1} i_{2}}, & l \notin \bigcup_{\nu = 1}^{r} B_{\nu}. \nonumber
\end{cases}
\end{align}
When $\calB = \emptyset$, we let
$K_{\emptyset} (i_{1}, i_{2}) \coloneqq A_{i_1} X_{i_1}^\top \hat{\Omega} \calR_{i_{1} i_{2}}^{\iota} X_{i_{2}} Y_{i_{2}}$. We are now ready to present the following lemma, a proof of which is deferred to Appendix~\ref{app:mobius}.

\begin{lemma}
\label{lem:mobius}
$\hat{\IIFF}_{j, j, k} (\hat{\Omega})$ can be decomposed as follows:
\begin{equation}
\label{Mobius}
\hat{\IIFF}_{j, j, k} (\hat{\Omega}) = \sum_{\calB \in \bbB_{\iota}} c_{\calB, n} \bbU_{n, 2 + |\calB|} (K_{\calB}), \ \text{where} \ c_{\calB, n} \coloneqq (-1)^{|\calB|} \Big\{ \prod_{B \in \calB}(|B| - 1) \Big\} \frac{(n - j) !}{(n - 2 - |\calB|) !}.
\end{equation}
The coefficients $c_{\calB,n}$ are the so-called Möbius coefficients. In particular, every term in the expansion is a $U$-statistic of order at most
$2 + \lfloor \frac{\iota}{2} \rfloor = 2 + \lfloor \frac{j - 2}{2} \rfloor$.
\end{lemma}

\begin{remark}
\label{rem:decomposition}
We illustrate Lemma~\ref{lem:mobius} with the cases $j = 3$ and $j = 4$. The lengthier explicit expansion for $j = 5$ is provided in Appendix~\ref{app:fifth-order-expansion}.
\begin{itemize}
\item When $j = 3$, $\iota = 1$. Since no non-singleton element can be formed from the singleton set $\{1\}$, the only element family is $\calB = \emptyset$. Hence,
\begin{align*}
\hat{\IIFF}_{3, 3, k} (\hat{\Omega}) & = \frac{1}{n - 2} \bbU_{n, 2} (A_{1} X_{1}^{\top} \hat{\Omega} \calR_{1 2} X_{2} Y_{2}) = \frac{1}{n - 2} \bbU_{n,2} (A_{1} X_{1}^{\top} \hat{\Omega} (H_{1} + H_{2}) X_{2} Y_{2}) \\
& = \frac{1}{n - 2} \bbU_{n, 2} (A_{1} X_{1}^{\top} \hat{\Omega} X_{1} X_{1}^{\top} \hat{\Omega} X_{2} Y_{2}) + \frac{1}{n - 2} \bbU_{n, 2} (A_{1} X_{1}^{\top} \hat{\Omega} X_{2} X_{2}^{\top} \hat{\Omega} X_{2} Y_{2}) \\
& \quad - \frac{2}{n - 2} \underbrace{\bbU_{n, 2} (A_{1} X_{1}^{\top} \hat{\Omega} X_{2} Y_{2})}_{\equiv \, \hat{\IIFF}_{2, 2, k} (\hat{\Omega})}.
\end{align*}
In particular, it is not difficult to see that the dominating terms in $\hat{\IIFF}_{3, 3, k} (\hat{\Omega})$, corresponding to the first two terms in the last equality of the above display, match the dominating bias terms of $\hat{\psi}_{2, k} (\hat{\Omega}) = \hat{\IIFF}_{2, 2, k} (\hat{\Omega})$, except that $\Omega$ is replaced by $\hat{\Omega}$. It is worth noting that the monomials of the leverage scores (terms of the form $X_{i}^{\top} \hat{\Omega} X_{j}$) up to degree $2$ appear in $\hat{\IIFF}_{3, 3, k} (\hat{\Omega})$.

\item When $j = 4$, $\iota = 2$. There are two possible element families: $\calB = \emptyset$ and  $\calB = \{\{1, 2\}\}$. Hence,
\begin{align*}
\hat{\IIFF}_{4, 4, k}(\hat{\Omega}) = \frac{1}{(n - 2) (n - 3)} \bbU_{n, 2} (A_{1} X_{1}^{\top} \hat{\Omega} \calR_{1 2}^{2} X_{2} Y_{2}) - \frac{1}{n - 3} \bbU_{n, 3} (A_{1} X_{1}^{\top} \hat{\Omega} H_{3}^{2} X_{2} Y_{2}).
\end{align*}
By elementary algebra, we have the following:
\begin{align*}
\calR_{1 2}^{2} & = (X_{1} X_{1}^{\top} \hat{\Omega} + X_{2} X_{2}^{\top} \hat{\Omega} - 2 I)^2 \\
& = X_{1} X_{1}^{\top} \hat{\Omega} X_{1} X_{1}^{\top} \hat{\Omega} + X_{1} X_{1}^{\top} \hat{\Omega} X_{2} X_{2}^{\top} \hat{\Omega} + X_{2} X_{2}^{\top} \hat{\Omega} X_{1} X_{1}^{\top} \hat{\Omega} \\
& \quad + X_{2} X_{2}^{\top} \hat{\Omega} X_{2} X_{2}^{\top} \hat{\Omega} - 4 X_{1} X_{1}^{\top} \hat{\Omega} - 4 X_{2} X_{2}^{\top} \hat{\Omega} + 4 I,
\end{align*}
and $H_{3}^{2} = (X_{3} X_{3}^{\top} \hat{\Omega} - I)^{2} = X_{3} X_{3}^{\top} \hat{\Omega} X_{3} X_{3}^{\top} \hat{\Omega} - 2 X_{3} X_{3}^{\top} \hat{\Omega} + I$. Therefore, $\hat{\IIFF}_{4, 4, k} (\hat{\Omega})$ reads as:
\begin{align*}
& \ \hat{\IIFF}_{4, 4, k} (\hat{\Omega}) \\ 
= & \ \frac{1}{(n - 2)(n - 3)} \Big\{ \bbU_{n,2} (A_{1} (X_{1}^{\top} \hat{\Omega} X_{1})^{2} X_{1}^{\top} \hat{\Omega} X_{2} Y_{2}) + \bbU_{n, 2} (A_{1} X_{1}^{\top} \hat{\Omega} X_{2} (X_{2}^{\top} \hat{\Omega} X_{2})^{2} Y_{2}) \Big\} \\
&  +  \frac{1}{(n - 2) (n - 3)} \Big\{ \bbU_{n,2} (A_{1} (X_{1}^{\top} \hat{\Omega} X_{2})^{3} Y_{2}) + \bbU_{n,2} (A_{1} X_{1}^{\top} \hat{\Omega} X_{1} X_{1}^{\top} \hat{\Omega} X_{2} X_{2}^{\top} \hat{\Omega} X_{2} Y_{2}) \Big\} \\
& - \frac{6}{(n - 2)(n - 3)} \Big\{ \bbU_{n,2} (A_{1} X_{1}^{\top} \hat{\Omega} X_{1} X_{1}^{\top} \hat{\Omega} X_{2} Y_{2}) + \bbU_{n,2} (A_{1} X_{1}^{\top} \hat{\Omega} X_{2} X_{2}^{\top} \hat{\Omega} X_{2} Y_{2}) \Big\} \\
& + \frac{n + 6}{(n - 2) (n - 3)} \underbrace{\bbU_{n,2} (A_{1} X_{1}^{\top} \hat{\Omega} X_{2} Y_{2})}_{\equiv \, \hat{\IIFF}_{2, 2, k} (\hat{\Omega})} - \frac{1}{n - 3} \bbU_{n, 3} (A_{1} X_{1}^{\top} \hat{\Omega} X_{3} X_{3}^{\top} \hat{\Omega} X_{3} X_{3}^{\top} \hat{\Omega} X_{2} Y_{2}).
\end{align*}
Similarly, $\hat{\IIFF}_{4, 4, k} (\hat{\Omega})$ matches the dominating bias terms of $\hat{\psi}_{3, k} (\hat{\Omega}) = \hat{\IIFF}_{2, 2, k} (\hat{\Omega}) + \hat{\IIFF}_{3, 3, k} (\hat{\Omega})$, except that $\Omega$ is replaced by $\hat{\Omega}$. It is also straightforward to see that the monomials of the leverage scores up to degree $3$ appear in $\hat{\IIFF}_{4, 4, k} (\hat{\Omega})$.
\end{itemize}
\end{remark}

\section{Proof Sketch of Theorem~\ref{thm:main}}
\label{sec:proof}

We now sketch the proof of Theorem~\ref{thm:main}. We focus only on the first two statements of Theorem~\ref{thm:main}, as we have argued in Remark~\ref{rem:clt} how to prove that $\hat{\psi}_{m, k} (\hat{\Omega})$ is $\sqrt{n}$-CAN. Specifically, Section~\ref{sec:bias} below provides a sketch of the bias analysis establishing part (1) of Theorem~\ref{thm:main}, whereas Section~\ref{sec:variance} sketches the proof of variance bound in Theorem~\ref{thm:main}. Before embarking on the proof sketch, in Section~\ref{sec:graph}, we first introduce a useful proof device, which we refer to as the \emph{graph-counting lemma} (Lemma~\ref{lem:graph-counting}). Lemma~\ref{lem:graph-counting} turns moment bounds of certain $U$-statistic kernels into enumerative combinatorics problems on graphs, which streamlines the proof. Throughout the bias and variance analyses, we impose $\Sigma = \Omega = I$ without loss of generality by Assumption~\ref{as:cov}.

\subsection{A graph-counting lemma}
\label{sec:graph}

The following graph-counting lemma connects the desired moment bound with the first Betti number of a graph.
\begin{lemma}
\label{lem:graph-counting}
Let $X_{1}, X_{2}, \cdots \in \bbR^{k}$ be i.i.d. random vectors. Define $\Delta_{n} \coloneqq I - \hat{\Sigma}$. For $L \geq 1$, let $i_{h}, j_{h} \in [n]$ and $\ell_{h} \in \bbN_{0}$, $h = 1, \cdots, L$, and define $K_{\bm{\ell}} \coloneqq \prod_{h = 1}^{L} X_{i_{h}}^{\top} \Delta_{n}^{\ell_{h}} X_{j_{h}}$ and $D \coloneqq \sum_{h = 1}^{L} \ell_{h}$. Let $\calA \coloneqq \{i_{h}, j_{h} : 1 \leq h \leq L\}$, $s \coloneqq |\calA|$, and let $\Gamma_{0}$ be the undirected graph associated with vertex set $V = \calA$ and one edge $(i_{h}, j_{h})$ for each $h \in [L]$. Thus $\Gamma_{0}$ is the graph associated with kernel obtained from $K_{\bm{\ell}}$ by replacing every occurrence of $\Delta_{n}$ by $I$. Self-loops are allowed and each self-loop contributes two half-edges at the same vertex. Let $v(\Gamma_{0}) \coloneqq|V|$, $e (\Gamma_{0}) := |E|$, $\kappa(\Gamma_{0}) :=$ the number of connected components in $\Gamma_{0}$, and $\frakr (\Gamma_{0}) \coloneqq e (\Gamma_{0}) + \kappa (\Gamma_{0}) - v (\Gamma_{0})$ the first Betti number of $\Gamma_{0}$. Under Assumptions~\ref{as:cov} and~\ref{as:kernel-stability}, there exists a universal constant $C < \infty$, such that:
\begin{itemize}
\item if $D = 0$, then
\begin{equation}
\label{kernel moment}
\Big| \bbE \prod_{(u, v)\in E (\Gamma_{0})} X_{u}^{\top} X_{v} \Big| \le C^{e (G)}k^{\frakr (G)},
\end{equation} 
\item if $D \geq 1$, and $\frac{C (s \vee D) k}{n} \leq \eta < 1$, then $|\bbE K_{\bm{\ell}}| \lesssim C^{e(\Gamma_{0})} k^{\frakr(\Gamma_{0})} \Big\{ \dfrac{C (s \vee D) k}{n} \Big\}^{\lceil D/2 \rceil}$.
\end{itemize}
\end{lemma}

A proof of both assertions can be found in Appendix~\ref{app:graph-counting}. For $D = 0$, the integrand in \eqref{kernel moment} is the form of $\prod_{e = (u,v) \in E} X_{u}^{\top} X_{v}$. As in \citet{chen2026computing}, we refer to a $U$-statistic kernel of this form as a multiplicative kernel. Lemma~\ref{lem:graph-counting} associates such a kernel with an undirected graph $\Gamma_{0} = (V, E)$, so that controlling a moment bound of the form \eqref{kernel moment} reduces to counting its first Betti number $\frakr (\Gamma_{0})$. For $D \geq 1$, $\Gamma_{0}$ is the base graph obtained by replacing every occurrence of $\Delta_{n}$ by $I$. Expanding the powers of $\Delta_{n}$ produces multiplicative kernels with associated undirected graphs, and the second assertion yields the corresponding additional factor in the moment bound.


Before turning to the bias analysis, we first rewrite the stabilized HOIF estimator as follows:
\begin{equation}
\label{stable}
\hat{\psi}_{m, k} (\hat{\Omega})
= \sum_{j = 2}^{m} (- 1)^{j} \bbU_{n, j}
\Big\{ A_{1} X_{1}^{\top} \hat{\Omega}
\prod_{s = 3}^{r} \Big(
X_{s} X_{s}^{\top} \hat{\Omega} - I
\Big) X_{2} Y_{2} \Big\}.
\end{equation}
The Neumann series expansion (Lemma~\ref{lem:Neumann series} in Appendix~\ref{app:matrix lemma}) gives:
\begin{equation}
\label{main Neumann}
\hat{\Omega} - I = \sum_{j = 1}^{J} \Delta_{n}^{j} + \sfR_{J}, \quad \Delta_{n} \coloneqq I - \hat{\Sigma}, \quad \sfD_{J} \coloneqq \sum_{j = 1}^{J} \Delta_{n}^{j}, \quad \hat{\Omega}_{J} \coloneqq I + \sfD_{J},
\quad \sfR_{J} \coloneqq \Delta_{n}^{J + 1} \hat{\Omega}.
\end{equation}
We then define $\hat{\psi}_{m, k}^{(J)} (\Omega)$ and $\hat{\IIFF}^{(J)}_{r , r , k} (\Omega)$, respectively, as $\hat{\psi}_{m, k} (\hat{\Omega}_{J})$ and $\hat{\IIFF}_{r , r , k} (\hat{\Omega}_{J})$. So $\hat{\psi}_{m, k}^{(J)} (\Omega) = \sum_{r = 2}^{m} \hat{\IIFF}^{(J)}_{r , r , k} (\Omega)$. Set $\chi_{m, k} (\Omega) \coloneqq \hat{\psi}_{m, k}^{(J)} (\Omega)$. For convenience, we suppress the dependence of $\chi_{m, k}(\Omega)$ on the fixed truncation level $J$. Since $\hat{\Omega} = \hat{\Omega}_{J} + \sfR_{J}$, expanding every occurrence of $\hat{\Omega}$ gives the exact identity $\hat{\psi}_{m, k}(\hat{\Omega})
= \chi_{m, k}(\Omega) + \hat{R}_{m, k, J}$. Here, $\hat{R}_{m, k, J}$ collects precisely the expansion terms containing at least one occurrence of $\sfR_{J}$. 
\begin{remark}
The identity \eqref{main Neumann} is exact for every $J$. However, in the proof, to avoid the last term $\sfR_{J}$ as it involves the nonlinear $\hat{\Omega}$, we take $J = J (n) = \lceil C_{0} \log n \rceil$ for some sufficiently large constant $C_{0}$. This choice of $J$ makes $\sfR_J$ negligible: on the event $\|\Delta_{n}\|_{\op} \le r_{n}$ and $\|\hat{\Omega}\|_{\op} \le C$ for some large enough constant $C > 0$, $\|\sfR_{J}\|_{\op} \le \|\Delta_{n}\|_{\op}^{J + 1} \|\hat{\Omega}\|_{\op} \lesssim r_{n}^{J + 1}$. At the same time, under the regime  $m \asymp \log n$ and $k \lesssim n / \log^{3} n$, this choice satisfies
\begin{equation}
\label{J require}
\frac{m J k}{n} \lesssim \frac{\log^{2} n}{\log^{3} n} = o(1).
\end{equation}
The condition \eqref{J require} is needed in various places in the proof details; e.g., Lemma~\ref{lem:bias-surviving-levels} in Appendix~\ref{app:bias iii}.
With this choice, Lemma~\ref{lem:random-neumann-remainder} shows that $\hat{R}_{m, k, J} = o_{\bbP} (n^{- 1 / 2})$. Sections~\ref{sec:bias} and~\ref{sec:variance} below
analyze, respectively, the bias and variance of
$\chi_{m, k} (\Omega)$.
\end{remark}

\subsection{Bias analysis}
\label{sec:bias}

Since $\hat{\psi}_{2, k} (\Omega)$, as defined in \eqref{oracle}, is unbiased for $\psi$, we can represent the bias of $\chi_{m, k} (\Omega)$ as:
\begin{equation}
\label{bias}
\calB_{m, k} \coloneqq  \bbE \{\chi_{m, k} (\Omega) - \psi\} = \bbE \{\chi_{m, k} (\Omega) - \hat{\psi}_{2, k} (\Omega)\} = \bbE \{\chi_{m, k} (\Omega) - \hat{\psi}_{2, k} (I)\}.
\end{equation}
We divide the bias analysis into the following steps, with details deferred to Appendix~\ref{app:bias}.

\begin{enumerate}[label = \roman*.]
\item The first step rewrites $\calB_{m, k}$ by applying Lemma~\ref{lem:bias-binomial-representation} and Lemma~\ref{lem:bias-degree-decomposition} presented later in this subsection in a sequence, up to the point that $\calB_{m, k}$ can be decomposed into a summation of terms $\calM_{c}^{(J)}$ defined in \eqref{eq:bias-degree-component}. The essential idea is to ``linearize'' $\hat{\Omega}$ by the Neumann series expansion (Lemma~\ref{lem:Neumann series} in Appendix~\ref{app:matrix lemma}).

 
\item In the second step, we further refine the representation of $\calB_{m, k}$ obtained in \textbf{Step i}. Specifically, Lemma~\ref{lem:bias-low-degree-cancellation}, to be presented later, demonstrates that many $\calM_{c}^{(J)}$'s obtained in \textbf{Step i} are zero when $c$ is sufficiently small. This critical observation results from a couple of intermediate results (Lemma~\ref{lem:bias-centered-word-expansion} and Lemma~\ref{lem:bias-centered-word-bound}), which we prove in Appendix~\ref{app:bias ii}. As will be clear in the proof, these intermediate results are used to show that the terms in $\calM_{c}^{(J)}$ cancel each other meticulously when $c$ is below a certain threshold, $\sfc_{m} \coloneqq \lceil (m - 1) / 2 \rceil$. 

\item Finally, we bound all relevant terms from \textbf{Step~ii} by applying the graph-counting lemma (Lemma~\ref{lem:graph-counting}) introduced in Section~\ref{sec:graph}, culminating in Lemma~\ref{lem:bias-surviving-levels}, which completes bias bound.

\end{enumerate}

\paragraph{Step i.}

We first represent $\calB_{m, k}$ in a particular form as stated in the following lemma; see its proof at the beginning of Appendix~\ref{app:bias i}.

\begin{lemma}
\label{lem:bias-binomial-representation}
Recall the definition of $\sfD_{J}$ in \eqref{main Neumann}. $\calB_{m, k}$ admits the following alternative representations:
\begin{align}
\label{eq:bias-binomial-representation}
\calB_{m, k} & = \sum_{j = 1}^{m - 1} (-1)^{j + 1} \binom{m - 1}{j} \bbE \Big\{ A_{m - 1} X_{m - 1}^{\top} \Big( \prod_{s = 0}^{j - 1} X_{s} X_{s}^{\top} \hat{\Omega}_{J} - I \Big) X_{m} Y_{m} \Big\} \notag \\
& = \sum_{j = 1}^{m - 1} (-1)^{j + 1} \binom{m - 1}{j} \sum_{\emptyset \neq S \subseteq [j - 1] \cup \{0\}} \bbE \Big\{ A_{m - 1} X_{m - 1}^{\top} \prod_{s = 0}^{j - 1} \Big( X_{s} X_{s}^{\top} \sfD_{J}^{\mathbbm{1} \{s \in S\}} \Big) X_{m} Y_{m} \Big\}.
\end{align}
Here, we use the convention that $s = 0$ corresponds to the identity matrix $I$.
\end{lemma}
We next state a lemma that further decomposes $\calB_{m, k}$ into components that share the same multiplicity of $\Delta_{n} = I - \hat{\Sigma}$. The proof is delegated to Appendix~\ref{app:bias i}.

\begin{lemma}
\label{lem:bias-degree-decomposition}
For an integer $J \geq 1$,
\begin{equation}
\label{eq:bias-degree-decomposition}
\calB_{m, k} = \sum_{c = 1}^{(m - 1) J} \calM_{c}^{(J)} ,
\end{equation}
where $\calM_{c}^{(J)}$ is defined as:
\begin{align}
& \calM_{c}^{(J)} \coloneqq \label{eq:bias-degree-component} \\
& \sum_{j = 1}^{m - 1}(-1)^{j + 1} \binom{m - 1}{j} \sum_{r = 1}^{c \wedge j} \sum_{\substack{S \subseteq [j - 1] \cup \{0\} \\ |S| = r}}  \sum_{\substack{(\ell_{s'})_{s' \in S} \in [J]^{r} \\ \sum_{s' \in S} \ell_{s'} = c}} \bbE \Big\{ A_{m - 1} X_{m - 1}^{\top} \prod_{s = 0}^{j - 1} (X_{s} X_{s}^{\top} \Delta_{n}^{\ell_{s} \mathbbm{1} \{s \in S\}}) X_{m} Y_{m} \Big\}. \nonumber
\end{align}
Here, we use the convention that $\ell_{s} = 0$ for $s \notin S$.
\end{lemma}

Equivalently, $\calM_{c}^{(J)}$ sums up all terms for which the multiplicity of $\Delta_{n} = I - \hat{\Sigma}$ equals $\sum_{s = 0}^{j - 1} \ell_{s} = c$. When the truncation level $J$ is fixed, we write $\calM_{c}$ for $\calM_{c}^{(J)}$ to simplify the notation. This step reduces the analysis to each term $\calM_{c}^{(J)}$ for $c \in [(m - 1) J]$.

\paragraph{Step ii.}

We further refine $\calM_{c}^{(J)}$ for $c \in [(m - 1) J]$ by showing that $\calM_{c}^{(J)} = 0$ when $c$ is sufficiently small, or Lemma~\ref{lem:bias-low-degree-cancellation} below.

\begin{lemma}
\label{lem:bias-low-degree-cancellation}
Under the notation of Lemma~\ref{lem:bias-degree-decomposition}, let $ \sfc_{m} \coloneqq \Big \lceil \frac{m - 1}{2} \Big \rceil$. Then, for every integer $J \ge 1$,
\begin{equation}
\label{eq:bias-step-ii-cancellation}
\calM_{c}^{(J)} = 0, \quad 1 \leq c < \sfc_{m}.
\end{equation}
Equivalently, when $J$ is fixed and we write $\calM_{c}$ for $\calM_{c}^{(J)}$, one has $\calM_{c} = 0$ for all $1 \le c < \sfc_{m}$.
\end{lemma}

The proof of Lemma~\ref{lem:bias-low-degree-cancellation} is deferred to Appendix~\ref{app:bias ii}. As mentioned, showing that $\calM_{c}^{(J)}$ is exactly zero requires a careful calculation to show that all terms involved in $\calM_{c}^{(J)}$ cancel each other out. To achieve this, in Appendix~\ref{app:bias ii}, we first establish Lemma~\ref{lem:bias-centered-word-expansion} and~\ref{lem:bias-centered-word-bound}, based on which Lemma~\ref{lem:bias-low-degree-cancellation} can be proved.

\paragraph{Step iii.}

We now bound the non-zero $\calM_{c}^{(J)}$'s after \textbf{Step~ii}. Define
\begin{equation*}
s_{c} \coloneqq \Big\lceil \frac{c}{2} \Big\rceil \vee 1, \ \rho_{j} \coloneqq j \rho = \frac{j k}{n},\ \zeta_{A, Y}\coloneqq \|A\|_2 \|Y\|_2 + \|A\|_\infty \|Y\|_2 + \|A\|_2 \|Y\|_\infty.
\end{equation*}
First, Lemma~\ref{lem:bias-surviving-levels} below exhibits the order of $\calM_{c}^{(J)}$ when it is not identically zero.

\begin{lemma}
\label{lem:bias-surviving-levels}
Let $ \sfc_{m} \coloneqq \Big \lceil \frac{m - 1}{2} \Big \rceil, \ s_{c} \coloneqq \Big \lceil \frac{c}{2} \Big \rceil \vee 1 $.
Suppose that $ \frac{C m J k}{n} \leq \eta < 1 $.
Then, for every $\sfc_{m} \leq c \leq (m - 1)J$,
\begin{equation}
\label{eq:bias-surviving-level-bound}
| \calM_{c}^{(J)} | \lesssim
\zeta_{A, Y} \Big( \frac{C (m \vee c) k}{n} \Big)^{s_{c}}.
\end{equation}
\end{lemma}

We defer the proof of lemma~\ref{lem:bias-surviving-levels} to Appendix~\ref{app:bias iii}. In particular, the proof of the result relies on the graph-counting Lemma~\ref{lem:graph-counting} by associating $U$-statistic kernels derived from rewriting $\calB_{m, k}$ with undirected graphs. Bounding the mean of such a $U$-statistic kernel reduces to counting the first Betti number of the associated undirected graph. We then divide our analysis into two scenarios.
\begin{itemize}
\item For $\sfc_{m} \leq c \leq m$, we have $m \vee c = m,\ \frac{C (m \vee c) k}{n} = C \rho_{m}$. Therefore, $\sum_{c = \sfc_{m}}^{m} |\calM_{c}^{(J)}| \lesssim \zeta_{A, Y} \sum_{c = \sfc_{m}}^{m}
(C \rho_{m})^{s_{c}}$. Since $s_{c} = \lceil c / 2 \rceil \vee 1$, each exponent $s_{c}$ occurs at most twice in the summation. Thus, when $C \rho_{m} < 1$, $\sum_{c = \sfc_{m}}^{m} (C \rho_{m})^{s_{c}} \leq 2 \sum_{\ell = s_{\sfc_{m}}}^{s_{m}} (C \rho_{m})^{\ell} \lesssim (C \rho_{m})^{s_{\sfc_{m}}}$. Hence,
\begin{equation*}
\sum_{c = \sfc_{m}}^{m} |\calM_{c}^{(J)}| \lesssim \zeta_{A, Y} (C \rho_{m})^{s_{\sfc_{m}}}.
\end{equation*}

\item For $c > m$, set $g_{c} \coloneqq (C \rho_{c})^{c / 2}$. Since $C \rho_{c} < 1$ and $s_{c} \geq c / 2$, $|\calM_{c}^{(J)}| \lesssim \zeta_{A, Y} g_{c}$. There exists $C > 0$ such that $\sqrt{C e \rho m J} \leq q < 1$. Then $\frac{g_{c + 1}}{g_{c}} = \sqrt{C \rho_{c + 1}} \Big(1 + \frac{1}{c} \Big)^{c/2} \leq \sqrt{C e\rho_{c + 1}} \leq q$, so the sequence $\{g_{c}\}_{c > m}$ decreases to zero at a geometric rate. Hence,
\begin{equation*}
\sum_{c = m + 1}^{(m - 1) J} |\calM_{c}^{(J)}| \lesssim \zeta_{A, Y} g_{m} = \zeta_{A, Y} (C \rho_{m})^{m / 2} \leq \zeta_{A, Y} (C \rho_{m})^{s_{\sfc_{m}}}.
\end{equation*}
\end{itemize}
Integrating the above two scenarios leads to $\sum\limits_{c = \sfc_{m}}^{(m - 1) J} |\calM_{c}^{(J)}| \lesssim \zeta_{A, Y} \Big( \frac{C m k}{n} \Big)^{s_{\sfc_{m}}}$. Combining this with \eqref{eq:bias-degree-decomposition} and \eqref{eq:bias-step-ii-cancellation} yields $|\calB_{m, k}| \lesssim \zeta_{A, Y} \Big( \frac{C m k}{n} \Big)^{s_{\sfc_{m}}}$ with $s_{\sfc_{m}} = \Big\lceil \frac{\sfc_{m}}{2} \Big\rceil = \Big\lceil \frac{m - 1}{4} \Big\rceil$, the desired bias bound.

\subsection{Variance analysis}
\label{sec:variance}

The variance part is more involved than that of $\hat{\psi}_{m, k} (\hat{\Omega}_{\nuis})$ in \citet{liu2017semiparametric}, because we cannot use the Hoeffding decomposition. We divide the variance analysis into the following six steps:
\begin{enumerate}[label = \roman*.]
\item We first apply Minkowski's inequality to reduce the task of bounding the variance of $ \chi_{m, k} (\Omega)$ to the task of bounding the variance of each $\hat{\IIFF}_{j, j, k}^{(J)} (\Omega)$ for $j = 2, \cdots, m$; see Lemma~\ref{lem:Minkowski}. 

\item We then invoke Lemma~\ref{lem:finite-neumann-mobius} (an extension of the M\"obius-inversion expansion in Lemma~\ref{lem:mobius}) to rewrite each $\hat{\IIFF}_{j, j, k}^{(J)} (\Omega)$ as a finite sum of lower-order $U$-statistics.

\item Starting from the lower-order $U$-statistics obtained in \textbf{Step~ii}, we further expand each $U$-statistic to a finite linear combination of a statistic $T_{a, q, \gamma, \bm{\ell}}$ (see \eqref{before chain} for its form) more amenable for analysis; see Lemma~\ref{lem:mobius-block-to-chain}.

\item We decompose the variance of each $T_{a, q, \gamma, \bm{\ell}}$ into a sum of covariances between pairs of kernels that share different numbers of sample indices; see Lemma~\ref{lem:chain-covariance-decomposition} in Appendix~\ref{app:var iv}.

\item Expanding the powers of $\Delta_{n}$ in the kernels obtained in \text{Step~iii} then yields kernels involving only products of inner products $X_{i}^{\top} X_{j}$ for $i, j \in [n]$ (abbreviated as \emph{multiplicative kernels}). We then associate each multiplicative kernel with an undirected graph, whose vertices correspond to all sample indices involved in the inner products $X_{i}^{\top} X_{j}$ and whose edges describe whether a pair of indices $i, j$ is present in any of these inner products. We prove a generic variance bound for $U$-statistics with multiplicative kernels by combining the decomposition in \textbf{Step~iv} with a counting argument based on the first Betti number of the graph associated with the kernel, as stated previously in Lemma~\ref{lem:graph-counting}.

\item Finally, we combine the expansion based on M\"obius inversion in Lemma~\ref{lem:finite-neumann-mobius} in \textbf{Step~ii} and the results in \textbf{Step~v} to obtain the desired variance bound for $\hat{\IIFF}_{j, j, k}^{(J)} (\Omega)$ for each $j = 2, 3, \cdots, m$.
\end{enumerate}
The flow of the proof argument is summarized in Figure~\ref{fig:variance-proof-roadmap}. We now move onto the proof sketch.

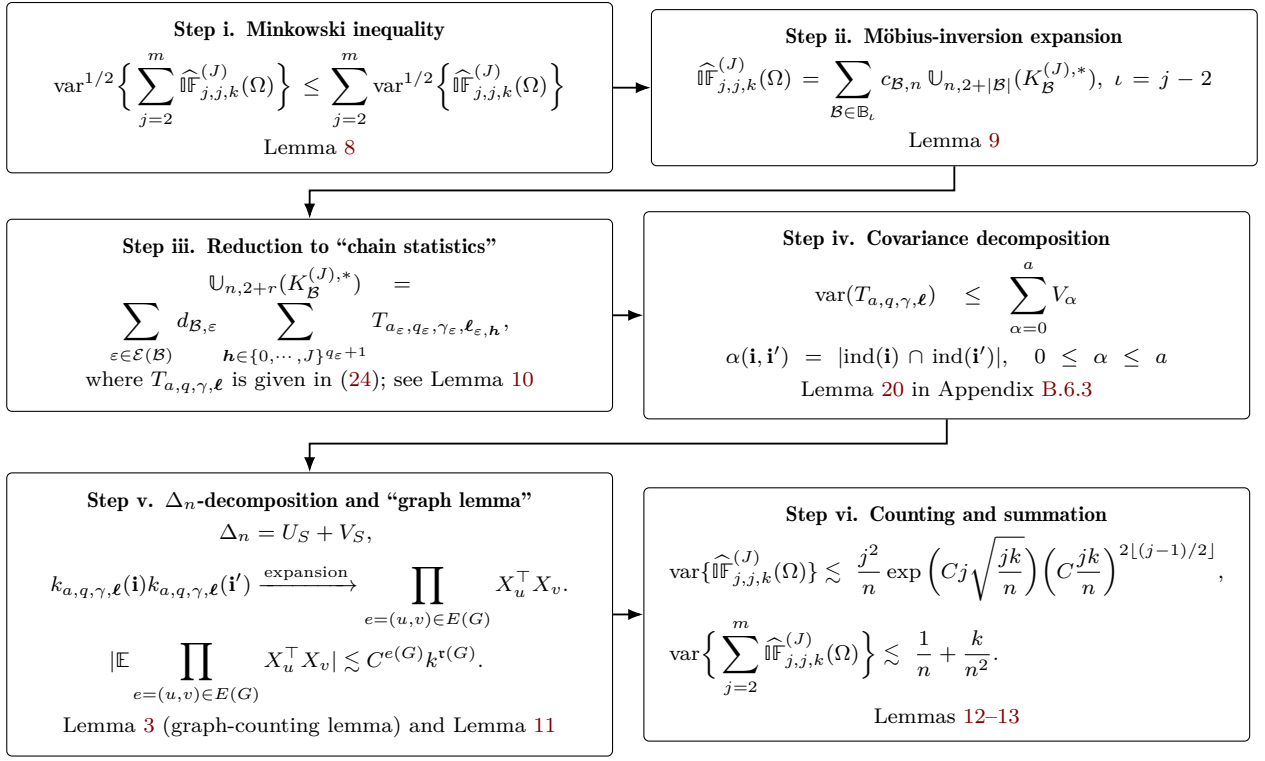
\begin{figure}[!t]
\centering
\begin{tikzpicture}[
    >={Latex[length=2mm]},
    roadbox/.style={
        draw=black,
        rounded corners=2pt,
        align=center,
        text width=0.455\textwidth,
        inner xsep=7pt,
        inner ysep=7pt,
        font=\scriptsize
    },
    arr/.style={
        ->,
        line width=0.7pt
    },
    downmark/.style={
        font=\large
    }
]

\node[roadbox] (s1) {
\textbf{Step i. Minkowski inequality}\\[1mm]
$\displaystyle
\var^{1 / 2} \Big\{ \sum_{j = 2}^{m}\hat{\IIFF}_{j, j, k}^{(J)} (\Omega) \Big\} \le \sum_{j = 2}^{m}
\var^{1 / 2} \Big\{ \hat{\IIFF}_{j, j, k}^{(J)} (\Omega) \Big\}
$\\[1mm]
Lemma~\ref{lem:Minkowski}
};

\node[roadbox, right=0.50cm of s1] (s2) {
\textbf{Step ii. M\"obius-inversion expansion}\\[1mm]
$\displaystyle
\hat{\IIFF}_{j, j, k}^{(J)} (\Omega)
= \sum_{\calB \in \bbB_{\iota}}
c_{\calB, n}\, \bbU_{n, 2 + |\calB|} (K_{\calB}^{(J), *}),
\ \iota = j-2
$\\[1mm]
Lemma~\ref{lem:finite-neumann-mobius}
};

\node[roadbox, below=0.60cm of s1] (s3) {
\textbf{Step iii. Reduction to ``chain statistics''}\\[1mm]
$\displaystyle
\bbU_{n, 2 + r} (K_{\calB}^{(J), *})
= \sum_{\varepsilon \in \calE (\calB)}
d_{\calB, \varepsilon} \sum_{\bm{h} \in
\{0, \cdots, J\}^{q_{\varepsilon} + 1}}
T_{a_{\varepsilon}, q_{\varepsilon}, \gamma_{\varepsilon},
\bm{\ell}_{\varepsilon, \bm{h}}},
$ \\
where $T_{a, q, \gamma, \bm{\ell}}$ is given in \eqref{before chain}; see Lemma~\ref{lem:mobius-block-to-chain}
};

\node[roadbox, right=0.40cm of s3] (s4) {
\textbf{Step iv. Covariance decomposition}\\[1mm]
$\displaystyle
\var (T_{a, q, \gamma, \bm{\ell}}) \le \sum_{\alpha = 0}^{a} V_{\alpha}
$\\[1mm]
$\displaystyle
\alpha (\bfi, \bfi') = | \ind (\bfi) \cap \ind (\bfi') |, \quad 0 \leq \alpha \leq a
$\\[1mm]
Lemma~\ref{lem:chain-covariance-decomposition} in Appendix~\ref{app:var iv}
};

\node[roadbox, below=0.80cm of s3] (s5) {
\textbf{Step v. $\Delta_n$-decomposition and ``graph lemma''}\\[1mm]
$\begin{aligned}
\Delta_{n}  & = U_{S} + V_{S}, \\[1mm]
k_{a, q, \gamma, \bm{\ell}} (\bfi) k_{a, q, \gamma, \bm{\ell}} (\bfi')
& \xrightarrow{\text{expansion}} \prod_{e = (u, v) \in E(G)} X_{u}^{\top} X_{v} .
\end{aligned}$\\[1mm]
$\begin{aligned}
| \bbE \prod_{e = (u, v) \in E(G)}
X_{u}^{\top} X_{v} | &\lesssim C^{e(G)} k^{\frakr (G)}.
\end{aligned}$\\[1mm]
Lemma~\ref{lem:graph-counting} (graph-counting lemma) and Lemma~\ref{lemma:generic-chain-variance-explicit} 
};

\node[roadbox, right=0.40cm of s5] (s6) {
\textbf{Step vi. Counting and summation}\\[1mm]
$\displaystyle
\begin{aligned}
&\var \{\hat{\IIFF}_{j, j, k}^{(J)} (\Omega)\}
\lesssim  \ \frac{j^{2}}{n} \exp \Big( C j \sqrt{\frac{j k}{n}} \Big) \Big(C \frac{j k}{n}\Big)^{2\lfloor (j-1) / 2 \rfloor},\\[1mm]
&\var \Big\{\sum_{j = 2}^{m}\hat{\IIFF}_{j, j, k}^{(J)} (\Omega)\Big\} \lesssim  \ \frac{1}{n} + \frac{k}{n^{2}}.
\end{aligned}
$\\[1mm]
Lemmas~\ref{lem:mobius-r-variance}--\ref{lem:fixed-j-ifhat-variance}
};

\draw[arr] (s1.east) -- (s2.west);
\draw[arr] (s3.east) -- (s4.west);
\draw[arr] (s5.east) -- (s6.west);

\draw[arr] (s2.south) -- ++(0,-0.32) -| (s3.north);
\draw[arr] (s4.south) -- ++(0,-0.32) -| (s5.north);


\end{tikzpicture}

\caption{Schematic overview of the variance analysis. The diagram displays the main algebraic reductions, the covariance decomposition for the generic
multiplicative kernel $U$-statistic, the centered-factor expansion followed by graph-counting bounds, and the final summation
over collection levels and correction orders.}
\label{fig:variance-proof-roadmap}
\end{figure}

\paragraph{Step i.}

We have the following result, which is a direct consequence of Minkowski's inequality.

\begin{lemma}
\label{lem:Minkowski}
The following inequality holds.
\begin{align}
\label{Minkowski}
\var^{1 / 2} \{\chi_{m, k} (\Omega)\} = \var^{1 / 2} \Big\{ \sum_{j = 2}^{m} \hat{\IIFF}^{(J)}_{j, j, k} (\Omega) \Big\} \le \sum_{j = 2}^{m} \var^{1 / 2} \{\hat{\IIFF}^{(J)}_{j, j, k} (\Omega)\}.
\end{align}
\end{lemma}

By Lemma~\ref{lem:Minkowski}, the variance analysis of $ \chi_{m, k}(\Omega)$ reduces to bounding each fixed-order term $\hat{\IIFF}^{(J)}_{j, j, k} (\Omega)$, while keeping track of the dependence on $j$ (and eventually on $m$), $k$, and $n$.

\paragraph{Step ii.}

We next extend the M\"obius-inversion expansion in Lemma~\ref{lem:mobius} to 
$\hat{\IIFF}_{j , j , k}^{(J)} (\Omega)$ appearing in $\chi_{m , k} (\Omega)$. Fix $3 \leq j \leq m$, let $\iota \coloneqq j - 2$, and define the following quantities:
\begin{equation*}
E_{J} \coloneqq \Delta_n^{J + 1} ,
\quad H_{i}^{(J)} \coloneqq X_{i} X_{i}^{\top} \hat{\Omega}_{J} - I, \quad \calR_{i_{1} i_{2}}^{(J)} \coloneqq H_{i_{1}}^{(J)} + H_{i_{2}}^{(J)}.
\end{equation*}
The geometric identity gives $\hat{\Sigma} \hat{\Omega}_{J} = (I - \Delta_{n}) \sum_{h = 0}^{J} \Delta_{n}^{h} = I - E_{J}$, implying the following identities
\begin{equation}
\label{eq:finite-neumann-centering-main}
\sum_{i = 1}^{n} H_{i}^{(J)} = n (\hat{\Sigma} \hat{\Omega}_{J} - I) = - n E_{J}, \quad 
\sum_{i \notin \{i_{1} , i_{2} \}} H_{i}^{(J)} = - (\calR_{i_{1} i_{2}}^{(J)} + n E_{J}).
\end{equation}

When $\hat{\Omega}$ is replaced by $\hat{\Omega}_{J}$, the first identity in \eqref{eq:finite-neumann-centering-main} replaces \eqref{empirical centering} in Lemma~\ref{lem:empirical centering}. We then directly plug-in the M\"obius-inversion expansion of Lemma~\ref{lem:mobius} by replacing relevant terms in that expansion with the quantities defined above verbatim, obtaining a similar lemma (Lemma~\ref{lem:finite-neumann-mobius}). Recall that $\bbB_{\iota}$ consists of all finite collections $\calB = \{B_{1}, \cdots, B_{r}\}$ of pairwise disjoint subsets of $[\iota]$ such that $|B_{\nu}| \ge 2$ for every $\nu \in [r]$. Specifically, for $\calB \in \bbB_{\iota}$, let $\calU_{\calB} \coloneqq [\iota] \setminus \bigcup_{B \in \calB} B$, and let $K_{\calB}^{(J), *}$ be obtained from the kernel $K_{\calB}$ in \eqref{kernel} by the substitutions: $\hat{\Omega} \mapsto \hat{\Omega}_{J}, H_{i} \mapsto H_{i}^{(J)}, \calR_{i_{1} i_{2}} \mapsto \calR_{i_{1} i_{2}}^{(J)} + n E_{J}$, while preserving the displayed order of all matrix products. For $S \subseteq \calU_{\calB}$, let $K_{\calB, S}^{(J)}$ be obtained from $K_{\calB}$ by $\hat{\Omega} \mapsto \hat{\Omega}_{J}$ and $H_{i} \mapsto H_{i}^{(J)}$, and replacing the occurrence of $\calR_{i_{1} i_{2}}$ at each position $\ell \in \calU_{\calB}$ by $n E_{J}$ if $\ell \in S$, and by $\calR_{i_{1} i_{2}}^{(J)}$ if otherwise. We then have the identity $K_{\calB}^{(J), *} = \sum_{S \subseteq \calU_{\calB}} K_{\calB, S}^{(J)}$.
With $c_{\calB, n}$ denoting the M\"obius coefficient in Lemma~\ref{lem:mobius}, define
\begin{align*}
M_{j, J} \coloneqq \sum_{\calB \in \bbB_{\iota}} c_{\calB, n} \bbU_{n, 2 + |\calB|} (K_{\calB, \emptyset}^{(J)}) , \quad L_{j, J}  = \sum_{\calB \in \bbB_{\iota}} c_{\calB, n} \sum_{\emptyset \ne S \subseteq \calU_{\calB}} \bbU_{n, 2 + |\calB|} (K_{\calB, S}^{(J)}).
\end{align*}
The following lemma gives the resulting expansion.
\begin{lemma}
\label{lem:finite-neumann-mobius}
For $3 \leq j \leq m$,
\begin{align}
\label{eq:finite-neumann-mobius-main}
\hat{\IIFF}_{j, j, k}^{(J)} (\Omega)
=\sum_{\calB \in \bbB_{\iota}} c_{\calB, n}
\bbU_{n, 2 + |\calB|} (K_{\calB}^{(J), *})
= M_{j, J} + L_{j, J} .
\end{align}
\end{lemma}
Indeed, by the identity $K_{\calB}^{(J), *} = \sum_{S \subseteq \calU_{\calB}} K_{\calB, S}^{(J)}$ and the linearity of $\bbU_{n, 2 + |\calB|}$ in its kernel, for each $\calB$, $\bbU_{n, 2 + |\calB|} (K_{\calB}^{(J), *})$ decomposes into the term indexed by $S = \emptyset$ and the sum of the terms indexed by $\emptyset \ne S \subseteq \calU_{\calB}$. Thus, $M_{j, J}$ collects all terms that contain no occurrence of $n E_{J}$, whereas $L_{j, J}$ collects all terms that contain at least one occurrence of $n E_{J}$; see Appendix~\ref{app:var ii} for a proof.

\paragraph{Step iii.}
We next rewrite each summand in the decomposition of Lemma~\ref{lem:finite-neumann-mobius} as a finite linear combination of ``chain statistics'', in the sense that their kernels consist of ordered products of
rank-one matrices $X X^{\top}$ and powers of $\Delta_{n}$
connecting the endpoint factors $A_{i_{1}} X_{i_{1}}^{\top}$ and $X_{i_{2}} Y_{i_{2}}$.
Throughout the variance analysis, we denote
\begin{equation*}
\calI_{n, a} \coloneqq
\{(i_{1}, \cdots, i_{a}) \in [n]^{a}:
i_{s} \ne i_{t} \text{ for } s \ne t\}
\end{equation*}
as the set of ordered tuples of pairwise distinct sample indices. For $2 \leq a \leq n$, $q \geq 0$, $\bm{\ell} = (\ell_{0}, \cdots, \ell_{q}) \in \bbN_{0}^{q + 1}$, and $\gamma : \{1, \cdots, q\} \to [a]$ satisfying $\{1, 2, \gamma (1), \cdots, \gamma (q)\} = [a]$, define the kernel $$k_{a, q, \gamma, \bm{\ell}} (O_{\bfi}) \coloneqq A_{i_{1}} X_{i_{1}}^{\top} \Delta_{n}^{\ell_{0}} \Big( \prod_{u = 1}^{q} X_{i_{\gamma (u)}} X_{i_{\gamma (u)}}^{\top} \Delta_{n}^{\ell_{u}} \Big) X_{i_{2}} Y_{i_{2}}$$ for $\bfi \in \calI_{n, a}$ and in turn define
\begin{equation}
\label{before chain}
T_{a, q, \gamma, \bm{\ell}} \coloneqq
\frac{(n - a) !}{n !} \sum_{\bfi \in \calI_{n, a}}
k_{a, q, \gamma, \bm{\ell}} (O_{\bfi}),
\quad d (\bm{\ell}) \coloneqq \sum_{u = 0}^{q} \ell_{u}. 
\end{equation}
Here $\gamma : \{1, \cdots, q\} \to [a]$, and the condition $\{1, 2, \gamma (1), \cdots, \gamma (q)\} = [a]$ means that every entry of $\bfi = (i_{1}, \cdots, i_{a})$ appears in the kernel either as one of the endpoint indices $i_{1}, i_{2}$ or through some $i_{\gamma (u)}$ associated with the $u$-th rank-one matrix $X_{i_{\gamma (u)}} X_{i_{\gamma (u)}}^{\top}$ in the product.

We next introduce the notation used in the decomposition. Fix $3 \leq j \leq m$, set $\iota = j - 2$, and let
$\calB = \{B_{1}, \cdots, B_{r}\} \in \bbB_{\iota}$. For each $l \in [\iota]$ ($\iota = j - 2$), define
\begin{equation*}
\calE_{l} (\calB) \coloneqq
\begin{cases}
\{0, \nu\}, & l \in B_{\nu} \text{ for some } \nu \in \{1, \cdots, r\},\\
\{0, 1, 2, \star\}, & l \notin \bigcup_{\nu = 1}^{r} B_{\nu}.
\end{cases}
\end{equation*}
Set $\calE (\calB) \coloneqq
\prod_{l = 1}^{\iota} \calE_{l} (\calB)$. For $\varepsilon = (\varepsilon_{1}, \cdots, \varepsilon_{\iota}) \in \calE (\calB)$, define $d_{\calB, \varepsilon} \coloneqq (-1)^{N_{H} (\varepsilon)} (-2)^{N_{R} (\varepsilon)} n^{s_{\varepsilon}}$, where $N_{H} (\varepsilon) \coloneqq |\{l \in \bigcup_{\nu = 1}^{r} B_{\nu}: \varepsilon_{l} = 0\}|$, $N_{R} (\varepsilon) \coloneqq |\{l \notin \bigcup_{\nu = 1}^{r} B_{\nu}: \varepsilon_{l} = 0\}|$, and $s_{\varepsilon} \coloneqq |\{l \notin \bigcup_{\nu = 1}^{r} B_{\nu}: \varepsilon_{l} = \star\}|$. Let $\calA_{\varepsilon} \coloneqq \{\nu \in \{1, \cdots, r\}: \varepsilon_{l} = \nu \text{ for at least one } l \in B_{\nu}\}$. Set $b_{\varepsilon} \coloneqq |\calA_{\varepsilon}|$, $a_{\varepsilon} \coloneqq 2 + b_{\varepsilon}$, and write $ \calA_{\varepsilon} = \{\nu_{1}, \cdots, \nu_{b_{\varepsilon}}\}, \ \nu_{1} < \cdots < \nu_{b_{\varepsilon}}$. 
Let $q_{\varepsilon} \coloneqq |\{l \in [\iota] : \varepsilon_{l} \notin \{0, \star\}\}|$ and write these positions as $l_{1} < \cdots < l_{q_{\varepsilon}}$. Define $\gamma_{\varepsilon} : [q_{\varepsilon}] \to [a_{\varepsilon}]$
as follows. For $h \in [q_{\varepsilon}]$, set
$\gamma_{\varepsilon} (h) = 2 + t$ if
$l_{h} \in B_{\nu_{t}}$ for some $t \in [b_{\varepsilon}]$, and set
$\gamma_{\varepsilon} (h) = \varepsilon_{l_{h}}$ otherwise.
For $\bm{h} = (h_{0}, \cdots, h_{q_{\varepsilon}})
\in \{0, \cdots, J\}^{q_{\varepsilon} + 1}$, set $l_{0} \coloneqq 0$ and
$l_{q_{\varepsilon} + 1} \coloneqq \iota + 1$.
For each $0 \leq u \leq q_{\varepsilon}$, let
$s_{\varepsilon, u}$ denote the number of positions
$v \in [\iota]$ satisfying
$l_{u} < v < l_{u + 1}$ and $\varepsilon_{v} = \star$, and define
$\ell_{\varepsilon, \bm{h}, u}
\coloneqq h_{u} + (J + 1) s_{\varepsilon, u}$.
Write $\bm{\ell}_{\varepsilon,\bm{h}}
\coloneqq ( \ell_{\varepsilon, \bm{h}, 0},
\cdots, \ell_{\varepsilon, \bm{h}, q_{\varepsilon}})$.
We then have the following lemma. See Appendix~\ref{app:var iii} for a proof.

\begin{lemma}
\label{lem:mobius-block-to-chain}
Each summand in the decomposition of $\hat{\IIFF}^{(J)}_{j, j, k} (\Omega)$ in Lemma~\ref{lem:finite-neumann-mobius} satisfies:
\begin{equation}
\label{lemma 10}
\bbU_{n, 2 + r} (K_{\calB}^{(J), *}) = \sum_{\varepsilon \in \calE(\calB)} d_{\calB, \varepsilon} \sum_{\bm{h} \in
\{0, \cdots, J\}^{q_{\varepsilon} + 1}}
T_{a_{\varepsilon}, q_{\varepsilon}, \gamma_{\varepsilon},
\bm{\ell}_{\varepsilon, \bm{h}}}.
\end{equation}
Moreover, for every $\varepsilon \in \calE (\calB)$, $q_{\varepsilon} - b_{\varepsilon} \le \iota - r - s_{\varepsilon}$. Consequently,
\begin{align*}
M_{j, J} & = \sum_{\calB \in \bbB_{\iota}} c_{\calB, n}
\sum_{\substack{\varepsilon \in \calE (\calB)\\
s_{\varepsilon} = 0}} d_{\calB, \varepsilon}
\sum_{\bm{h} \in \{0, \cdots, J\}^{q_{\varepsilon} + 1}}
T_{a_{\varepsilon}, q_{\varepsilon}, \gamma_{\varepsilon}, \bm{\ell}_{\varepsilon, \bm{h}}},\\
L_{j, J} & = \sum_{\calB \in \bbB_{\iota}} c_{\calB, n}
\sum_{\substack{\varepsilon \in \calE (\calB)\\
s_{\varepsilon} \geq 1}} d_{\calB, \varepsilon}
\sum_{\bm{h} \in \{0, \cdots, J\}^{q_{\varepsilon} + 1}}
T_{a_{\varepsilon}, q_{\varepsilon}, \gamma_{\varepsilon}, \bm{\ell}_{\varepsilon, \bm{h}}}.
\end{align*}
\end{lemma}

Thus each summand $\bbU_{n, 2 + |\calB|} (K_{\calB}^{(J), *})$ is a finite linear combination of $U$-statistics $T_{a, q, \gamma, \bm{\ell}}$. It remains to control the variance of $T_{a, q, \gamma, \bm{\ell}}$ uniformly in $(a, q, \gamma, \bm{\ell})$.

\paragraph{Step iv.}

We write $\operatorname{ind} (\bfi) \coloneqq \{i_{1}, \cdots, i_{a}\}$ as the unordered set of sample indices appearing in the tuple $\bfi$. To control $\var (T_{a, q, \gamma, \bm{\ell}})$, we analyze the covariance between its kernels indexed by ordered tuples $\bfi = (i_{1}, \cdots, i_{a})$ and $\bfi' = (i'_{1}, \cdots, i'_{a})$. We group the covariances by the number of shared sample indices: $\alpha (\bfi, \bfi') \coloneqq |\operatorname{ind} (\bfi) \cap \operatorname{ind} (\bfi')|$. Thus, $\var (T_{a, q, \gamma, \bm{\ell}})$ can be decomposed into a sum of covariances indexed by $\alpha = 0, 1, \cdots, a$. More precisely, Lemma~\ref{lem:chain-covariance-decomposition} in Appendix~\ref{app:var iv} bounds $\var (T_{a, q, \gamma, \bm{\ell}})$ as follows:
\begin{equation}
\label{var bound decomposition}
\var (T_{a, q, \gamma, \bm{\ell}}) \le \sum_{\alpha = 0}^{a} V_{\alpha}.
\end{equation}
For $\alpha \geq 1$, the term $V_{\alpha}$ collects the covariances between summands whose displayed index tuples share exactly $\alpha$ sample indices, whereas $V_{0}$ collects the zero-overlap terms. Here $\alpha$ records only the overlap between the displayed indices; the two summands may remain dependent through $\Delta_{n}$.
\paragraph{Step v.}

We next bound the covariance contributions $V_{\alpha}$,
$\alpha = 0, 1, \cdots, a$, by the graph-counting Lemma~\ref{lem:graph-counting}. To this end, we first record the following lemma, which is proved in Appendix~\ref{app:var v}.

\begin{lemma}
\label{lemma:generic-chain-variance-explicit}
Let $T_{a, q, \gamma, \bm{\ell}}$ be defined as in \eqref{before chain}, with $d \coloneqq d (\bm{\ell})$. The integers $a \geq 2$ and $q \geq 0$ may depend on $n$. Assume that $n \geq 2 a$ and $\frac{C (a \vee d) k}{n} \leq \eta < 1$. Then we have
\begin{align*}
V_{0} \lesssim \frac{C^{2q} a^{2}}{n} k^{2 q - 2 a + 4} \Big\{ \frac{C (a \vee d) k}{n} \Big\}^{(d - 1)} \text{ and } \sum_{\alpha = 1}^{a} V_{\alpha}
\lesssim \frac{C^{2 q} a^{2} }{n}k^{2 q - 2 a + 4} \Big\{\frac{C (a \vee d) k}{n} \Big\}^{d}
\exp \Big( \frac{C a^{2} k}{n} \Big).
\end{align*}
As a consequence of \eqref{var bound decomposition}, $\var (T_{a, q, \gamma, \bm{\ell}})
\lesssim \dfrac{C^{2 q} a^{2}}{n}
k^{2 q - 2 a + 4} \exp\Big(\dfrac{C a^{2} k}{n}\Big)\Big\{ \dfrac{C (a \vee d) k}{n} \Big\}^{\max\{ d - 1 , 0\}}$.
\end{lemma}

In the proof, we separately analyze the cases $\alpha \geq 1$ and $\alpha = 0$. For the former, we essentially directly apply the graph counting lemma Lemma~\ref{lem:graph-counting}. For the latter, we need a more careful analysis to leverage certain cancellation structures similar to the bias analysis; otherwise, applying Lemma~\ref{lem:graph-counting} directly will yield an extra $k$ in the bound.

\paragraph{Step vi.}
We now combine Lemma~\ref{lemma:generic-chain-variance-explicit} with the \Mobius{}-inversion expansion in Lemma~\ref{lem:finite-neumann-mobius}. Let $r_{\iota}^{*} \coloneqq \lfloor \frac{\iota}{2} \rfloor$. 
For $0 \le r \le r_{\iota}^{*}$, let $ \bbB_{\iota,r} \coloneqq \{\calB \in \bbB_{\iota}: |\calB| = r\}$. Equivalently, $r$ is the number of sets in the collection $\calB$. Then we have $M_{j, J} = \sum_{r = 0}^{r_{\iota}^{*}} Z_{\iota, r}$, where $Z_{\iota, r} \coloneqq \sum_{\calB \in \bbB_{\iota, r}} c_{\calB, n} \bbU_{n, 2 + r} (K_{\calB, \emptyset}^{(J)})$. The maximal possible value of $r$ is $r_{\iota}^{*}$ because every set $B_{\nu}$ in $\calB$ has cardinality of at least two. The next lemma first controls the contribution from a given $r$.

\begin{lemma}
\label{lem:mobius-r-variance}
Let $j \ge 3$ and $0 \le r \le r_{\iota}^{*}$. When $n \ge 2 (\iota + 2)$ and $\frac{C j J k}{n} \leq \eta < 1$,
\begin{equation*}
\var (Z_{\iota, r}) \lesssim 2^{2 (\iota - r)} w_{\iota, r}^{2} \exp \Big( C j \sqrt{\frac{j k}{n}} \Big) \frac{ j^{2} k^{2 (\iota - r)}}{n^{2 (\iota - r) + 1}},
\end{equation*}
where, for $\calB \in \bbB_{\iota, r}$, define $D (\calB) \coloneqq \sum_{\nu = 1}^{r} |B_{\nu}|$, with the convention $D (\emptyset) = 0$, and
\begin{equation*}
w_{\iota, r} \coloneqq \sum_{\calB \in \bbB_{\iota, r}} \Big\{ \prod_{B \in \calB}(|B| - 1) \Big\} 2^{D (\calB)} 4^{\iota - D (\calB)}.
\end{equation*}
\end{lemma}

In Lemma~\ref{lem:mobius-r-variance}, $w_{\iota, r}$ bounds the absolute sum of coefficients in the expansion of all level-$r$ terms $c_{\calB ,n} \bbU_{n, 2 + r} (K_{\calB})$, $\calB \in \bbB_{\iota, r}$, into multiplicative kernel $U$-statistics, up to the common factor $(n - 2 - \iota) ! / (n - 2 - r) !$; see Lemma~\ref{lem:weighted-block-counting} in Appendix~\ref{app:var vi}. Moreover, Lemma~\ref{lem:finite-neumann-mobius} gives $\hat{\IIFF}_{j, j, k}^{(J)} (\Omega) = \sum_{r = 0}^{r_{\iota}^{*}} Z_{\iota, r} + L_{j, J}$. The terms in $L_{j, J}$ arise from at least one appearance of $n E_{J}$ and are controlled in Appendix~\ref{app:var vi}. Aided by these results, we can finally give a bound on $\var \{\hat{\IIFF}_{j, j, k}^{(J)} (\Omega)\}$.
\begin{lemma}
\label{lem:fixed-j-ifhat-variance}
Let $j \geq 3$. Suppose that $n \geq 2 j$, $C J j k / n \leq \eta < 1$ and $J \geq C_{2} j$, where $C_{2}$ is a sufficiently large constant depending only on $\eta$. Then
\begin{equation*}
\var \{\hat{\IIFF}_{j, j, k}^{(J)} (\Omega)\} \lesssim \frac{j^{2}}{n} \exp \Big( C j \sqrt{\frac{j k}{n}} \Big) \Big( C \frac{j k}{n} \Big)^{2 \lfloor (j - 1) / 2\rfloor}.
\end{equation*}
\end{lemma}
The proofs of Lemmas~\ref{lem:mobius-r-variance} and \ref{lem:fixed-j-ifhat-variance} are deferred to Appendix~\ref{app:var vi}.
The term $j = 2$ is controlled by
Lemma~\ref{lem:second-order-correction-variance} in Appendix~\ref{app:var vi}, which yields $\var^{1 / 2}\{\hat{\IIFF}_{2, 2, k}^{(J)} (\Omega)\} \lesssim \frac{1}{\sqrt n} \Big(1 + \frac{k}{n}\Big)^{1 / 2}$. Therefore,
\begin{align*}
\var^{1 / 2} \Big\{ \sum_{j = 2}^{m} \hat{\IIFF}_{j, j, k}^{(J)} (\Omega) \Big\} \lesssim \frac{1}{\sqrt n} \Big( 1 + \frac{k}{n} \Big)^{1 / 2} + \frac{1}{\sqrt n} \sum_{j = 3}^{m} j \exp \Big( C j \sqrt{\frac{j k}{n}} \Big) \Big( C \frac{j k}{n} \Big)^{\lfloor (j - 1) / 2 \rfloor}.
\end{align*}
Recall that $\rho = k / n$. Under $C m \rho < 1$, pairing adjacent orders $j = 2 \ell + 1$ and $j = 2 \ell + 2$ yields
\begin{align*}
\sum_{j = 3}^{m} j \exp (C j \sqrt{j \rho}) (C j \rho)^{\lfloor (j - 1) / 2 \rfloor} & \lesssim \exp (C m \sqrt{m \rho}) \sum_{\ell = 1}^{\lfloor (m - 1) / 2 \rfloor} \ell (C \ell \rho)^{\ell}  \lesssim \rho \exp (C m \sqrt{m \rho}).
\end{align*}
Consequently, $\var^{1 / 2} \Big\{ \sum_{j = 2}^{m} \hat{\IIFF}_{j, j, k}^{(J)} (\Omega) \Big\} \lesssim \frac{1}{\sqrt{n}} \Big\{ (1 + \rho)^{1 / 2} + \rho \exp \Big( C m \sqrt{\frac{m k}{n}} \Big) \Big\}$, which immediately implies that $$\var \{\chi_{m , k} (\Omega)\} \lesssim \frac{1}{n} + \frac{k}{n^{2}}.$$ The proof of the variance bound is now complete.

\section{Concluding Remarks}
\label{sec:conclusions}
    
We conclude our article by mentioning several future research directions.
\begin{enumerate}[label = (\arabic*)]
\item It will be interesting to study if one can extend the idea developed in this article to the case $k \gtrsim n$ by, for instance, estimating $\Omega$ via shrinkage or regularized methods. As conjectured in \citet{robins2016technical}, the optimal convergence rate of the functionals studied in this article may depend on the regularity of the density of $X$. It is then reasonable to conjecture that the shrinkage or regularization also depends on the density of $X$. Simulation studies in \citet{liu2020nearly} suggest the nonlinear shrinkage covariance matrix estimators \citep{ledoit2020analytical} could be a viable option. It will also be interesting to investigate the statistical theoretical guarantees when $\Omega$ is estimated by the inverse of the ridge penalized estimator in the proportional asymptotic regime ($k \asymp n$) \citep{chen2024method}.

\item We expect to see the analysis strategy developed here to be further generalized to more complex problems, such as assumption-lean estimands \citep{vansteelandt2022assumption, vansteelandt2025towards}, functionals of NPIV models \citep{breunig2024adaptive}, functionals beyond bilinear forms \citep{lin2024worthwhile, zhang2026higher}, moment-condition models \citetext{\citealp{bonhomme2026higher}; \citealp[Section~6]{robins2016technical}}, multi-index models \citep{damian2025generative}, and other related problems \citep{wein2019kikuchi, liu2025quantum}.

\item Finally, in an ongoing work, we further compare, both in theory and practice, the stabilized HOIF estimator studied here and other related proposals such as sample-splitting- or sampling-based bias correction methods \citep{newey2018cross, cattaneo2018kernel, koltchinskii2022bootstrap, cattaneo2019two, koltchinskii2025estimation}.
\end{enumerate}

\section*{Acknowledgments}

LL thanks programme \emph{Causal inference: From theory to practice and back again} held at the Isaac Newton Institute (INI) of Mathematical Sciences in Cambridge UK, the School of Mathematics and Statistics at the University College Dublin, and the Center of Data Science at Zhejiang University for hospitality during the completion of this work. The authors thank Rohit Bhattacharya, Kwun Chuen Gary Chan, Fengnan Gao, Zhenyu Liao, Rajarshi Mukherjee, Jamie Robins, Andrea Rotnitzky, Eric Tchetgen Tchetgen, Aad van der Vaart, Cheng Wang, and participants in the \href{https://www.newton.ac.uk/event/cifw04/}{Causality and Machine Learning Workshop} held at INI for helpful discussions. A public talk related to this paper can be found in \href{https://www.newton.ac.uk/seminar/50892/}{https://www.newton.ac.uk/seminar/50892/}. This research is supported by the National Key R\&D Program of China Project Number 2025YFA1016700, NSFC Grant No.12471274, and Science and Technology Talent and Platform Program of Yunnan Province Grant No.202605AF35007. The authors used large language models solely to assist with computational verification, code review, and language editing. All conceptual and theoretical contributions are the authors' own.

\bibliographystyle{plainnat}
\bibliography{Master}

\begin{thebibliography}{68}
\providecommand{\natexlab}[1]{#1}
\providecommand{\url}[1]{\texttt{#1}}
\expandafter\ifx\csname urlstyle\endcsname\relax
  \providecommand{\doi}[1]{doi: #1}\else
  \providecommand{\doi}{doi: \begingroup \urlstyle{rm}\Url}\fi

\bibitem[Balakrishnan et~al.(2026)Balakrishnan, Kennedy, and
  Wasserman]{balakrishnan2026fundamental}
Sivaraman Balakrishnan, Edward Kennedy, and Larry Wasserman.
\newblock The fundamental limits of structure-agnostic functional estimation.
\newblock \emph{Statistical Science}, 41\penalty0 (3):\penalty0 659--670, 2026.

\bibitem[Bhattacharya and Ghosh(1992)]{bhattacharya1992class}
Rabi~N Bhattacharya and Jayanta~K Ghosh.
\newblock A class of ${U}$-statistics and asymptotic normality of the number of
  $k$-clusters.
\newblock \emph{Journal of Multivariate Analysis}, 43\penalty0 (2):\penalty0
  300--330, 1992.

\bibitem[Bickel and Ritov(1988)]{bickel1988estimating}
Peter~J Bickel and Ya'acov Ritov.
\newblock Estimating integrated squared density derivatives: Sharp best order
  of convergence estimates.
\newblock \emph{Sankhy{\=a}: The Indian Journal of Statistics, Series A},
  50\penalty0 (3):\penalty0 381--393, 1988.

\bibitem[Bobkov(2024)]{bobkov2024fisher}
Sergey~G Bobkov.
\newblock Fisher-type information involving higher order derivatives.
\newblock \emph{arXiv preprint arXiv:2412.10200}, 2024.

\bibitem[Bonhomme et~al.(2026)Bonhomme, Jochmans, Newey, and
  Weidner]{bonhomme2026higher}
St{\'e}phane Bonhomme, Koen Jochmans, Whitney~K Newey, and Martin Weidner.
\newblock Higher-order {N}eyman orthogonality in moment-condition models.
\newblock \emph{arXiv preprint arXiv:2605.10842}, 2026.

\bibitem[Bonvini and Kennedy(2026)]{bonvini2026fast}
Matteo Bonvini and Edward~H Kennedy.
\newblock Fast convergence rates for dose-response estimation.
\newblock \emph{Electronic Journal of Statistics}, 20\penalty0 (2):\penalty0
  3839--3907, 2026.

\bibitem[Bonvini et~al.(2024)Bonvini, Kennedy, Dukes, and
  Balakrishnan]{bonvini2024doubly}
Matteo Bonvini, Edward~H Kennedy, Oliver Dukes, and Sivaraman Balakrishnan.
\newblock Doubly-robust inference and optimality in structure-agnostic models
  with smoothness.
\newblock \emph{arXiv preprint arXiv:2405.08525}, 2024.

\bibitem[Breunig and Chen(2024)]{breunig2024adaptive}
Christoph Breunig and Xiaohong Chen.
\newblock Adaptive, rate-optimal hypothesis testing in nonparametric {IV}
  models.
\newblock \emph{Econometrica}, 92\penalty0 (6):\penalty0 2027--2067, 2024.

\bibitem[Breunig et~al.(2025)Breunig, Liu, and Yu]{breunig2025double}
Christoph Breunig, Ruixuan Liu, and Zhengfei Yu.
\newblock Double robust {B}ayesian inference on average treatment effects.
\newblock \emph{Econometrica}, 93\penalty0 (2):\penalty0 539--568, 2025.

\bibitem[Bruns-Smith et~al.(2026)Bruns-Smith, Dukes, Feller, and
  Ogburn]{bruns2026augmented}
David Bruns-Smith, Oliver Dukes, Avi Feller, and Elizabeth~L Ogburn.
\newblock Augmented balancing weights as linear regression.
\newblock \emph{Journal of the Royal Statistical Society Series B: Statistical
  Methodology}, 88\penalty0 (3):\penalty0 699--723, 2026.

\bibitem[Cattaneo and Jansson(2018)]{cattaneo2018kernel}
Matias~D Cattaneo and Michael Jansson.
\newblock Kernel-based semiparametric estimators: Small bandwidth asymptotics
  and bootstrap consistency.
\newblock \emph{Econometrica}, 86\penalty0 (3):\penalty0 955--995, 2018.

\bibitem[Cattaneo et~al.(2018)Cattaneo, Jansson, and
  Newey]{cattaneo2018inference}
Matias~D Cattaneo, Michael Jansson, and Whitney~K Newey.
\newblock Inference in linear regression models with many covariates and
  heteroscedasticity.
\newblock \emph{Journal of the American Statistical Association}, 113\penalty0
  (523):\penalty0 1350--1361, 2018.

\bibitem[Cattaneo et~al.(2019)Cattaneo, Jansson, and Ma]{cattaneo2019two}
Matias~D Cattaneo, Michael Jansson, and Xinwei Ma.
\newblock Two-step estimation and inference with possibly many included
  covariates.
\newblock \emph{The Review of Economic Studies}, 86\penalty0 (3):\penalty0
  1095--1122, 2019.

\bibitem[Cavaliere et~al.(2024)Cavaliere, Gon{\c{c}}alves, Nielsen, and
  Zanelli]{cavaliere2024bootstrap}
Giuseppe Cavaliere, S{\'\i}lvia Gon{\c{c}}alves, Morten~{\O}rregaard Nielsen,
  and Edoardo Zanelli.
\newblock Bootstrap inference in the presence of bias.
\newblock \emph{Journal of the American Statistical Association}, 119\penalty0
  (548):\penalty0 2908--2918, 2024.

\bibitem[Chakrabortty and Kuchibhotla(2025)]{chakrabortty2025tail}
Abhishek Chakrabortty and Arun~K Kuchibhotla.
\newblock Tail bounds for canonical ${U}$-statistics and ${U}$-processes with
  unbounded kernels.
\newblock \emph{arXiv preprint arXiv:2504.01318}, 2025.

\bibitem[Chen et~al.(2024)Chen, Liu, and Mukherjee]{chen2024method}
Xingyu Chen, Lin Liu, and Rajarshi Mukherjee.
\newblock Method-of-moments inference for {GLM}s and doubly-robust functionals
  under proportional asymptotics.
\newblock \emph{arXiv preprint arXiv:2408.06103}, 2024.

\bibitem[Chen et~al.(2026)Chen, Liu, and Zhang]{chen2026computing}
Xingyu Chen, Lin Liu, and Ruiqi Zhang.
\newblock On computing and the complexity of computing higher-order
  ${U}$-statistics, exactly.
\newblock \emph{Statistics and Computing}, 2026.

\bibitem[Chernozhukov et~al.(2018)Chernozhukov, Chetverikov, Demirer, Duflo,
  Hansen, Newey, and Robins]{chernozhukov2018double}
Victor Chernozhukov, Denis Chetverikov, Mert Demirer, Esther Duflo, Christian
  Hansen, Whitney Newey, and James Robins.
\newblock Double/debiased machine learning for treatment and structural
  parameters.
\newblock \emph{The Econometrics Journal}, 21\penalty0 (1):\penalty0 C1--C68,
  2018.

\bibitem[Chernozhukov et~al.(2022)Chernozhukov, Escanciano, Ichimura, Newey,
  and Robins]{chernozhukov2022locally}
Victor Chernozhukov, Juan~Carlos Escanciano, Hidehiko Ichimura, Whitney~K
  Newey, and James~M Robins.
\newblock Locally robust semiparametric estimation.
\newblock \emph{Econometrica}, 90\penalty0 (4):\penalty0 1501--1535, 2022.

\bibitem[Damian et~al.(2025)Damian, Lee, and Bruna]{damian2025generative}
Alex Damian, Jason~D Lee, and Joan Bruna.
\newblock The generative leap: Tight sample complexity for efficiently learning
  {G}aussian multi-index models.
\newblock In \emph{Proceedings of The Thirty-ninth Annual Conference on Neural
  Information Processing Systems}, pages 28276--28311, 2025.

\bibitem[Fisher and Kennedy(2021)]{fisher2021visually}
Aaron Fisher and Edward~H Kennedy.
\newblock Visually communicating and teaching intuition for influence
  functions.
\newblock \emph{The American Statistician}, 75\penalty0 (2):\penalty0 162--172,
  2021.

\bibitem[Hahn(1998)]{hahn1998role}
Jinyong Hahn.
\newblock On the role of the propensity score in efficient semiparametric
  estimation of average treatment effects.
\newblock \emph{Econometrica}, 66\penalty0 (2):\penalty0 315--331, 1998.

\bibitem[Hines et~al.(2022)Hines, Dukes, Diaz-Ordaz, and
  Vansteelandt]{hines2022demystifying}
Oliver Hines, Oliver Dukes, Karla Diaz-Ordaz, and Stijn Vansteelandt.
\newblock Demystifying statistical learning based on efficient influence
  functions.
\newblock \emph{The American Statistician}, 76\penalty0 (3):\penalty0 292--304,
  2022.

\bibitem[Kandasamy et~al.(2015)Kandasamy, Krishnamurthy, P{\'o}czos, Wasserman,
  and Robins]{kandasamy2015nonparametric}
Kirthevasan Kandasamy, Akshay Krishnamurthy, Barnab{\"y}s P{\'o}czos, Larry
  Wasserman, and James~M Robins.
\newblock Nonparametric von {M}ises estimators for entropies, divergences and
  mutual informations.
\newblock In \emph{Proceedings of the 29th International Conference on Neural
  Information Processing Systems-Volume 1}, pages 397--405, 2015.

\bibitem[Kennedy(2023)]{kennedy2023towards}
Edward~H Kennedy.
\newblock Towards optimal doubly robust estimation of heterogeneous causal
  effects.
\newblock \emph{Electronic Journal of Statistics}, 17\penalty0 (2):\penalty0
  3008--3049, 2023.

\bibitem[Kennedy et~al.(2024)Kennedy, Balakrishnan, Robins, and
  Wasserman]{kennedy2024minimax}
Edward~H Kennedy, Sivaraman Balakrishnan, James~M Robins, and Larry Wasserman.
\newblock Minimax rates for heterogeneous causal effect estimation.
\newblock \emph{The Annals of Statistics}, 52\penalty0 (2):\penalty0 793--816,
  2024.

\bibitem[Koltchinskii(2022)]{koltchinskii2022bootstrap}
Vladimir Koltchinskii.
\newblock Estimation of smooth functionals in high-dimensional models:
  Bootstrap chains and {G}aussian approximation.
\newblock \emph{The Annals of Statistics}, 50\penalty0 (4):\penalty0
  2386--2415, 2022.

\bibitem[Koltchinskii(2025)]{koltchinskii2025estimation}
Vladimir Koltchinskii.
\newblock Estimation of smooth functionals of covariance operators: Jackknife
  bias reduction and bounds in terms of effective rank.
\newblock \emph{Annales de l'Institut Henri Poincare (B) Probabilites et
  statistiques}, 61\penalty0 (1):\penalty0 665--712, 2025.

\bibitem[Ledoit and Wolf(2020)]{ledoit2020analytical}
Olivier Ledoit and Michael Wolf.
\newblock Analytical nonlinear shrinkage of large-dimensional covariance
  matrices.
\newblock \emph{The Annals of Statistics}, 48\penalty0 (5):\penalty0
  3043--3065, 2020.

\bibitem[Lin et~al.(2024)Lin, Su, Mou, Ding, and Wainwright]{lin2024worthwhile}
Licong Lin, Fangzhou Su, Wenlong Mou, Peng Ding, and Martin Wainwright.
\newblock When is it worthwhile to jackknife? {B}reaking the quadratic barrier
  for ${Z}$-estimators.
\newblock \emph{arXiv preprint arXiv:2411.02909}, 2024.

\bibitem[Liu and Li(2023)]{liu2023hoif}
Lin Liu and Chang Li.
\newblock New $\sqrt{n}$-consistent, numerically stable empirical higher-order
  influence function estimators.
\newblock \emph{arXiv preprint arXiv:2302.08097}, 2023.

\bibitem[Liu et~al.(2017)Liu, Mukherjee, Newey, and
  Robins]{liu2017semiparametric}
Lin Liu, Rajarshi Mukherjee, Whitney~K Newey, and James~M Robins.
\newblock Semiparametric efficient empirical higher order influence function
  estimators.
\newblock \emph{arXiv preprint arXiv:1705.07577}, 2017.

\bibitem[Liu et~al.(2020)Liu, Mukherjee, and Robins]{liu2020nearly}
Lin Liu, Rajarshi Mukherjee, and James~M Robins.
\newblock On nearly assumption-free tests of nominal confidence interval
  coverage for causal parameters estimated by machine learning.
\newblock \emph{Statistical Science}, 35\penalty0 (3):\penalty0 518--539, 2020.

\bibitem[Liu et~al.(2023)Liu, Wang, and Wang]{liu2023root}
Lin Liu, Xinbo Wang, and Yuhao Wang.
\newblock Root-n consistent semiparametric learning with high-dimensional
  nuisance functions under minimal sparsity.
\newblock \emph{arXiv preprint arXiv:2305.04174}, 2023.

\bibitem[Liu et~al.(2024)Liu, Mukherjee, and Robins]{liu2024assumption}
Lin Liu, Rajarshi Mukherjee, and James~M Robins.
\newblock Assumption-lean falsification tests of rate double-robustness of
  double-machine-learning estimators.
\newblock \emph{Journal of Econometrics}, 240\penalty0 (2):\penalty0 105500,
  2024.

\bibitem[Liu et~al.(2026)Liu, Mukherjee, and Robins]{liu2026asymptotic}
Lin Liu, Rajarshi Mukherjee, and James~M Robins.
\newblock On the asymptotic inadmissibility of double machine learning
  estimators under structure-agnostic models.
\newblock \emph{arXiv preprint arXiv:2606.22391}, 2026.

\bibitem[Liu et~al.(2025)Liu, Minervini, Patel, and Wilde]{liu2025quantum}
Nana Liu, Michele Minervini, Dhrumil Patel, and Mark~M Wilde.
\newblock Quantum thermodynamics and semi-definite optimization.
\newblock \emph{arXiv preprint arXiv:2505.04514}, 2025.

\bibitem[McClean et~al.(2026)McClean, Balakrishnan, Kennedy, and
  Wasserman]{mcclean2026double}
Alec McClean, Sivaraman Balakrishnan, Edward~H Kennedy, and Larry Wasserman.
\newblock Double cross-fit doubly robust estimators: Beyond series regression.
\newblock \emph{Journal of the Royal Statistical Society Series B: Statistical
  Methodology}, 2026.

\bibitem[McGrath and Mukherjee(2026)]{mcgrath2026nuisance}
Sean McGrath and Rajarshi Mukherjee.
\newblock Nuisance function tuning and sample splitting for optimally
  estimating a doubly robust functional.
\newblock \emph{The Annals of Statistics}, 2026.

\bibitem[Newey and Robins(2018)]{newey2018cross}
Whitney~K Newey and James~M Robins.
\newblock Cross-fitting and fast remainder rates for semiparametric estimation.
\newblock \emph{arXiv preprint arXiv:1801.09138}, 2018.

\bibitem[Newey et~al.(2004)Newey, Hsieh, and Robins]{newey2004twicing}
Whitney~K Newey, Fushing Hsieh, and James~M Robins.
\newblock Twicing kernels and a small bias property of semiparametric
  estimators.
\newblock \emph{Econometrica}, 72\penalty0 (3):\penalty0 947--962, 2004.

\bibitem[Niu et~al.(2024)Niu, Chakraborty, Dukes, and
  Katsevich]{niu2024reconciling}
Ziang Niu, Abhinav Chakraborty, Oliver Dukes, and Eugene Katsevich.
\newblock Reconciling model-{X} and doubly robust approaches to conditional
  independence testing.
\newblock \emph{The Annals of Statistics}, 52\penalty0 (3):\penalty0 895--921,
  2024.

\bibitem[Pfanzagl(1983)]{pfanzagl1983asymptotic}
Johann Pfanzagl.
\newblock \emph{Asymptotic Expansions for General Statistical Models},
  volume~31 of \emph{Lecture Notes in Statistics}.
\newblock Springer Science \& Business Media, 1983.

\bibitem[Ray and van~der Vaart(2020)]{ray2020semiparametric}
Kolyan Ray and Aad van~der Vaart.
\newblock Semiparametric {B}ayesian causal inference.
\newblock \emph{The Annals of Statistics}, 48\penalty0 (5):\penalty0
  2999--3020, 2020.

\bibitem[Ritov and Bickel(1990)]{ritov1990achieving}
Ya'acov Ritov and Peter~J Bickel.
\newblock Achieving information bounds in non and semiparametric models.
\newblock \emph{The Annals of Statistics}, 18\penalty0 (2):\penalty0 925--938,
  1990.

\bibitem[Robins et~al.(2007)Robins, Sued, Lei-Gomez, and
  Rotnitzky]{robins2007comment}
James Robins, Mariela Sued, Quanhong Lei-Gomez, and Andrea Rotnitzky.
\newblock Comment: Performance of double-robust estimators when ``inverse
  probability'' weights are highly variable.
\newblock \emph{Statistical Science}, 22\penalty0 (4):\penalty0 544--559, 2007.

\bibitem[Robins et~al.(2008)Robins, Li, Tchetgen~Tchetgen, and van~der
  Vaart]{robins2008higher}
James Robins, Lingling Li, Eric Tchetgen~Tchetgen, and Aad van~der Vaart.
\newblock Higher order influence functions and minimax estimation of nonlinear
  functionals.
\newblock In \emph{Probability and Statistics: Essays in Honor of David A.
  Freedman}, pages 335--421. Institute of Mathematical Statistics, 2008.

\bibitem[Robins et~al.(2009)Robins, Tchetgen~Tchetgen, Li, and van~der
  Vaart]{robins2009semiparametric}
James Robins, Eric Tchetgen~Tchetgen, Lingling Li, and Aad van~der Vaart.
\newblock Semiparametric minimax rates.
\newblock \emph{Electronic Journal of Statistics}, 3:\penalty0 1305--1321,
  2009.

\bibitem[Robins et~al.(2016)Robins, Li, Tchetgen~Tchetgen, and van~der
  Vaart]{robins2016technical}
James Robins, Lingling Li, Eric Tchetgen~Tchetgen, and Aad van~der Vaart.
\newblock Technical report: Higher order influence functions and minimax
  estimation of nonlinear functionals.
\newblock \emph{arXiv preprint arXiv:1601.05820}, 2016.

\bibitem[Robins et~al.(1994)Robins, Rotnitzky, and Zhao]{robins1994estimation}
James~M Robins, Andrea Rotnitzky, and Lue~Ping Zhao.
\newblock Estimation of regression coefficients when some regressors are not
  always observed.
\newblock \emph{Journal of the American Statistical Association}, 89\penalty0
  (427):\penalty0 846--866, 1994.

\bibitem[Robins et~al.(2017)Robins, Li, Mukherjee, Tchetgen~Tchetgen, and
  van~der Vaart]{robins2017minimax}
James~M Robins, Lingling Li, Rajarshi Mukherjee, Eric Tchetgen~Tchetgen, and
  Aad van~der Vaart.
\newblock Minimax estimation of a functional on a structured high-dimensional
  model.
\newblock \emph{The Annals of Statistics}, 45\penalty0 (5):\penalty0
  1951--1987, 2017.

\bibitem[Rotnitzky et~al.(2021)Rotnitzky, Smucler, and
  Robins]{rotnitzky2021characterization}
Andrea Rotnitzky, Ezequiel Smucler, and James~M Robins.
\newblock Characterization of parameters with a mixed bias property.
\newblock \emph{Biometrika}, 108\penalty0 (1):\penalty0 231--238, 2021.

\bibitem[Sch{\"a}fer(2026)]{schafer2026mobius}
Florian Sch{\"a}fer.
\newblock M{\"o}bius inversion and the iterated bootstrap.
\newblock \emph{SIAM Journal on Mathematics of Data Science}, 8\penalty0
  (2):\penalty0 362--381, 2026.

\bibitem[Shah and Peters(2020)]{shah2020hardness}
Rajen~D Shah and Jonas Peters.
\newblock The hardness of conditional independence testing and the generalised
  covariance measure.
\newblock \emph{The Annals of Statistics}, 48\penalty0 (3):\penalty0
  1514--1538, 2020.

\bibitem[Shpitser et~al.(2011)Shpitser, Richardson, and
  Robins]{shpitser2011efficient}
Ilya Shpitser, Thomas~S Richardson, and James~M Robins.
\newblock An efficient algorithm for computing interventional distributions in
  latent variable causal models.
\newblock In \emph{Proceedings of the Twenty-Seventh Conference on Uncertainty
  in Artificial Intelligence}, pages 661--670, 2011.

\bibitem[Small and McLeish(1989)]{small1989projection}
Christopher~G Small and Don~L McLeish.
\newblock Projection as a method for increasing sensitivity and eliminating
  nuisance parameters.
\newblock \emph{Biometrika}, 76\penalty0 (4):\penalty0 693--703, 1989.

\bibitem[Stanley(2011)]{stanley2011enumerative}
Richard~P Stanley.
\newblock \emph{Enumerative Combinatorics}, volume~1.
\newblock Cambridge University Press, 2011.

\bibitem[Tropp(2015)]{tropp2015introduction}
Joel~A Tropp.
\newblock An introduction to matrix concentration inequalities.
\newblock \emph{Foundations and Trends{\textregistered} in Machine Learning},
  8\penalty0 (1-2):\penalty0 1--230, 2015.

\bibitem[van~der Laan et~al.(2021)van~der Laan, Wang, and van~der
  Laan]{van2021higher}
Mark van~der Laan, Zeyi Wang, and Lars van~der Laan.
\newblock Higher order targeted maximum likelihood estimation.
\newblock \emph{arXiv preprint arXiv:2101.06290}, 2021.

\bibitem[van~der Laan and Rubin(2006)]{van2006targeted}
Mark~J van~der Laan and Daniel Rubin.
\newblock Targeted maximum likelihood learning.
\newblock \emph{The International Journal of Biostatistics}, 2\penalty0
  (1):\penalty0 11, 2006.

\bibitem[van~der Vaart(1991)]{van1991differentiable}
Aad van~der Vaart.
\newblock On differentiable functionals.
\newblock \emph{The Annals of Statistics}, 19\penalty0 (1):\penalty0 178--204,
  1991.

\bibitem[van~der Vaart(2014)]{van2014higher}
Aad van~der Vaart.
\newblock Higher order tangent spaces and influence functions.
\newblock \emph{Statistical Science}, 29\penalty0 (4):\penalty0 679--686, 2014.

\bibitem[Vansteelandt(2025)]{vansteelandt2025towards}
Stijn Vansteelandt.
\newblock Towards efficient and interpretable assumption-lean generalized
  linear modeling of continuous exposure effects.
\newblock \emph{Biometrics}, 81\penalty0 (2):\penalty0 ujaf071, 2025.

\bibitem[Vansteelandt and Dukes(2022)]{vansteelandt2022assumption}
Stijn Vansteelandt and Oliver Dukes.
\newblock Assumption-lean inference for generalised linear model parameters.
\newblock \emph{Journal of the Royal Statistical Society Series B: Statistical
  Methodology}, 84\penalty0 (3):\penalty0 657--685, 2022.

\bibitem[Verzelen and Gassiat(2018)]{verzelen2018adaptive}
Nicolas Verzelen and Elisabeth Gassiat.
\newblock Adaptive estimation of high-dimensional signal-to-noise ratios.
\newblock \emph{Bernoulli}, 24\penalty0 (4B):\penalty0 3683--3710, 2018.

\bibitem[Waterman and Lindsay(1996)]{waterman1996projected}
Richard~P Waterman and Bruce~G Lindsay.
\newblock Projected score methods for approximating conditional scores.
\newblock \emph{Biometrika}, 83\penalty0 (1):\penalty0 1--13, 1996.

\bibitem[Wein et~al.(2019)Wein, El~Alaoui, and Moore]{wein2019kikuchi}
Alexander~S Wein, Ahmed El~Alaoui, and Cristopher Moore.
\newblock The {K}ikuchi hierarchy and tensor {PCA}.
\newblock In \emph{2019 IEEE 60th Annual Symposium on Foundations of Computer
  Science (FOCS)}, pages 1446--1468. IEEE, 2019.

\bibitem[Zhang et~al.(2026)Zhang, Liu, and Zhang]{zhang2026higher}
Yulin Zhang, Lin Liu, and Zheng Zhang.
\newblock Higher-order debiased estimators for general treatment models.
\newblock \emph{Econometric Theory}, 2026.

\end{thebibliography}

\appendix

\allowdisplaybreaks

\part*{Appendix}

The Appendix of this article is divided into three parts. Appendix~\ref{app:sim} describes the setup of the simulation results exhibited in Section~\ref{sec:theory} of the main text. Appendix~\ref{app:theory} contains the proof of Lemma~\ref{lem:mobius} and Lemma~\ref{lem:graph-counting} and fills in the sketch of the proof of Theorem~\ref{thm:main} delineated in Section~\ref{sec:proof}. Appendix~\ref{app:lemma} further contains some technical results used in Appendix~\ref{app:theory}.

\renewcommand{\theHsection}{A\arabic{section}}

\section{Simulation Setup}
\label{app:sim}

In this section, we describe the simulation setup of Figure~\ref{fig:sim} reported in Section~\ref{sec:theory} of the main text. We consider a simple example of the bilinear form \eqref{target}. Let $X \sim \calN (0, I)$, so that $\Sigma = \Omega = I$ with changing dimensions ($k$). Let
$$A = X_{1} + \varepsilon_{A}, \quad Y = X_{1} + \varepsilon_{Y}, \quad \varepsilon_{A}, \varepsilon_{Y} \stackrel{\mathrm{iid}}{\sim} \calN (0,1),$$
where $X_{1}$ denotes the first coordinate of $X$. Then $\mu = \bbE (X A) = e_{1}$, $\eta = \bbE (X Y) = e_{1}$, and the target is $\psi = \mu^{\top} \Omega \eta = 1$, independent of $k$. We compare three estimators of $\psi$:
\begin{itemize}
\item \textbf{Oracle}: $\hat{\psi}_{2, k}(I)$ is exactly unbiased for $\psi$.

\item \textbf{Sample-split HOIF} at $m = 3$: $\hat{\psi}_{3, k} (\hat{\Omega}_{\nuis})$, with $\hat{\Omega}_{\nuis} = \hat \Sigma_{\nuis}^{-1}$ computed from an \emph{independent} nuisance sample $\calD_{\nuis}$ of the same size $n$.

\item \textbf{Same-sample (stabilized) HOIF} at $m = 3$: $\hat{\psi}_{3, k} (\hat{\Omega})$.
\end{itemize}

We fix $n = 300$ and vary $k$ so that $\rho \coloneqq k / n$ ranges over $\{0.05, 0.15, 0.30, 0.50, 0.70, 0.85\}$, using $B = 250$ Monte-Carlo replications per configuration. Table~\ref{tab:stability} and Figure~\ref{fig:sim} report the bias, standard deviation (SD), and root mean squared error (RMSE) of each estimator.

\begin{table}[H]
\centering
\small
\begin{tabular}{c c ccc ccc}
\toprule
 & Oracle & \multicolumn{3}{c}{Sample-split ($m{=}3$)}
       & \multicolumn{3}{c}{Same-sample ($m{=}3$)}\\
\cmidrule(lr){3-5}\cmidrule(lr){6-8}
$\rho$ & RMSE
       & bias & SD & RMSE & bias & SD & RMSE\\
\midrule
0.05   & 0.168 & $-0.062$ & 0.117 & 0.132 & \phantom{$-$}0.049 & 0.113 & 0.123\\
0.15  & 0.177 & $-0.258$ & 0.174 & 0.311 & \phantom{$-$}0.103 & 0.125 & 0.162\\
0.30  & 0.180 & $-1.127$ & 0.783 & 1.372 & \phantom{$-$}0.109 & 0.130 & 0.169\\
0.50  & 0.194 & $-5.566$ & 3.873 & 6.781 & \phantom{$-$}0.015 & 0.117 & 0.118\\
0.70 & 0.228 & $-32.9$  & 30.2  & 44.7  & $-0.274$           & 0.108 & 0.295\\
0.85 & 0.214 & $-343.2$ & 403.7 & 529.9 & $-0.595$           & 0.073 & 0.599\\
\bottomrule
\end{tabular}
\caption{Finite-sample performance at $n = 300$ based on $B=250$ Monte Carlo runs. As $\rho = k / n \to 1$, $\hat{\psi}_{3, k} (\hat{\Omega}_{\nuis})$ starts to diverge, while $\hat{\psi}_{3, k} (\hat{\Omega})$ remains comparable to the oracle $\hat{\psi}_{2, k} (I)$.}
\label{tab:stability}
\end{table}

From Table~\ref{tab:stability} and Figure~\ref{fig:sim}, it is quite evident that the finite-sample performance of $\hat{\psi}_{3, k} (\hat{\Omega})$ is superior to that of $\hat{\psi}_{3, k} (\hat{\Omega}_{\nuis})$, especially when $\rho$ gets larger.


\section{Technical Details of the Proof}
\label{app:theory}

\subsection{Proof of Lemma~\ref{lem:mobius}}
\label{app:mobius}

\begin{proof}
Fix $j \ge 3$ and write $\iota = j - 2$. Given any
$i_{1} \neq i_{2}$, let $C_{i_{1} i_{2}} \coloneqq [n] \setminus \{i_{1}, i_{2}\}$.
Define
\begin{align*}
\calS_{\iota}^{i_{1} i_{2}} \coloneqq \frac{(n - j) !}{(n - 2) !}
\sum_{\substack{(\ell_{1}, \cdots, \ell_{\iota}) \in C_{i_{1} i_{2}}^{\iota}\\
\ell_{1} \neq \cdots \neq \ell_{\iota}}}
H_{\ell_{1}} \cdots H_{\ell_{\iota}}.
\end{align*}
Then $\hat{\IIFF}_{j, j, k} (\hat{\Omega})$ admits the following representation: $\hat{\IIFF}_{j, j, k} (\hat{\Omega}) = (-1)^{j} \bbU_{n, 2} (A_{1} X_{1}^{\top} \hat{\Omega} \calS_{\iota}^{1 2} X_{2} Y_{2})$.

We first expand the summation in $\calS_{\iota}^{i_{1} i_{2}}$.
Let $\Pi_{\iota}$ denote the lattice of all partitions of
$\{1, \cdots, \iota\}$. For $\frakm \in \Pi_{\iota}$ (so $\frakm$ is a partition and contains non-overlapping subsets of $\{1, \cdots, \iota\}$), write $B_\frakm(l)$ as the element in the partition $\frakm$ containing $l$, and define $ \mu_{\iota} (\frakm) \coloneqq (-1)^{\iota - |\frakm|} \prod_{B \in \frakm} (|B| - 1) ! $.

By M\"obius inversion on the partition lattice \citep{stanley2011enumerative} (see Lemma~\ref{lem:mobius-inversion expansion}),
\begin{equation}
\label{mobius summation}
\sum_{\substack{(\ell_{1}, \cdots, \ell_{\iota}) \in C_{i_{1} i_{2}}^{\iota}\\
\ell_{1} \neq \cdots \neq \ell_{\iota}}}
H_{\ell_{1}} \cdots H_{\ell_{\iota}} = \sum_{\frakm \in \Pi_{\iota}}
\mu_{\iota} (\frakm) \sum_{\{b_B\}_{B \in \frakm} \in C_{i_{1} i_{2}}^{|\frakm|}} \prod_{l = 1}^{\iota} H_{b_{B_{\frakm} (l)}},
\end{equation}
where in $\{b_{B}\}_{B \in \frakm}$, each $b_{B}$ takes values in $C_{i_{1}, i_{2}}$ and $b_{B_{1}}$ and $b_{B_{2}}$ can take the same value even if $B_{1} \neq B_{2} \in \frakm$.

For a partition $\frakm$, let $\frakm_{\ge 2} \coloneqq \{B \in \frakm: |B| \ge 2\}$ and $\frakm_{1} \coloneqq \{B \in \frakm: |B| = 1\}$.
Since $\sum_{i = 1}^{n} H_{i} = 0$ (Lemma~\ref{lem:empirical centering}), for any $i_{1} \ne i_{2}$, we have 
\begin{equation}
\label{empirical centering 2}
\sum_{\ell \in C_{i_{1} i_{2}}} H_{\ell} = -(H_{i_{1}} + H_{i_{2}}) = - \calR_{i_{1} i_{2}}. 
\end{equation}
We now separate the variables indexed by the singleton elements of $\frakm$ from those indexed by $\frakm_{\ge2}$. The second summation on the RHS of \eqref{mobius summation} can be written as
\begin{align*}
\sum_{\{b_{B}\}_{B \in \frakm} \in C_{i_{1} i_{2}}^{|\frakm|}}
\prod_{l = 1}^{\iota} H_{b_{B_{\frakm} (l)}} =
\sum_{\{b_{B}\}_{B \in \frakm_{\ge2}}\in C_{i_{1} i_{2}}^{|\frakm_{\ge 2}|}}
\sum_{\{b_{B}\}_{B \in \frakm_{1}}\in C_{i_{1} i_{2}}^{|\frakm_{1}|}}
\prod_{l = 1}^{\iota} H_{b_{B_\frakm (l)}}.
\end{align*}
Since each element of $\frakm_{1}$ is a singleton, each corresponding summation appears in exactly one factor of the ordered product. Using  $ \sum_{\ell\in C_{i_{1} i_{2}}} H_{\ell} = - \calR_{i_{1} i_{2}}$,
and preserving the original order of multiplication, we obtain
\begin{align*}
\sum_{\{b_{B}\}_{B \in \frakm} \in C_{i_{1} i_{2}}^{|\frakm|}}
\prod_{l = 1}^{\iota} H_{b_{B_{\frakm} (l)}} =
(-1)^{|\frakm_{1}|}\sum_{\{b_{B}\}_{B \in \frakm_{\ge 2}}\in C_{i_{1} i_{2}}^{|\frakm_{\ge 2}|}} \prod_{l = 1}^{\iota} N_{\frakm, l}^{i_{1} i_{2}},
\end{align*}
where the product is ordered in $l$, and
\begin{equation*}
N_{\frakm, l}^{i_{1} i_{2}} \coloneqq
\begin{cases}
\calR_{i_{1} i_{2}}, & B_{\frakm} (l) = \{l\}, \\
H_{b_{B_{\frakm} (l)}}, & |B_{\frakm} (l)| \ge 2.
\end{cases}
\end{equation*}
Consequently,
\begin{align*}
\calS_{\iota}^{i_{1} i_{2}} & = \frac{(n - j) !}{(n - 2) !} \sum_{\frakm \in \Pi_{\iota}} \mu_{\iota} (\frakm) (-1)^{|\frakm_{1}|} \sum_{\{b_{B}\}_{B \in \frakm_{\ge 2}} \in C_{i_{1} i_{2}}^{|\frakm_{\ge 2}|}} \prod_{l = 1}^{\iota} N_{\frakm, l}^{i_{1} i_{2}}.
\end{align*}

Furthermore, let $\tau$ be a partition of $\frakm_{\ge 2}$, and write $\tau = \{D_1, \cdots, D_{|\tau|}\}$. 
Here, the element $D_{c}$, for $c \in \{1, \cdots, |\tau|\}$, collects all elements in $\frakm_{\geq 2}$ that share the same sample index $a_c$ in the subscript of $H$'s. Define
\begin{equation*}
M_{\frakm, \tau, l}^{i_{1} i_{2}} \coloneqq 
\begin{cases}
\calR_{i_{1} i_{2}}, & B_{\frakm} (l) = \{l\}, \\
H_{a_{c}}, & B_{\frakm} (l) \in D_{c}.
\end{cases}
\end{equation*}
Then:
\begin{equation}
\label{equality pattern}
\sum_{\{b_{B}\}_{B \in \frakm_{\ge 2}} \in C_{i_{1} i_{2}}^{|\frakm_{\ge 2}|}} \prod_{l = 1}^{\iota} N_{\frakm, l}^{i_{1} i_{2}} = \sum_{\tau \in \Pi (\frakm_{\ge 2})} \sum_{\substack{a_{1}, \cdots, a_{|\tau|} \in C_{i_{1} i_{2}} \\
a_{1} \neq \cdots \neq a_{|\tau|}}} \prod_{l = 1}^{\iota} M_{\frakm, \tau, l}^{i_{1} i_{2}}.
\end{equation}
For any kernel $K (i_{1}, i_{2}; a_{1}, \cdots, a_{r})$, the following identity holds
\begin{align*}
\bbU_{n, 2} \Big\{\sum_{\substack{a_{1}, \cdots, a_{r} \in C_{i_{1} i_{2}} \\
a_{1} \neq \cdots \neq a_{r}}}
K (i_{1}, i_{2}; a_{1}, \cdots, a_{r}) \Big\} = \frac{(n - 2)!}{(n - 2 - r)!} \bbU_{n, 2 + r} \{K (i_{1}, i_{2}; a_{1}, \cdots, a_{r})\}.
\end{align*}
Combining the preceding identities gives a two-level expansion indexed by $\frakm$ and $\tau$.

It remains to merge all terms that lead to the same non-singleton element family; see Remark~\ref{rem:merging} for an illustration on how the merging step is carried out. Recall the definition of $\calB = \{B_{1}, \cdots, B_{r}\}$ given in Lemma~\ref{lem:mobius}. The elements in $[\iota]$ not covered by $\calB$, $[\iota] \setminus \bigcup_{\nu = 1}^{r} B_{\nu}$, correspond exactly to all singleton elements and are therefore reduced to $\calR_{i_{1} i_{2}}$. Each set $B_{\nu}$ corresponds to one distinct remaining interior index $a_{\nu}$.

For the first-level partition $\frakm$, the sign and M\"obius factor appearing together with the original factor $(-1)^{j}$ are
\begin{align*}
(-1)^{j} \mu_{\iota} (\frakm) (-1)^{|\frakm_{1}|} = (-1)^{j + \iota -|\frakm| + |\frakm_{1}|} \prod_{A \in \frakm}(|A| - 1)!.
\end{align*}
Since $\iota = j - 2$, the integer $j + \iota = 2 j - 2$ is even. Hence $(-1)^{j + \iota -|\frakm| + |\frakm_{1}|} = (-1)^{|\frakm| - |\frakm_{1}|} = (-1)^{|\frakm_{\ge 2}|}$. Singleton elements contribute $0 ! = 1$, so the coefficient attached to the non-singleton elements of $\frakm$ is $(-1)^{|\frakm_{\ge 2}|} \prod_{A \in \frakm_{\ge 2}} (|A| - 1)!$.


Fix a collection $\calB = \{B_{1}, \cdots, B_{r}\} \in \bbB_{\iota}$ arising in the merging step, and collect all pairs $(\frakm, \tau)$ that lead to this same collection. For such a pair, the non-singleton part of $\frakm$ decomposes uniquely as $\frakm_{\ge 2} = \pi_{1} \sqcup \cdots \sqcup \pi_{r}$, where $\pi_{\nu} \in \Pi_{\ge 2} (B_{\nu})$ and $\Pi_{\ge 2} (B_{\nu})$ denotes the set of partitions of $B_{\nu}$ whose elements all have cardinality at least two. Here $\pi_{\nu}$ is the collection of non-singleton elements of $\frakm$ whose union is $B_{\nu}$. Hence, for each fixed choice $(\pi_{1}, \cdots, \pi_{r})$,
\begin{equation*}
(-1)^{|\frakm_{\ge 2}|} \prod_{A \in \frakm_{\ge 2}}(|A| - 1)! = \prod_{\nu = 1}^{r} \Big\{ (-1)^{|\pi_{\nu}|} \prod_{A \in \pi_{\nu}} (|A| - 1)! \Big\}.
\end{equation*}
Therefore the merged coefficient attached to $\calB$ is
\begin{align*}
C (\calB) & = \sum_{\pi_{1} \in \Pi_{\ge 2} (B_{1})} \cdots \sum_{\pi_{r} \in \Pi_{\ge 2} (B_{r})} \prod_{\nu = 1}^{r} \Big\{ (-1)^{|\pi_{\nu}|} \prod_{A \in \pi_{\nu}} (|A| - 1)! \Big\} \\
& = \prod_{\nu = 1}^{r} \Big\{ \sum_{\pi_{\nu} \in \Pi_{\ge 2} (B_{\nu})}(-1)^{|\pi_{\nu}|} \prod_{A \in \pi_{\nu}} (|A| - 1)! \Big\}.
\end{align*}

For a finite set $B$ with $|B| = d$, define $S_{d} \coloneqq \sum_{\pi \in \Pi_{\ge 2} (B)} (-1)^{|\pi|} \prod_{A \in \pi} (|A| - 1) !$, with the conventions $S_{0} = 1$ and $S_{1} = 0$. By the following sequence of equalities, we have
\begin{align*}
\sum_{d = 0}^{\infty} S_{d} \frac{z^{d}}{d !} = \exp \Big\{ \sum_{s = 2}^{\infty} - \frac{(s - 1) !}{s !} z^{s} \Big\} = \exp \Big\{ - \sum_{s = 2}^{\infty} \frac{z^{s}}{s} \Big\} = \exp \{\log (1 - z) + z\} = e^{z} (1 - z).
\end{align*}
Since $e^{z} (1 - z) = \sum_{d = 0}^{\infty} \frac{(1 - d) z^{d}}{d !}$, comparing coefficients yields $S_{d} = 1 - d$. Consequently,
\begin{equation*}
C (\calB) = \prod_{B \in \calB} (1 - |B|) = (-1)^{|\calB|} \prod_{B \in \calB} (|B| - 1).
\end{equation*}

Putting all the above calculations together yields the following identity:
\begin{align*}
\hat{\IIFF}_{j, j, k} (\hat{\Omega}) = \sum_{\calB \in \bbB_{\iota}} (-1)^{|\calB|} \Big\{ \prod_{B \in \calB} (|B| - 1) \Big\} \frac{(n - j)!}{(n - 2 - |\calB|)!} \bbU_{n, 2 + |\calB|} (K_{\calB}).
\end{align*}
Here $K_{\calB}$ is exactly the kernel defined in \eqref{kernel}. The factor $ \frac{(n - j)!}{(n - 2 - |\calB|) !}$ comes from the original normalization $(n - j)! / (n - 2)!$ in $\calS_{\iota}^{i_{1} i_{2}}$ and from converting the remaining distinct interior sum into an ordered $U$-statistic.

Finally, since every element in $\calB$ has cardinality at least two, one has $|\calB| \le \lfloor \frac{\iota}{2} \rfloor$. Hence, the maximal order of the $U$-statistics appearing in the expansion is $2 + \lfloor \frac{\iota}{2} \rfloor = 2 + \lfloor \frac{j - 2}{2} \rfloor$.
\end{proof}

\begin{remark}[Illustration of the merging step]
\label{rem:merging}
The merging step should be understood for one fixed collection $\calB$. The final expansion then sums over all possible such $\calB$. For example, take $\iota = 7$ and fix $\calB = \{B_{1}, B_{2}\}, B_{1} = [4], B_2 = \{5, 6\}$. The position $7$ is not covered by this particular $\calB$, and therefore it is treated as a singleton position and is represented by $\calR_{i_{1} i_{2}}$ after applying Lemma~\ref{lem:empirical centering}. For this fixed $\calB$, the possible internal partitions of $B_{1}$ into sets of cardinality at least two are
\begin{align*}
\pi_{1}^{(1)} = \{\{1, 2, 3, 4\}\}, \pi_{1}^{(2)} = \{\{1, 2\},\{3, 4\}\}, \pi_{1}^{(3)} = \{\{1, 3\},\{2, 4\}\}, \pi_{1}^{(4)} = \{\{1, 4\},\{2, 3\}\}.
\end{align*}
For $B_{2} = \{5, 6\}$, there is only one such partition,
$ \pi_{2} = \{\{5, 6\}\}$. Thus all pairs $(\frakm, \tau)$ that generate this fixed $\calB$ are obtained by choosing one of the four possibilities for $\pi_{1}$ above and the unique choice of $\pi_{2}$. For instance, if $\pi_{1} = \{\{1, 2\}, \{3, 4\}\}$ and $\pi_{2} = \{\{5, 6\}\}$, then the corresponding M\"obius partition has non-singletons $\frakm_{\ge 2} = \{\{1, 2\}, \{3, 4\}, \{5, 6\}\}$, together with the singleton element $\frakm_{1} = \{7\}$. The equality pattern $\tau$ places $\{1, 2\}$ and $\{3, 4\}$ in the same element, so that they share the same final index and induce $\{1, 2\} \cup \{3, 4\} = \{1, 2, 3, 4\} = B_{1}$. The set $\{5, 6\}$ forms another element of $\tau$ and induces $B_{2}$. The coefficient associated with $B_1$ is therefore
\begin{align*}
\sum_{\pi_{1} \in \Pi_{\ge 2} (B_{1})} (-1)^{|\pi_{1}|} \prod_{A \in \pi_{1}} (|A|-1) ! = - (4 - 1) ! + 3 \{(2 - 1) ! (2 - 1) !\} = - 3.
\end{align*}
The coefficient associated with $B_{2}$ is
\begin{equation*}
\sum_{\pi_{2} \in \Pi_{\ge 2} (B_{2})} (-1)^{|\pi_{2}|} \prod_{A \in \pi_{2}} (|A| - 1) ! = - (2 - 1) ! = -1.
\end{equation*}
Hence the merged coefficient for this retained collection is $ C(\calB) = (-3)(-1) = 3 $,
which agrees with the general formula
\begin{equation*}
C (\calB) = (-1)^{|\calB|}
\prod_{B \in \calB} (|B| - 1)
= (-1)^{2} (4 - 1) (2 - 1) = 3.
\end{equation*}
This example concerns a fixed collection $\calB$; the full M\"obius-inversion expansion sums over all $\calB \in \bbB_{\iota}$.
\end{remark}

\subsection{M\"obius inversion decomposition of \texorpdfstring{$\hat{\IIFF}_{5, 5, k} (\hat{\Omega})$}{}}
\label{app:fifth-order-expansion}

We provide the explicit expansion for the case $j = 5$ referred to in Remark~\ref{rem:decomposition}. When $j = 5$, we have $\iota = 3$. The possible element families are $\emptyset, \{\{1, 2\}\}, \{\{1, 3\}\}, \{\{2, 3\}\}, \{\{1, 2, 3\}\}$. Lemma~\ref{lem:mobius} gives
\begin{align*}
\hat{\IIFF}_{5, 5, k} (\hat{\Omega})
= & \ \frac{1}{(n - 2) (n - 3) (n - 4)}
\bbU_{n, 2} (A_{1} X_{1}^{\top} \hat{\Omega} \calR_{1 2}^{3} X_{2} Y_{2}) \\
& - \frac{1}{(n - 3) (n - 4)}
\bbU_{n, 3} ( A_{1} X_{1}^{\top} \hat{\Omega}
( H_{3}^{2} \calR_{1 2} + H_{3} \calR_{1 2} H_{3}
+ \calR_{1 2} H_{3}^{2} + 2 H_{3}^{3} ) X_{2} Y_{2} ).
\end{align*}
By elementary algebra, we have the following:
\begin{align*}
\calR_{1 2}^{3} & = (
X_{1} X_{1}^{\top} \hat{\Omega}
+ X_{2} X_{2}^{\top} \hat{\Omega}
- 2 I )^{3}  =
X_{1} X_{1}^{\top} \hat{\Omega}
X_{1} X_{1}^{\top} \hat{\Omega}
X_{1} X_{1}^{\top} \hat{\Omega}
+ X_{1} X_{1}^{\top} \hat{\Omega}
X_{1} X_{1}^{\top} \hat{\Omega}
X_{2} X_{2}^{\top} \hat{\Omega} \\
& \quad + X_{1} X_{1}^{\top} \hat{\Omega}
X_{2} X_{2}^{\top} \hat{\Omega}
X_{1} X_{1}^{\top} \hat{\Omega}
+ X_{1} X_{1}^{\top} \hat{\Omega}
X_{2} X_{2}^{\top} \hat{\Omega}
X_{2} X_{2}^{\top} \hat{\Omega} 
+ X_{2} X_{2}^{\top} \hat{\Omega}
X_{1} X_{1}^{\top} \hat{\Omega}
X_{1} X_{1}^{\top} \hat{\Omega} \\
& \quad + X_{2} X_{2}^{\top} \hat{\Omega}
X_{1} X_{1}^{\top} \hat{\Omega}
X_{2} X_{2}^{\top} \hat{\Omega} 
+ X_{2} X_{2}^{\top} \hat{\Omega}
X_{2} X_{2}^{\top} \hat{\Omega}
X_{1} X_{1}^{\top} \hat{\Omega}
+ X_{2} X_{2}^{\top} \hat{\Omega}
X_{2} X_{2}^{\top} \hat{\Omega}
X_{2} X_{2}^{\top} \hat{\Omega} \\
& \quad - 6 X_{1} X_{1}^{\top} \hat{\Omega}
X_{1} X_{1}^{\top} \hat{\Omega}
- 6 X_{1} X_{1}^{\top} \hat{\Omega}
X_{2} X_{2}^{\top} \hat{\Omega}  -
6 X_{2} X_{2}^{\top} \hat{\Omega}
X_{1} X_{1}^{\top} \hat{\Omega}\\
& \quad - 6 X_{2} X_{2}^{\top} \hat{\Omega}
X_{2} X_{2}^{\top} \hat{\Omega} 
+ 12 X_{1} X_{1}^{\top} \hat{\Omega}
+ 12 X_{2} X_{2}^{\top} \hat{\Omega}
- 8 I,
\end{align*}
and
\begin{align*}
& \ H_{3}^{2} \calR_{1 2} + H_{3} \calR_{1 2} H_{3} + \calR_{1 2} H_{3}^{2}
+ 2 H_{3}^{3} \\
= &  ( X_{3} X_{3}^{\top} \hat{\Omega}
- I )^{2} ( X_{1} X_{1}^{\top} \hat{\Omega} + X_{2} X_{2}^{\top} \hat{\Omega}
- 2 I ) \\
&  + ( X_{3} X_{3}^{\top} \hat{\Omega}
- I ) ( X_{1} X_{1}^{\top} \hat{\Omega} + X_{2} X_{2}^{\top} \hat{\Omega}
- 2 I ) ( X_{3} X_{3}^{\top} \hat{\Omega} - I ) \\
& + ( X_{1} X_{1}^{\top} \hat{\Omega}
+ X_{2} X_{2}^{\top} \hat{\Omega}
- 2 I ) ( X_{3} X_{3}^{\top} \hat{\Omega} - I )^{2} + 2 ( X_{3} X_{3}^{\top} \hat{\Omega} - I )^{3} \\
& = X_{3} X_{3}^{\top} \hat{\Omega}
X_{3} X_{3}^{\top} \hat{\Omega}
X_{1} X_{1}^{\top} \hat{\Omega}
+ X_{3} X_{3}^{\top} \hat{\Omega}
X_{3} X_{3}^{\top} \hat{\Omega}
X_{2} X_{2}^{\top} \hat{\Omega} 
+ X_{3} X_{3}^{\top} \hat{\Omega}
X_{1} X_{1}^{\top} \hat{\Omega}
X_{3} X_{3}^{\top} \hat{\Omega}\\
& + X_{3} X_{3}^{\top} \hat{\Omega}
X_{2} X_{2}^{\top} \hat{\Omega}
X_{3} X_{3}^{\top} \hat{\Omega} 
+ X_{1} X_{1}^{\top} \hat{\Omega}
X_{3} X_{3}^{\top} \hat{\Omega}
X_{3} X_{3}^{\top} \hat{\Omega}
+ X_{2} X_{2}^{\top} \hat{\Omega}
X_{3} X_{3}^{\top} \hat{\Omega}
X_{3} X_{3}^{\top} \hat{\Omega} \\
& + 2 X_{3} X_{3}^{\top} \hat{\Omega}
X_{3} X_{3}^{\top} \hat{\Omega}
X_{3} X_{3}^{\top} \hat{\Omega} 
- 3 X_{3} X_{3}^{\top} \hat{\Omega}
X_{1} X_{1}^{\top} \hat{\Omega}
- 3 X_{3} X_{3}^{\top} \hat{\Omega}
X_{2} X_{2}^{\top} \hat{\Omega} \\
& - 3 X_{1} X_{1}^{\top} \hat{\Omega}
X_{3} X_{3}^{\top} \hat{\Omega}
- 3 X_{2} X_{2}^{\top} \hat{\Omega}
X_{3} X_{3}^{\top} \hat{\Omega}
- 12 X_{3} X_{3}^{\top} \hat{\Omega}
X_{3} X_{3}^{\top} \hat{\Omega} \\
& + 3 X_{1} X_{1}^{\top} \hat{\Omega}
+ 3 X_{2} X_{2}^{\top} \hat{\Omega}
+ 18 X_{3} X_{3}^{\top} \hat{\Omega} - 8 I.
\end{align*}
Hence,
\begin{align*}
& \ \hat{\IIFF}_{5, 5, k} (\hat{\Omega}) \\
={}& \frac{1}{(n - 2) (n - 3) (n - 4)}
\Big\{ \bbU_{n, 2} ( A_{1}
(X_{1}^{\top} \hat{\Omega} X_{1})^{3}
X_{1}^{\top} \hat{\Omega} X_{2}
Y_{2} ) + \bbU_{n, 2}
( A_{1} X_{1}^{\top} \hat{\Omega} X_{2}
(X_{2}^{\top} \hat{\Omega} X_{2})^{3}
Y_{2} ) \Big\} \\
& + \frac{1}{(n - 2) (n - 3) (n - 4)}
\Bigg\{\begin{aligned}
& \bbU_{n, 2} ( A_{1} (X_{1}^{\top} \hat{\Omega} X_{1})^{2}
X_{1}^{\top} \hat{\Omega} X_{2}
X_{2}^{\top} \hat{\Omega} X_{2}
Y_{2} ) \\
& + \bbU_{n, 2} (
A_{1} X_{1}^{\top} \hat{\Omega} X_{1}
X_{1}^{\top} \hat{\Omega} X_{2}
(X_{2}^{\top} \hat{\Omega} X_{2})^{2}
Y_{2} ) \end{aligned}
\Bigg\} \\
& + \frac{2}{(n - 2) (n - 3) (n - 4)}
\Big\{ \bbU_{n, 2}
( A_{1} X_{1}^{\top} \hat{\Omega} X_{1}
(X_{1}^{\top} \hat{\Omega} X_{2})^{3}
Y_{2} ) + \bbU_{n, 2}
( A_{1} (X_{1}^{\top} \hat{\Omega} X_{2})^{3} X_{2}^{\top} \hat{\Omega} X_{2}
Y_{2} ) \Big\} \\
& - \frac{1}{(n - 3) (n - 4)}
\Big\{ \bbU_{n, 3}
( A_{1} (X_{1}^{\top} \hat{\Omega} X_{3})^{2} X_{3}^{\top} \hat{\Omega} X_{3}
X_{1}^{\top} \hat{\Omega} X_{2}
Y_{2} ) + \bbU_{n, 3}
( A_{1} X_{1}^{\top} \hat{\Omega} X_{2}
(X_{2}^{\top} \hat{\Omega} X_{3})^{2}
X_{3}^{\top} \hat{\Omega} X_{3}
Y_{2} ) \Big\}\\
& - \frac{1}{(n - 3) (n - 4)}
\Bigg\{ \begin{aligned}
& \bbU_{n, 3} ( A_{1} X_{1}^{\top} \hat{\Omega} X_{3}
X_{3}^{\top} \hat{\Omega} X_{3}
X_{3}^{\top} \hat{\Omega} X_{2}
X_{2}^{\top} \hat{\Omega} X_{2}
Y_{2} ) \\
& + \bbU_{n, 3} ( A_{1}
X_{1}^{\top} \hat{\Omega} X_{1}
X_{1}^{\top} \hat{\Omega} X_{3}
X_{3}^{\top} \hat{\Omega} X_{3}
X_{3}^{\top} \hat{\Omega} X_{2}
Y_{2} ) \end{aligned}
\Bigg\} \\
& - \frac{1}{(n - 3) (n - 4)}
\Big\{ \bbU_{n, 3} ( A_{1}
(X_{1}^{\top} \hat{\Omega} X_{3})^{3}
X_{3}^{\top} \hat{\Omega} X_{2}
Y_{2} ) + \bbU_{n, 3}
( A_{1} X_{1}^{\top} \hat{\Omega} X_{3}
(X_{3}^{\top} \hat{\Omega} X_{2})^{3}
Y_{2} ) \Big\} \\
& - \frac{2}{(n - 3) (n - 4)}
\bbU_{n, 3} ( A_{1} X_{1}^{\top} \hat{\Omega} X_{3} (X_{3}^{\top} \hat{\Omega} X_{3})^{2} X_{3}^{\top} \hat{\Omega} X_{2}
Y_{2} ) \\
& - \frac{12}{(n - 2) (n - 3) (n - 4)}
\Bigg\{ \begin{aligned}
& \bbU_{n, 2} ( A_{1} (X_{1}^{\top} \hat{\Omega} X_{1})^{2}
X_{1}^{\top} \hat{\Omega} X_{2}
Y_{2} ) + \bbU_{n, 2} ( A_{1}
X_{1}^{\top} \hat{\Omega} X_{2}
(X_{2}^{\top} \hat{\Omega} X_{2})^{2}
Y_{2} ) \\
& + \bbU_{n, 2} ( A_{1}
(X_{1}^{\top} \hat{\Omega} X_{2})^{3}
Y_{2} ) + \bbU_{n, 2} (
A_{1} X_{1}^{\top} \hat{\Omega} X_{1}
X_{1}^{\top} \hat{\Omega} X_{2}
X_{2}^{\top} \hat{\Omega} X_{2}
Y_{2} ) \end{aligned}
\Bigg\} \\
& + \frac{12}{(n - 3) (n - 4)}
\bbU_{n, 3} ( A_{1} X_{1}^{\top} \hat{\Omega} X_{3}
X_{3}^{\top} \hat{\Omega} X_{3}
X_{3}^{\top} \hat{\Omega} X_{2}
Y_{2} ) \\
& + \frac{3 (n + 12)}{(n - 2) (n - 3) (n - 4)}
\Big\{ \bbU_{n, 2} ( A_{1} X_{1}^{\top} \hat{\Omega} X_{1} X_{1}^{\top} \hat{\Omega} X_{2} Y_{2} ) + \bbU_{n, 2}
( A_{1} X_{1}^{\top} \hat{\Omega} X_{2}
X_{2}^{\top} \hat{\Omega} X_{2}
Y_{2} ) \Big\} \\
& - \frac{2 (5 n + 12)}{(n - 2) (n - 3) (n - 4)}
\underbrace{ \bbU_{n, 2}
( A_{1} X_{1}^{\top} \hat{\Omega} X_{2} Y_{2} ) }_{\equiv \, \hat{\IIFF}_{2, 2, k} (\hat{\Omega})}.
\end{align*}
Similarly, $\hat{\IIFF}_{5, 5, k} (\hat{\Omega})$ matches the dominating bias terms of $\hat{\psi}_{4, k} (\hat{\Omega})
= \hat{\IIFF}_{2, 2, k} (\hat{\Omega})
+ \hat{\IIFF}_{3, 3, k} (\hat{\Omega})
+ \hat{\IIFF}_{4, 4, k} (\hat{\Omega})$,
except that $\Omega$ is replaced by $\hat{\Omega}$. It is also straightforward to see that the monomials of the leverage scores up to degree $4$ appear in
$\hat{\IIFF}_{5, 5, k} (\hat{\Omega})$.

\subsection{Proof of Lemma~\ref{lem:graph-counting}}
\label{app:graph-counting}
We first prove the assertion for $D = 0$. Let $G \coloneqq \Gamma_{0}$. Let $\calX$ denote the support of $X$, and let $C_{K} < \infty$ be such that $\sup_{x \in \calX} \|x\|_{2}^{2} \leq C_{K} k$ by Assumption~\ref{as:cov}.
Let $C_{\Pi}$ denote the operator-norm constant in Assumption~\ref{as:kernel-stability}, and let $C = \max\{C_{K}, C_{\Pi}, 1\}$.



Define 
$K (x, x') \coloneqq x^{\top} x'$, for any $x, x' \in \calX$. By Assumptions \ref{as:cov} and Cauchy-Schwarz inequality, 
\begin{equation}
\label{eq:kernel-sup-bound}
|K (x, y)| = |x^{\top} y| \le \|x\|_{2} \|y\|_{2} \leq C_{K} k.
\end{equation}
Choose a spanning forest $F$ of $G$, that is, one spanning tree inside each connected component of $G$. Self-loops are not included in $F$. Hence, $e (F) = v (G) -\kappa (G)$. Let $E_{\nt} \coloneqq E (G) \setminus E (F)$ be the set of non-tree edges, namely the edges not selected in the spanning forest $F$. This set contains all self-loops. Moreover, if several parallel edges have the same pair of endpoints, the forest $F$ can contain at most one of them; otherwise $F$ would contain a cycle.
All unselected parallel copies are therefore included in $E_{\nt}$. Hence $|E_{\nt}| = e (G) -v (G) + \kappa (G)  = \frakr (G)$. We first separate the non-tree edges. Define $\Psi_{\nt} (\{x_{w}\}_{w \in V}) \coloneqq \prod_{(u, v) \in E_{\nt}} K (x_{u}, x_{v})$. If $(u, v)$ is a self-loop, then the corresponding factor is $K (x_{u}, x_{u})$. By \eqref{eq:kernel-sup-bound},
\begin{equation}
\label{eq:psi-nt-bound}
\| \Psi_{\nt} \|_{\infty} \le (C_{K} k)^{|E_{\nt}|} = (C_{K} k)^{\frakr (G)}.
\end{equation}

The bound in \eqref{eq:psi-nt-bound} controls all non-tree edges by a uniform estimate. After this step, the original product over $E (G)$ is reduced to $\bbE \Big\{ \Psi_{\nt}(\{X_{w}\}_{w\in V}) \prod_{(u, v)\in E (F)} K (X_{u}, X_{v}) \Big\}$.
It remains to control the contribution of the forest edges. These terms are treated differently from the non-tree contribution since $F$ is acyclic, its vertices can be integrated out one leaf at a time. Each leaf integration applies the  integral operator $\Pi$ to the current bounded function and introduces only the constant $C_{\Pi}$ in the bound. We prove the following auxiliary bound.

\begin{lemma}
\label{lem:forest}
For every forest $F$ on a finite vertex set $V$ and every bounded measurable function $\Psi (\{x_{w}\}_{w \in V})$,
\begin{equation}
\label{eq:forest-bound}
\Big| \bbE \Big\{ \Psi (\{X_{w}\}_{w \in V}) \prod_{(u, v) \in E (F)} K (X_{u}, X_{v}) \Big\} \Big| \le C_{\Pi}^{e (F)} \cdot \|\Psi\|_{\infty},
\end{equation}
where we recall the definition of $C_{\Pi}$ in Assumption~\ref{as:kernel-stability}.
\end{lemma}

We now apply Lemma~\ref{lem:forest} with $\Psi$ set to $\Psi_{\nt}$. Using \eqref{eq:psi-nt-bound}, we obtain
\begin{align*}
\Big| \bbE \prod_{(u, v) \in E (G)} K (X_{u}, X_{v}) \Big| & = \Big| \bbE \Big\{ \Psi_{\nt} (\{X_{w}\}_{w \in V}) \prod_{(u, v) \in E (F)} K (X_{u}, X_{v}) \Big\} \Big| \\
& \le C_{\Pi}^{e (F)} \|\Psi_{\nt}\|_{\infty} \le C_{\Pi}^{v (G) - \kappa (G)} (C_{K} k)^{\frakr (G)}.
\end{align*}
Note that $e (G) = e (F) + \frakr (G) = v (G) - \kappa (G) + \frakr (G)$ and $C = \max \{C_{K}, C_{\Pi}, 1\}$. Therefore, we have $|\bbE \prod_{(u, v) \in E (G)} K (X_{u}, X_{v})| \le C^{e (G)} k^{\frakr (G)}$, and in turn $\Big| \bbE \prod_{(u, v) \in E (G)} X_{u}^{\top} X_{v} \Big| \le   C^{e (G)} k^{\frakr (G)}$. If $G$ is connected, then $\kappa (G) = 1$ and $\frakr (G) = e(G) + 1 - v(G)$, which gives the desired result. 
Finally, we are left to prove Lemma~\ref{lem:forest}.

\begin{proof}[Proof of Lemma~\ref{lem:forest}]
We prove \eqref{eq:forest-bound} by induction on $e (F)$. If $e (F) = 0$, then $ |\bbE \Psi (\{X_{w}\}_{w \in V})| \le \|\Psi\|_{\infty}$, so the claim holds. Now assume $e (F) \ge 1$. 

Choose a leaf vertex $v$ of forest $F$, and let $u$ be its unique neighbor. Let $ V' \coloneqq V \setminus \{v\}$, and let $F'$ be the forest on $V'$ obtained by deleting the vertex $v$ and the edge $(u, v)$. Then $e (F') = e (F)-1$.

For fixed values $\bm{x}_{V'} \coloneqq \{x_{w}\}_{w\in V'}$, define a function of one variable as follows: $h_{\bm{x}_{V'}} (t) \coloneqq \Psi (\{x_{w}\}_{w \in V'}, t)$, where $t$ is assigned to the leaf vertex $v$, while $x_{w}$ is assigned to each vertex $w \in V'$. Then $\|h_{\bm{x}_{V'} }\|_{\infty} \le \|\Psi\|_{\infty}$. The edge deleted together with $v$ is $(u, v)$. Integrating out the variable at $v$, define $\tilde{\Psi} (\{x_{w}\}_{w \in V'}) \coloneqq \bbE \{K (x_{u}, X) h_{\bm{x}_{V'}}(X)\} = (\Pi h_{\bm{x}_{V'}}) (x_{u})$. Only the variable at $v$ is integrated out; the dependence on the remaining variables is preserved through $h_{\bm{x}_{V'}}$.

By the $L_{\infty}$-stability assumption (Assumption~\ref{as:kernel-stability}), $|\tilde{\Psi} (\bm{x}_{V'})| = |(\Pi h_{\bm{x}_{V'}}) (x_{u})| \le C_{\Pi} \|h_{\bm{x}_{V'}}\|_{\infty} \le C_{\Pi} \|\Psi\|_{\infty}$. Therefore, $\|\tilde{\Psi}\|_{\infty} \le C_{\Pi} \|\Psi\|_{\infty}$. Conditioning on the variables $\{X_{w} : w \in V'\}$ and using the independence of $X_{v}$ from these variables, we obtain
\begin{align*}
\bbE \Big\{ \Psi (\{ X_{w} \}_{w \in V}) \prod_{(a, b)\in E (F)} K (X_{a}, X_{b}) \Big\} = \bbE \Big\{ \tilde{\Psi} (\{ X_{w} \}_{w \in V'}) \prod_{(a, b) \in E (F')} K (X_{a}, X_{b}) \Big\}.
\end{align*}
Applying the induction hypothesis to $F'$ and $\tilde{\Psi}$ yields the following inequalities, proving~\eqref{eq:forest-bound}.
\begin{align*}
\Big| \bbE \Big\{ \Psi (\{X_{w}\}_{w \in V}) \prod_{(a, b) \in E (F)} K (X_{a}, X_{b}) \Big\} \Big| \le C_{\Pi}^{e(F')} \|\tilde{\Psi}\|_{\infty} \le C_{\Pi}^{e(F') + 1} \|\Psi\|_{\infty} = C_{\Pi}^{e(F)} \|\Psi\|_{\infty}.
\end{align*}
\end{proof}
The preceding argument proves the case $D = 0$ in Lemma~\ref{lem:graph-counting}. The case $D \geq 1$ is proved next.

\begin{proof}[Proof of the case $D \geq 1$ in Lemma~\ref{lem:graph-counting}]
Write $\Delta_{n} = U_{\calA} + V_{\calA}$, where $U_{\calA} \coloneqq \frac{1}{n} \sum_{i\in\calA} (I - X_{i} X_{i}^{\top})$ and $V_{\calA} \coloneqq \frac{1}{n} \sum_{i\notin\calA} (I -X_{i} X_{i}^{\top})$. Since the multiplicity of $\Delta_{n}$ is $D$, expanding $\Delta_{n}$ by $U_{\calA}$ and $V_{\calA}$ leads to a sum of terms, each of which contains $D$ multiplications of $U_{\calA}$ or $V_{\calA}$. Given such a summand, let $r_{A}$ and $r_{F}$ be the multiplicities of $U_{\calA}$ and $V_{\calA}$, respectively; $r_{A} + r_{F} \equiv D$.

For $U_{\calA}$, the following holds:
\begin{equation*}
U_{\calA} = \frac{s}{n} I - \frac{1}{n} \sum_{i\in\calA} X_{i} X_{i}^{\top}.
\end{equation*}
In the fixed summand under consideration, the multiplicity of $U_{\calA}$ is $r_{A}$. 
Now suppose that the term $ - \frac{1}{n} \sum_{i \in \calA} X_{i} X_{i}^{\top}$
has multiplicity $q_{A}$, while the term $\frac{s}{n} I$ has multiplicity $r_{A} - q_{A}$.

The scalar coefficients produced by these $r_A$ positions are bounded in absolute value by
\begin{equation*}
\Big( \frac{s}{n} \Big)^{r_{A} - q_{A}} \Big( \frac{1}{n} \Big)^{q_{A}} s^{q_{A}} \le \Big( \frac{s}{n} \Big)^{r_{A}}.
\end{equation*}
Indeed, each occurrence of $\frac{s}{n} I$ contributes the factor $s / n$, while each occurrence of $- n^{-1} \sum \limits_{i \in \calA} X_{i} X_{i}^{\top}$ contributes the factor $1 / n$ together with a finite summation over $\calA$, whose cardinality is $s$.

We next control the increase in the first Betti number caused by the $q_{A}$ positions where the matrix $X_{i} X_{i}^{\top}$ from $ - \frac{1}{n} \sum_{i \in \calA} X_{i} X_{i}^{\top}$ is selected. For each such position, the index $i$ belongs to $\calA$, and hence is already one of the vertices in the initial graph $\Gamma_{0}$. Therefore, no new vertex outside the initial path is introduced. At the graph level, inserting $X_{i} X_{i}^{\top}$ in the middle of an existing inner product of the form $X_{j}^{\top} X_{k}$ replaces the original edge $j \mbox{--} k$ by edges $i \mbox{--} j$ and $i \mbox{--} k$. Thus the number of edges increases by at most one, the number of vertices remains unchanged (because of the definition of $\calA$), and the number of connected components does not increase. Hence, these $q_{A}$ positions increase the first Betti number by at most $q_{A}$.

For $V_{\calA}$, let $u_{1}, \cdots, u_{r_{F}}$ be the corresponding sample indices, and let $\pi$ be the partition of $\{1, \cdots, r_{F}\}$ induced by the values of the indices. For example, if $u_{1} = u_{2}$, then positions $1$ and $2$ belong to the same element of $\pi$. Define $b \coloneqq |\pi|$, $d_{F} \coloneqq r_{F} - b$. If $\pi$ has a singleton element with index $u \notin \calA$,
then $X_{u}$ appears only in one factor $I - X_{u} X_{u}^{\top}$. Conditioning on all other samples, the corresponding expectation is zero because $\bbE (I - X_{u} X_{u}^{\top}) = 0$. Therefore, every element in the partition with nonzero expectation must have size at least two, and thus $b \leq \lfloor \frac{r_{F}}{2} \rfloor$ and $d_{F} \geq \lceil \frac{r_{F}}{2} \rceil$.

For a given partition $\pi$, the normalization factor $n^{- r_{F}}$ and the summation over its $b$ distinct indices contribute at most $n^{- r_{F}} n^{b} = n^{- d_{F}}$. When $r_{F} = 0$, there is one empty partition. When $r_{F} = 1$, the only index $u \notin \calA$
appears in a single factor $I - X_{u} X_{u}^{\top}$.
Conditioning on all other sample variables, the corresponding expectation is zero because $\bbE (I - X_{u} X_{u}^{\top}) = 0$.
For $r_{F} \geq 2$, partitions $\pi$ with no singleton
elements satisfy $d_{F} \geq 1$, and we have
\begin{align*}
|\{\pi : \ r_{F} - |\pi| = d_{F}\}| \le \binom{r_{F}}{d_{F}} (r_{F} - d_{F})^{d_{F}} \le \Big( \frac{e r_{F}}{d_{F}} \Big)^{d_{F}} (r_{F} - d_{F})^{d_{F}} \leq (C r_{F})^{d_{F}},
\end{align*}
where the second inequality follows from $\binom{r_{F}}{d_{F}} \le r_{F}^{d_{F}} / d_{F} ! \le \Big( \frac{e r_{F}}{d_{F}} \Big)^{d_{F}}$. The last inequality uses $r_{F} - d_{F} \leq d_{F}$,
which follows from $r_{F} \geq 2 |\pi| = 2 (r_{F} - d_{F})$.

Next, write $W_{i} = I - X_{i} X_{i}^{\top}$. For a given partition $\pi$, we further specify, at each position where a matrix $W_{i}$ appears, whether the term $I$ or the term $- X_{i} X_{i}^{\top}$ is selected. The resulting product of inner products defines an undirected graph $\Gamma$: its vertices are the sample indices appearing in the inner products, and each inner product $X_{a}^{\top} X_{b}$ defines an edge $(a, b)$.

For each element $B_{\nu} \in \pi$ in the partition (by definition, sharing the same sample indices; denote it by $u_{B_{\nu}}$), set $b_{\nu} \coloneqq |B_{\nu}|$. Suppose that $t_{\nu}$ rank-one matrices $X X^{\top}$ are selected from $B_{\nu}$. Then $t_{\nu} \le b_{\nu}$. The contribution from $B_{\nu}$ increases the first Betti number by at most $\max \{t_{\nu} - 1, 0\}$. Indeed, if $t_{\nu} = 0$, no inner product involving $u_{B_\nu}$ is introduced and thus there is no increase in the first Betti number. If $t_{\nu} \ge 1$, the first use of $-X_{u_{B_{\nu}}} X_{u_{B_{\nu}}}^{\top}$ introduces the sample index $u_{B_{\nu}}$ into the initial displayed path. 
At the graph level, this insertion replaces one edge by two consecutive edges and introduces a new vertex:
\begin{equation*}
\begin{tikzpicture}[baseline=(current bounding box.center), line width=0.8pt]
  \node[circle, draw, inner sep=1.2pt] (a1) at (0,0) {$a$};
  \node[circle, draw, inner sep=1.2pt] (b1) at (1.8,0) {$b$};
  \draw (a1) -- (b1);

  \node at (3.1,0) {$\longrightarrow$};

  \node[circle, draw, inner sep=1.2pt] (a2) at (4.4,0) {$a$};
  \node[circle, draw, inner sep=1.2pt] (u)  at (6.0,0) {$u_{B_{\nu}}$};
  \node[circle, draw, inner sep=1.2pt] (b2) at (7.6,0) {$b$};
  \draw (a2) -- (u);
  \draw (u) -- (b2);
\end{tikzpicture}
\end{equation*}
Thus, both the edge number and vertex number increase by one, while the number of connected components remains unchanged. Hence the first Betti number $\frakr (\Gamma) = e (\Gamma) - v (\Gamma) + \kappa (\Gamma)$ remains.

Each of the remaining $t_{\nu} - 1$ selections of $- X_{u_{B_{\nu}}} X_{u_{B_{\nu}}}^{\top}$ uses the same index $u_{B_{\nu}}$ again. It can therefore add at most one edge without adding a new vertex, and hence can increase $\frakr (\Gamma)$ by at most one.
Consequently, the increase in first Betti number caused by $B_{\nu}$ is at most $\max \{t_{\nu} - 1, 0\} \leq b_{\nu} - 1$. Summing over all possible elements of $\pi$, $V_{\calA}$ increases the first Betti number by at most $\sum_{B_{\nu} \in \pi} (b_{\nu} - 1) = r_{F} - |\pi| = r_{F} - b = d_{F}$. Together with the previous analysis of $U_{\calA}$, the undirected graph $\Gamma$ associated with the resulting product of inner products satisfies $\frakr (\Gamma) = e (\Gamma) - v (\Gamma) + \kappa (\Gamma) \leq \frakr (\Gamma_{0}) + q_{A} + d_{F}$. Moreover, every selected factor $X_{u} X_{u}^{\top}$ adds at most one edge, so that $e(\Gamma) \leq e(\Gamma_{0}) + D$.

For a fixed partition $\pi$ and a fixed selection of
$I$ or $- X_{i} X_{i}^{\top}$ at each position, let $\calT_{\pi}$ denote the corresponding contribution to $\bbE K_{\bm{\ell}}$. Applying the already proved $D = 0$ case to the expanded graph $\Gamma$, together with the preceding coefficient bounds, gives $|\calT_{\pi}| \lesssim C^{e(\Gamma_{0}) + D} k^{\frakr(\Gamma_{0})} \Big( \frac{C s}{n} \Big)^{r_{A}} n^{-d_{F}} k^{q_{A} + d_{F}}$.

We now sum over all possible partitions of the indices generated by $V_{\calA}$ satisfying $r_{F} - |\pi| = d_{F}$. Using the preceding
counting bound, together with $q_{A} \leq r_{A}$ and $r_{F} \leq D$, yields
\begin{align*}
\sum_{\pi: r_{F} - |\pi| = d_{F}} |\calT_{\pi}| & \lesssim C^{e(\Gamma_{0}) + D} k^{\frakr(\Gamma_{0})} \Big( \frac{C s}{n} \Big)^{r_{A}} n^{- d_{F}} (C r_{F})^{d_{F}} k^{q_{A} + d_{F}} \\
& \lesssim C^{e(\Gamma_{0}) + D} k^{\frakr(\Gamma_{0})} \Big( \frac{C (s \vee D) k}{n} \Big)^{r_{A} + d_{F}} .
\end{align*}

If $r_{F} = 1$, the contribution is zero by centering. If $r_{F} = 0$, then $r_{A} = D$ and hence $r_{A} + d_{F} \geq \lceil D/2 \rceil$. If $r_{F} \geq 2$, then $d_{F} \geq \lceil r_{F} / 2 \rceil$ and $r_{A} + r_{F} = D$, so $r_{A} + d_{F} \geq r_{A} + \Big\lceil \frac{r_{F}}{2} \Big\rceil \geq \Big\lceil \frac{D}{2} \Big\rceil$.
Therefore, using $C (s \vee D) k/n \leq \eta < 1$ and summing over the possible values of $d_{F}$, $\sum_{\pi} |\calT_{\pi}| \lesssim  \Big( \frac{C (s \vee D) k}{n} \Big)^{\lceil D/2 \rceil}$.

It remains to sum over the remaining choices not yet included. The $D$ occurrences of $\Delta_{n}$ in the product can be assigned to $U_{\calA}$ or $V_{\calA}$ in at most $2^{D}$ ways, and substituting $W_{i} = I - X_{i} X_{i}^{\top}$ for each matrix $W_{i}$ produces at most another $2^{D}$ terms. The summation over partitions $\pi$ of the indices associated with $V_{\calA}$ has already been counted through the factor $(C r_{F})^{d_{F}}$. Hence, the remaining summation contributes at most a factor $C^{D}$. This factor, together with the factor $C^{D}$ from $e(\Gamma)\leq e(\Gamma_{0}) + D$, can be absorbed by enlarging the constant $C$ appropriately. Therefore, $|\bbE K_{\bm{\ell}}| \lesssim C^{e(\Gamma_{0})} k^{\frakr(\Gamma_{0})} \Big\{ \frac{C (s \vee D) k}{n} \Big\}^{\lceil D/2 \rceil}.$ This completes the proof.
\end{proof}

\subsection{Control of the remainder in the Neumann series expansion}
\label{app:random-neumann-remainder}

Recall that $ \hat{\Omega}_{J} = I + \sfD_{J} = \sum_{\ell = 0}^{J} \Delta_{n}^{\ell}$, $\sfR_{J} \coloneqq \Delta_{n}^{J + 1} \hat{\Omega}$, $\hat{\Omega} = \hat{\Omega}_{J} + \sfR_{J}$. For $2 \leq r \leq m$, let
\begin{align}
\hat{\IIFF}_{r, r, k}^{(J)} (\Omega)\coloneqq (-1)^{r} \bbU_{n, r}
\Bigg\{ A_{1} X_{1}^{\top} \hat{\Omega}_{J} \prod_{s = 3}^{r}
\Big( X_{s} X_{s}^{\top} \hat{\Omega}_{J} - I \Big) X_{2} Y_{2} \Bigg\},
\end{align}
where the product is interpreted as the identity matrix when $r = 2$, and set $\chi_{m, k} (\Omega) \coloneqq  \hat{\psi}_{m, k}^{(J)} = \sum_{r = 2}^{m} \hat{\IIFF}_{r, r, k}^{(J)} (\Omega)$. Consequently, $\hat{\psi}_{m, k} (\hat{\Omega})
= \hat{\psi}_{m, k}^{(J)} + \hat{\calR}_{m, k, J}$.
Using Lemma~\ref{lem:hockey-stick} and the same summation argument as in the proof of Lemma~\ref{lem:bias-binomial-representation}, we obtain the explicit representation
\begin{align}
& \hat{\calR}_{m, k, J} \coloneqq \label{eq:finite-neumann-sample-remainder-expansion} \\
& \sum_{j = 1}^{m - 1} (-1)^{j + 1} \binom{m - 1}{j} \sum_{\emptyset \neq S \subseteq [j - 1] \cup \{0\}} \sum_{\emptyset \neq T \subseteq S} \bbU_{n, j + 1} \Big\{ A_{m - 1} X_{m - 1}^{\top} \prod_{s = 0}^{j - 1} \Big( X_{s} X_{s}^{\top} \sfD_{J}^{\mathbf 1 \{s \in S \setminus T\}} \sfR_{J}^{\mathbf 1 \{s \in T\}} \Big) X_{m} Y_{m} \Big\}. \nonumber
\end{align}
The product is taken in increasing order of $s$ with the convention $X_{0} X_{0}^{\top} = I$. Since $T \neq \emptyset$, every summand in \eqref{eq:finite-neumann-sample-remainder-expansion} contains at least one occurrence of $\sfR_{J}$.
\begin{lemma}
\label{lem:random-neumann-remainder}
Let $J = \lceil C_{0} \log n \rceil$ for a sufficiently large fixed constant $C_{0}$. Suppose that $m \asymp \log n$ and $k \lesssim n / \log^{3} n$. Then the remainder in \eqref{eq:finite-neumann-sample-remainder-expansion} satisfies $\hat{\calR}_{m, k, J} = o_{\bbP} (n^{- 1 / 2})$.
\end{lemma}
\begin{proof}
By the matrix Bernstein inequality
(Lemma~\ref{lem:matrix Bernstein} in Appendix~\ref{app:concentration}), there exists a deterministic sequence $r_{n}$ such that $r_{n} \lesssim \Big( \frac{k \log n}{n} \Big)^{1 / 2}$ and $\bbP (\calG_{n}^{c}) = o (1)$ for $\calG_{n} \coloneqq \{\|\Delta_{n}\|_{\op} \leq r_{n}\}$. Since $r_{n} \leq 1 / 2$ for all sufficiently large $n$, on $\calG_{n}$ we have $\| \Delta_{n} \|_{\op} \leq r_{n} < 1$, so $\hat{\Sigma}$ is invertible and $\hat{\Omega}\|_{\op} \| (I - \Delta_{n})^{-1} \|_{\op} \leq (1 - r_{n})^{-1} \leq 2$. On $\calG_{n}$, $$\sfR_{J} = \Delta_{n}^{J + 1} \hat{\Omega} = \sum_{\ell = J + 1}^{\infty} \Delta_{n}^{\ell},$$ where the series converges absolutely in operator norm. For $\bm{\ell} = (\ell_{0}, \cdots, \ell_{j - 1}) \in \bbN_{0}^{j}$, let $c (\bm{\ell}) \coloneqq \sum_{s = 0}^{j - 1} \ell_{s}$, and define
\begin{align*}
g_{j, \bm{\ell}} (i_{1}, \cdots, i_{j + 1})
\coloneqq A_{i_{1}} X_{i_{1}}^{\top}
\Delta_{n}^{\ell_{0}} \prod_{s = 1}^{j - 1}
( X_{i_{s + 2}} X_{i_{s + 2}}^{\top}
\Delta_{n}^{\ell_{s}} ) X_{i_{2}} Y_{i_{2}},
\end{align*}
where the product is interpreted as $I$ when $j = 1$. Rearranging all terms in \eqref{eq:finite-neumann-sample-remainder-expansion} gives, on $\calG_{n}$,
\begin{align}
\label{eq:random-remainder-power-expansion}
\hat{\calR}_{m, k, J} = \sum_{j = 1}^{m - 1}
(-1)^{j + 1} \binom{m - 1}{j}
\sum_{\substack{ \bm{\ell} \in \bbN_{0}^{j} \\
\max_{0 \leq s \leq j - 1} \ell_{s} \geq J + 1
}} \bbU_{n, j + 1} (g_{j, \bm{\ell}}).
\end{align}

To control the final remainder in the Neumann series expansion, we first derive a $L_{2}$ bound for the corresponding $U$-statistic. Expanding the squared $L_{2}$ norm gives a sum over two ordered $(j + 1)$-tuples. For each possible number of common indices between these tuples, we apply the argument in the proof of  Lemma~\ref{lem:graph-counting}. Summing over the number of common indices gives the following lemma.
\begin{lemma}
\label{lem:random-remainder-two-chain-bound}
Let $1 \leq j \leq m - 1$ and
$\bm{\ell} = (\ell_{0}, \cdots, \ell_{j - 1}) \in \bbN_{0}^{j}$. Define $c \coloneqq c (\bm{\ell}) = \sum_{s = 0}^{j - 1} \ell_{s}$. For a sufficiently large universal constant $C$, if $\frac{C (j \vee c)k}{n} \leq
\eta < 1$, then
\begin{equation}
\label{eq:random-remainder-two-chain-bound}
\| \bbU_{n, j + 1} (g_{j, \bm{\ell}}) \|_{2} \lesssim \zeta_{A, Y} C^{j} \Big\{ \frac{C (j \vee c)k}{n} \Big\}^{c / 2}.
\end{equation}
\end{lemma}    
\begin{proof}
Let $a = j + 1$. Fix two ordered $a$-tuples having exactly $\alpha$ common sample indices, and let $\calA$ be the union of their displayed indices. Then $|\calA| = 2 a - \alpha \leq 2 (j + 1)$. The product of the two
kernels contains $2c$ occurrences of $\Delta_{n}$. After replacing all of them by $I$, the resulting graph $\Gamma_{\alpha}$ is the union of
two paths, each having $j$ edges. Therefore,    
\begin{align*}
e (\Gamma_{\alpha}) = 2 j, \ v (\Gamma_\alpha) = 2 (j + 1) - \alpha,\ 
\kappa (\Gamma_{\alpha})= \begin{cases}
2,&\alpha = 0,\\ 1,&\alpha \geq 1, \end{cases} , \
\frakr(\Gamma_\alpha) = \max\{\alpha - 1 , 0\}.
\end{align*}
For $\bfi,\bfi' \in \calI_{n, a}$, define
$\alpha (\bfi, \bfi') \coloneqq
| \ind(\bfi) \cap \ind(\bfi')|$. Fix such a pair with overlap $\alpha$. Since $|\calA| \vee 2 c \leq C (j \vee c)$, Lemma~\ref{lem:graph-counting}, by setting $D = 2 c$, gives the power $\{C(j\vee c)k/n\}^{c}$. Considering the factors involving $A$ and $Y$ and applying Lemma~\ref{lem:graph-counting} give
\begin{equation}
\label{eq:random-remainder-fixed-overlap}
|\bbE (g_{j, \bm{\ell}}(\bfi)
g_{j, \bm{\ell}}(\bfi'))| \lesssim C^{2 j} k^{\max\{\alpha - 1, 0\}} \Big\{ \frac{C (j \vee c) k}{n} \Big\}^{c}.
\end{equation}
Expanding the squared $L_{2}$ norm and rearranging yield
\begin{equation*}
\|\bbU_{n, a} (g_{j, \bm{\ell}})\|_{2}^{2} = \sum_{\alpha = 0}^{a} \calV_{\alpha}^{'} ,\ \text{where}\ \calV_{\alpha}^{'} \coloneqq
\frac{\{(n - a) !\}^{2}}{(n !)^{2}}
\sum_{\substack{ \bfi, \bfi' \in \calI_{n, a}\\
\alpha (\bfi, \bfi') = \alpha }}
\bbE[ g_{j, \bm{\ell}} (\bfi)
g_{j, \bm{\ell}} (\bfi') ].
\end{equation*}

Fix the first ordered $a$-tuple $\bfi$. The number of ordered $a$-tuples $\bfi'$ sharing exactly $\alpha$ indices with $\bfi$ is $\binom{a}{\alpha}^{2} \alpha !
\frac{(n - a) !}{(n - 2 a + \alpha) !}$.
Indeed, one first chooses the $\alpha$ shared indices from $\bfi$, then chooses their $\alpha$ positions in $\bfi'$, orders them in those positions, and finally fills the remaining $a-\alpha$ positions with distinct indices outside $\bfi$. Since the number of possible first ordered $a$-tuples is $\frac{n !}{(n - a) !}$, the total number of ordered pairs with overlap size $\alpha$ is $\frac{n !}{(n - a) !}
\binom{a}{\alpha}^{2} \alpha ! \frac{(n - a) !}{(n - 2 a + \alpha) !} = \binom{a}{\alpha}^{2} \alpha ! \frac{n !}{(n - 2 a + \alpha) !}$. The preceding fixed-overlap bound therefore gives
\begin{align*}
|\calV_{\alpha}^{'} | &\lesssim
C^{2 j} \Big\{ \frac{C (j \vee c) k}{n}
\Big\}^{c} k^{\max \{ \alpha - 1, 0\}}
\binom{a}{\alpha}^{2} \alpha !
\frac{\{(n - a) !\}^{2}} {n ! (n - 2 a + \alpha) !}.
\end{align*}
For $a \leq n / 2$, the combinatorial coefficient satisfies
\begin{align*}
\binom{a}{\alpha}^{2} \alpha !
\frac{\{(n - a) !\}^{2}}
{n ! (n - 2 a + \alpha) !} \leq
\frac{a^{2 \alpha}}{\alpha !}
\frac{\{(n - a) !\}^{2}}
{n ! (n - 2 a + \alpha) !} \leq \frac{1}{\alpha !}
\Big( \frac{C a^{2}}{n} \Big)^{\alpha}.
\end{align*}
Then the following inequalities hold, which in turn prove \eqref{eq:random-remainder-two-chain-bound}.
\begin{align*}
\| \bbU_{n, a} \{g_{j, \bm{\ell}}\}
\|_2^{2} & \leq \sum_{\alpha = 0}^{a} |\calV_{\alpha}^{'} |      
\lesssim C^{2 j} \Big\{ \frac{C (j \vee c) k}{n}
\Big\}^{c} \sum_{\alpha = 0}^{a}
\frac{(C a^{2} / n)^{\alpha}}{\alpha !}
k^{\max \{ \alpha - 1 , 0 \}} \\
& \leq  C^{2 j} \Big\{ \frac{C (j\vee c) k}{n} \Big\}^{c}
\Big( 1 + \frac{1}{k} \sum_{\alpha = 1}^{a}
\frac{(C a^{2} k / n)^{\alpha}}{\alpha !} \Big) \lesssim C^{2 j} \Big\{ \frac{C (j \vee c) k}{n} \Big\}^{c}.
\end{align*}
\end{proof}

Since $m \asymp \log n$, fix a constant $b_{0} < \infty$ such that $m \leq b_{0} \log n$ for all sufficiently large $n$. Fix $q_{0} < \min \{\eta, e^{-2}\}$.
Since $m J k / n = O (1)$, we may choose a sufficiently small fixed constant $\delta > 0$ and define
$c_{\star} \coloneqq \lfloor \delta m J
\rfloor$ so that, for all sufficiently large $n$,
$c_{\star} \geq \max \{J + 1, 2m\}$, $C (m + c_{\star} + 1)k /n \leq q_{0}$.

Let $\hat{\calR}_{m, k, J}^{'}$ be the truncation of \eqref{eq:random-remainder-power-expansion} with $c (\bm{\ell}) \leq c_{\star}$. On $\calG_{n}$, let $\hat{\calR}_{m, k, J}^{''}$ be the remainder. For fixed $j$ and $c$, there are at most $\binom{c + j - 1}{j - 1}$ admissible vectors $\bm{\ell}$. By Minkowski's inequality and
\eqref{eq:random-remainder-two-chain-bound},
\begin{align*}
\|\hat{\calR}_{m, k, J}^{'}
\|_{2} \lesssim \sum_{j = 1}^{m - 1}
\binom{m - 1}{j} C^{j} \sum_{c = J + 1}^{c_{\star}}
\binom{c + j - 1}{j - 1} \Big( \frac{C (j + c) k}{n}
\Big)^{c / 2}.
\end{align*}
The sequence $\Big(\frac{C (m + c) k}{n}
\Big)^{c / 2}$, $J + 1 \leq c \leq c_{\star}$,
is decreasing, because the ratio of two consecutive terms is at most $q_{0}^{1 / 2} e^{1 / 2} < 1$. Moreover,
$\binom{c + j - 1}{j - 1} \leq (m + c_{\star})^{m}$.
It follows that
\begin{equation*}
\|\hat{\calR}_{m, k, J}^{'}
\|_{2} \lesssim C^{m} c_{\star}
(m + c_{\star})^{m} \Big(\frac{C (m + J + 1) k}{n}
\Big)^{(J + 1) / 2}.
\end{equation*}
Since $J = \lceil C_{0} \log n \rceil$, $\frac{C (m + J + 1) k}{n} \leq \frac{C}{\log n}$, $$\log \{\sqrt{n} \|\hat{\calR}_{m, k, J}^{'}\|_{2}\} \leq \frac{1}{2} \log n + \Big( 2 b_{0} - \frac{C_{0}}{2} + o (1) \Big) \log n \log \log n + O (\log n).$$ Choosing the fixed constant $C_{0} > 4 b_{0}$ gives
\begin{equation}
\label{eq:random-remainder-low-second-moment}
\bbE (\hat{\calR}_{m, k, J}^{'})^{2} = o (n^{-1}).
\end{equation}

It remains to control the terms satisfying $c (\bm{\ell}) > c_{\star}$. On $\calG_{n}$, Cauchy-Schwarz inequality gives
\begin{align*}
| g_{j, \bm{\ell}} | \leq |A_{i_{1}}|
|Y_{i_{2}}| \|X_{i_{1}}\|_{2}
\|X_{i_{2}}\|_{2} \prod_{s = 1}^{j - 1}
\| X_{i_{s + 2}} X_{i_{s + 2}}^{\top}
\|_{\op} \prod_{s = 0}^{j - 1}
\| \Delta_{n}^{\ell_{s}}
\|_{\op} \lesssim (C k)^{j} r_{n}^{c (\bm{\ell})}.
\end{align*}
For $a_{j, c} \coloneqq \binom{c + j - 1}{j - 1}r_{n}^{c}$,
one has $a_{j, c + 1} / a_{j, c} = r_{n} (c + j) / (c + 1)$. Since $c \geq c_{\star}$, $j \leq m - 1$, $c_{\star} \geq 2 m$, and $r_{n} \leq 1 / 2$, the last ratio is at most $3 / 4$ for all sufficiently large $n$. Therefore,
\begin{align*}
\sum_{c = c_{\star} + 1}^{\infty}
\binom{c + j - 1}{j - 1} r_{n}^{c}
\lesssim r_{n}^{c_{\star} + 1}
\binom{c_\star+j}{j-1} \lesssim r_{n}^{c_{\star}}
\{C (c_{\star} + m) \}^{j}.
\end{align*}
After summing over $j$,
\begin{equation}
\label{eq:random-remainder-high-bound}
| \hat{\calR}_{m, k, J}^{''}| \lesssim r_{n}^{c_{\star}}
\sum_{j = 1}^{m - 1} \binom{m - 1}{j}
\{ C k (c_{\star} + m) \}^{j} \leq r_{n}^{c_{\star}}
\{ 1 + C k (c_{\star} + m) \}^{m} \eqqcolon e_{n}.
\end{equation}
By the bound $r_{n}^{2} \lesssim k \log n / n$, the definition $c_{\star} = \lfloor \delta m J \rfloor$, and the conditions $J = \lceil C_{0} \log n \rceil$, $m \asymp \log n$, and $k \lesssim n / \log^{3} n$, we have $r_{n}^{2} \lesssim \frac{1}{J}$, $c_{\star} \asymp m J$, $k (c_{\star} + m) = O (n)$. Therefore, for some constant $c_{1} > 0$, $$\log e_{n} \leq - \frac{c_{\star}}{2} \log J + O (c_{\star}) + O (m \log n) \leq - c_{1} (\log n)^{2} \log \log n,$$ for all sufficiently large $n$. Thus, $e_{n} = o (n^{- 1 / 2})$. On $\calG_{n}$, $\hat{\calR}_{m, k, J} = \hat{\calR}_{m, k, J}^{'} + \hat{\calR}_{m, k, J}^{''}$.
Combining~\eqref{eq:random-remainder-low-second-moment} and
\eqref{eq:random-remainder-high-bound} gives $\bbE (\hat{\calR}_{m, k, J}^{2} \mathbbm{1}_{\calG_{n}}) \leq 2 \bbE (\hat{\calR}_{m, k, J}^{'})^{2} + 2 e_{n}^{2} = o (n^{-1})$. Thus, for every $\varepsilon > 0$, as $n \to \infty$, $$\bbP (\sqrt{n} |\hat{\calR}_{m, k, J}| > \varepsilon) \leq \bbP (\calG_{n}^{c}) + \frac{n}{\varepsilon^{2}} \bbE (\hat{\calR}_{m, k, J}^{2} \mathbbm{1}_{\calG_{n}}) \to 0.$$
\end{proof}

\subsection{Proof details of Section~\ref{sec:bias}}
\label{app:bias}

\subsubsection{Proof details related to \textbf{Step i}}
\label{app:bias i}

We first prove Lemma~\ref{lem:bias-binomial-representation}, an alternative representation of the bias $\calB_{m, k}$ defined in \eqref{bias}.

\begin{proof}[Proof of Lemma~\ref{lem:bias-binomial-representation}]
For $j \geq 1$, define $G_{j} \coloneqq \bbU_{n, j + 1} \Big\{ A_{i_{1}} X_{i_{1}}^{\top} \hat{\Omega}_{J} \prod_{s = 3}^{j + 1} \Big( X_{i_{s}} X_{i_{s}}^{\top} \hat{\Omega}_{J} \Big) X_{i_{2}} Y_{i_{2}} \Big\}$, where the product is interpreted as the identity operator when $j = 1$. Then
\begin{align*}
& \sum_{r = 2}^{m} \hat{\IIFF}_{r, r, k}^{(J)} (\Omega) = \sum_{r = 2}^{m} (-1)^{r} \bbU_{n, r} \Big\{ A_{i_{1}} X_{i_{1}}^{\top} \hat{\Omega}_{J} \prod_{s = 3}^{r} \Big( X_{i_{s}} X_{i_{s}}^{\top} \hat{\Omega}_{J} - I \Big) X_{i_{2}} Y_{i_{2}} \Big\} \\
& = \sum_{r = 2}^{m} \sum_{j = 1}^{r - 1} (-1)^{r} (-1)^{r - j - 1} \binom{r - 2}{j - 1} G_{j} = \sum_{j = 1}^{m - 1} (-1)^{j - 1} \sum_{r = j + 1}^{m} \binom{r - 2}{j - 1} G_{j} = \sum_{j = 1}^{m - 1}(-1)^{j - 1} \binom{m - 1}{j} G_{j},
\end{align*}
where the last equality follows from Lemma~\ref{lem:hockey-stick}. Let $G_{0} \coloneqq \hat{\psi}_{2, k} (I) = \bbU_{n, 2} (A_{i_{1}} X_{i_{1}}^{\top} X_{i_{2}} Y_{i_{2}})$, then
\begin{align*}
\chi_{m, k} (\Omega) - \hat{\psi}_{2, k} (I) = \sum_{j = 1}^{m - 1} (-1)^{j - 1} \binom{m - 1}{j} G_{j} - G_{0} = \sum_{j = 0}^{m - 1} (-1)^{j + 1} \binom{m - 1}{j} G_{j}.
\end{align*}

Since $\sum_{j = 0}^{m - 1} (-1)^{j} \binom{m - 1}{j} = 0$, we may subtract the identity operator from every $G_{j}$ without changing the sum. The term $j = 0$ then vanishes. 
We can thus rewrite $\calB_{m, k}$ as:
\begin{align*}
\calB_{m, k} = \bbE \{\chi_{m, k} (\Omega) - \hat{\psi}_{2, k} (I)\} = \sum_{j = 1}^{m - 1} (-1)^{j + 1} \binom{m - 1}{j} \bbE \Big\{ A_{m - 1} X_{m - 1}^{\top} \Big( \prod_{s = 0}^{j - 1} X_{s} X_{s}^{\top} \hat{\Omega}_{J} - I \Big) X_{m} Y_{m} \Big\}.
\end{align*}
Recall again that $X_{0} X_{0}^{\top} = I$ by convention. It remains to expand $\hat{\Omega}_{J}$ as $I + (\hat{\Omega}_{J} - I)$. For fixed $j$,
\begin{align*}
\prod_{s = 0}^{j - 1} X_{s} X_{s}^{\top} \hat{\Omega}_{J}
= \prod_{s = 0}^{j - 1} \{X_{s} X_{s}^{\top} + X_{s} X_{s}^{\top} (\hat{\Omega}_{J} - I)\} = \sum_{S \subseteq [j - 1] \cup \{0\}} \prod_{s = 0}^{j - 1} \{X_{s} X_{s}^{\top} (\hat{\Omega}_{J} - I)^{\mathbbm{1} \{s \in S\}}\}.
\end{align*}
The second equality in \eqref{eq:bias-binomial-representation} follows because, when $S = \emptyset$, $\bbE (\prod_{s = 0}^{j - 1} X_{s} X_{s}^{\top}) - I = 0$.
\end{proof}

We next prove Lemma~\ref{lem:bias-degree-decomposition}.

\begin{proof}[Proof of Lemma~\ref{lem:bias-degree-decomposition}]
By Lemma~\ref{lem:bias-binomial-representation}, it remains only to expand $\sfD_{J} = \sum_{j = 1}^{J} \Delta_{n}^{j}$ at each position indexed by $S$. For every nonempty $S$, we have
\begin{align*}
& \bbE \Big\{ A_{m - 1} X_{m - 1}^{\top} \prod_{s = 0}^{j - 1} \Big( X_{s} X_{s}^{\top} \sfD_{J}^{\mathbbm{1} \{s \in S\}} \Big) X_{m} Y_{m} \Big\} = \sum_{(\ell_{s})_{s \in S} \in [J]^{|S|}} \bbE \Big\{ A_{m - 1} X_{m - 1}^{\top} \prod_{s = 0}^{j - 1} \Big( X_{s} X_{s}^{\top} \Delta_{n}^{\ell_{s} \mathbbm{1} \{s \in S\}} \Big) X_{m} Y_{m} \Big\},
\end{align*}
where we use the convention $\ell_{s} = 0$ for $s \notin S$. We now group these terms according to $c = \sum_{s \in S} \ell_{s}$. Since $S$ is nonempty and $1 \leq \ell_{s} \leq J$ for every $s \in S$, we have $c \geq 1$. Moreover, $c = \sum_{s \in S} \ell_{s} \leq  |S| J \leq j J \leq (m - 1) J$. For a fixed total degree $c$, writing $r = |S|$, the positivity of the $\ell_{s}$ implies $1 \leq r \leq c \wedge j$. Therefore, the collection of all terms having total degree $c$ is exactly $\calM_{c}^{(J)}$ in \eqref{eq:bias-degree-component}. Consequently, $\calB_{m, k} = \sum_{c = 1}^{(m - 1) J} \calM_{c}^{(J)}$, proving \eqref{eq:bias-degree-decomposition}.
\end{proof}

\subsubsection{Proof details related to \textbf{Step ii}}
\label{app:bias ii}

This section is devoted to proving Lemma~\ref{lem:bias-low-degree-cancellation}. 

\begin{proof}[Proof of Lemma~\ref{lem:bias-low-degree-cancellation}]
As a first step toward proving Lemma~\ref{lem:bias-low-degree-cancellation}, we first record the following two intermediate results, which rewrites the expectation in $\calM_{c}^{(J)}$ in a way that facilitates analysis.

\begin{lemma}
\label{lem:bias-centered-word-expansion}
Let $\ell_{0}, \cdots, \ell_{j - 1}$ be nonnegative integers and $\bar\ell_{s} \coloneqq \sum_{r = 0}^{s}\ell_{r},
\ \bar\ell_{-1} \coloneqq 0,\ c \coloneqq \bar \ell_{j - 1} $.
Then
\begin{align}
\label{eq:bias-centered-word-expansion}
& \ \bbE \Big\{ A_{m - 1} X_{m - 1}^{\top} \prod_{s = 0}^{j - 1} (X_{s} X_{s}^{\top} \Delta_{n}^{\ell_{s}}) X_{m} Y_{m} \Big\} \notag \\
= & \ n^{- c} \sum_{i_{1} = 1}^{n} \cdots \sum_{i_{c} = 1}^{n} \bbE \Big[ A_{m - 1} X_{m - 1}^{\top} \prod_{s = 0}^{j - 1} \Big\{ X_{s} X_{s}^{\top} \prod_{l = \bar{\ell}_{s - 1} + 1}^{\bar{\ell}_{s}} (I - X_{i_{l}} X_{i_{l}}^{\top}) \Big\} X_{m} Y_{m} \Big].
\end{align}
If $\bar{\ell}_{s - 1} + 1 > \bar{\ell}_{s}$, the corresponding product is interpreted as the identity matrix. 
\end{lemma}

\begin{proof}[Proof of Lemma~\ref{lem:bias-centered-word-expansion}]
By definition, $\Delta_{n} = I - \hat{\Sigma} = \frac{1}{n} \sum_{i = 1}^{n} (I - X_{i} X_{i}^{\top})$. For each $s = 0, \cdots, j - 1$,
\begin{equation*}
\Delta_{n}^{\ell_{s}} = n^{- \ell_{s}} \sum_{i_{\bar{\ell}_{s - 1} + 1}, \cdots, i_{\bar{\ell}_{s}} = 1}^{n} \prod_{h = \bar{\ell}_{s - 1} + 1}^{\bar{\ell}_{s}} (I - X_{i_{h}} X_{i_{h}}^{\top}).
\end{equation*}
Then we immediately have \eqref{eq:bias-centered-word-expansion}, because $\prod_{s = 0}^{j - 1} n^{- \ell_{s}} = n^{-\sum_{s = 0}^{j - 1} \ell_{s}} = n^{- c}$. 
\end{proof}

\begin{lemma}
\label{lem:bias-centered-word-bound}
Set $W_{i} \coloneqq I - X_{i} X_{i}^{\top}$. Consider a term, obtained by further expanding the right-hand side of
\eqref{eq:bias-centered-word-expansion} using
$X_{s} X_{s}^{\top} = I - W_{s}$, of the form
\begin{equation}
\label{eq:fixed-centered-word}
T \coloneqq \bbE (A_{m - 1} X_{m - 1}^{\top} W_{a_{1}} \cdots W_{a_{c + c^{\dag}}} X_{m} Y_{m}).
\end{equation}
Among the $c + c^{\dag}$ occurrences of $W$, $c$ of them arise from expanding the powers of $\Delta_{n}$, referred to as Type-I indices. The remaining $c^{\dag}$ arise from rewriting $X_{s} X_{s}^{\top}$ as $I - W_{s}$, referred to as Type-II indices. By construction, Type-II indices exclude $m - 1$ and $m$. If $T \neq 0$, $c^{\dag} \leq c$.
\end{lemma}

\begin{proof}[Proof of Lemma~\ref{lem:bias-centered-word-bound}]
Let $\calS$ be Type-II indices described in the statement of the lemma. If some $s \in \calS$ does not coincide with any Type-I index, then $W_{s}$ appears only once and hence does not include $m - 1$ and $m$. Conditioning on all variables except $X_{s}$ gives $\bbE W_{s} = I - \bbE (X_{s} X_{s}^{\top}) = 0$. Hence, every nonzero term must match each of the $c^{\dag}$ Type-II indices with at least one Type-I index generated from the expansion of $\Delta_{n}$. Thus $c^{\dag} \leq c$ and the proof is complete. 
\end{proof}

Armed with Lemmas~\ref{lem:bias-centered-word-expansion} and~\ref{lem:bias-centered-word-bound}, we continue the proof of Lemma~\ref{lem:bias-low-degree-cancellation}. Fix $c$ such that $1 \leq c < \sfc_{m}$.
By Lemma~\ref{lem:bias-centered-word-expansion}, every summand in $\calM_{c}^{(J)}$ can be written as a linear combination of terms of the form $n^{-c} T$. 
By Lemma~\ref{lem:bias-centered-word-bound}, it is zero unless $c^{\dag} \leq c$. We now fix one potentially nonzero term of the above form and consider its binomially weighted coefficient. For $1 \leq r \leq c$, define $$\sfp_{J} (c, r) \coloneqq \Big| \Big\{ (\ell_{1}, \cdots, \ell_{r}) \in [J]^{r}: \sum_{u = 1}^{r} \ell_{u} = c \Big\} \Big|.$$ For a fixed order $j$, suppose that exactly $r$ of the $j$ positions have positive powers of $\Delta_{n}$. There are $\binom{j}{r}$ ways to choose these positions. For each choice, there are $\sfp_{J} (c, r)$ ways to assign positive exponents in $[J]$ whose sum is $c$. Since $1 \leq r \leq c \wedge j$, the total number of choices is $\sum_{r = 1}^{c \wedge j} \binom{j}{r} \sfp_{J} (c, r)$. Define $\Phi_{c, J} (j) \coloneqq \sum \limits_{r = 1}^{c} \binom{j}{r} \sfp_{J} (c,r)$, where $\binom{j}{r} = 0$ if $r > j$. Since $\sfp_{J} (c, r)$ does not depend on $j$ and $r \leq c$, $\Phi_{c, J} (j)$ is a polynomial in $j$ of degree at most $c$.

The $c^{\dag}$ Type-II indices are obtained by choosing $c^{\dag}$ of the $j - 1$ factors
$X_{s} X_{s}^{\top}$, $s \in [j - 1]$, and selecting $- W_{s}$ from
the expansion $X_{s} X_{s}^{\top} = I - W_{s}$ at each chosen position. These choices contribute the factor
$(-1)^{c^{\dag}} \binom{j - 1}{c^{\dag}}$. Combining this factor with the weight $(-1)^{j + 1} \binom{m - 1}{j}$, the corresponding unrestricted coefficient has the form ($j = 0$ included due to $c \geq 1 \Rightarrow \Phi_{c, J} (0) = 0$)
\begin{align}
\label{eq:low-degree-coefficient}
\frakC_{c, c^{\dag}, J}
= (-1)^{c^{\dag}} \sum_{j = 0}^{m - 1}
(-1)^{j + 1} \binom{m - 1}{j}
\Phi_{c, J} (j) \binom{j - 1}{c^{\dag}}.
\end{align}

Since $\Phi_{c, J} (j)$ has degree at most $c$ and $\binom{j - 1}{c^{\dag}}$ has degree $c^{\dag}$, there exist constants
$\gamma_{0}, \cdots, \gamma_{c + c^{\dag}}$ such that $\Phi_{c, J} (j) \binom{j - 1}{c^{\dag}}
= \sum_{\ell = 0}^{c + c^{\dag}}
\gamma_{\ell} j^{\ell}$. To determine the coefficient of a fixed term $T$ of the form \eqref{eq:fixed-centered-word}, fix $r \leq c$ and the positive exponents $\ell_{1}, \cdots, \ell_{r}$ with $\sum_{u = 1}^{r} \ell_{u} = c$, and specify the relative order in which the factors $\Delta_{n}^{\ell_{1}}, \cdots, \Delta_{n}^{\ell_{r}}$ and the $c^{\dag}$ selected $W_{s}$ factors appear in the product. The number of ways to choose these $r + c^{\dag}$ positions in the original product is a polynomial in $j$ of degree at most $r + c^{\dag} \leq c + c^{\dag}$. Summing over the choices of $r$, the exponents, and the orders that produce $T$ preserves this degree bound. Therefore the coefficient attached to $T$ has the form $$\frakC_{\ast} = \sum_{j = 0}^{m - 1} (-1)^{j + 1} \binom{m - 1}{j} p_{c, c^{\dag}, J} (j),$$ where $p_{c, c^{\dag}, J}$ is a polynomial satisfying $\deg p_{c, c^{\dag}, J} \leq c + c^{\dag} $. For every integer $d < m - 1$, $\sum_{j = 0}^{m - 1} (-1)^{j} \binom{m - 1}{j} j^{d} = 0$. Thus the coefficient $\frakC_{\ast}$ is zero if $c + c^{\dag} < m - 1$. By $c < \sfc_{m} = \lceil (m - 1) / 2 \rceil$, $c + c^{\dag} \leq 2 c < m - 1$. Hence the coefficient of every potentially nonzero term $T$ of the form \eqref{eq:fixed-centered-word} is zero; all other terms are zero by Lemma~\ref{lem:bias-centered-word-bound}, so $\calM_{c}^{(J)} = 0$, for $1 \leq c < \sfc_{m}$.    
\end{proof}

\subsubsection{Proof details related to \textbf{Step iii}}
\label{app:bias iii}

Finally, we are left to prove Lemma~\ref{lem:bias-surviving-levels}. We first
apply Lemma~\ref{lem:graph-counting} to bound the quantity below.
\begin{lemma}
\label{lem:bias-averaged-centered-product-bound}
Let $1 \leq j \leq m - 1$, and $\ell_{0}, \cdots, \ell_{j - 1}$ be nonnegative integers such that $ \sum_{s = 0}^{j - 1} \ell_{s} = c \geq 1 $. Define
\begin{equation}
\label{surplus}
\calS_{j, \bm \ell}
\coloneqq \bbE \Big\{ A_{m - 1} X_{m - 1}^{\top}
\prod_{s = 0}^{j - 1} (X_{s} X_{s}^{\top} \Delta_{n}^{\ell_{s}}) X_{m} Y_{m} \Big\}.
\end{equation}
Let $s_{c} \coloneqq \lceil\frac{c}{2}\rceil \vee 1, \ \zeta_{A, Y} \coloneqq \|A\|_{2} \|Y\|_{2} + \|A\|_{\infty} \|Y\|_{2} + \|A\|_{2} \|Y\|_{\infty}$. If $\frac{C (j \vee c) k}{n} \leq \eta < 1$, then
\begin{equation}
\label{eq:bias-averaged-centered-product-bound}
|\calS_{j, \bm \ell}|
\lesssim \zeta_{A, Y}
\Big( \frac{C (j \vee c) k}{n} \Big)^{s_{c}}.
\end{equation}
\end{lemma}

\begin{proof}
Let $\calA_{j} \coloneqq
\{m - 1, m, 1, \cdots, j - 1\}$. Then $|\calA_{j}| \lesssim j$. After replacing every occurrence of $\Delta_{n}$ in \eqref{surplus} by $I$, the product of inner products becomes $X_{m - 1}^{\top} \prod_{s = 1}^{j - 1} (X_{s} X_{s}^{\top}) X_{m}$. Let $\Gamma_{0}$ denote the associated graph. The scalar variables $A_{m-1}$ and $Y_{m}$ are attached to the two endpoint vertices and do not create new edges. Since $e (\Gamma_{0}) = j$, $v (\Gamma_{0}) = j + 1$, $\kappa (\Gamma_{0}) = 1$, we have $\frakr (\Gamma_{0}) = e (\Gamma_{0}) - v (\Gamma_{0}) + \kappa (\Gamma_{0}) = 0$.
Applying the argument in the proof of Lemma~\ref{lem:graph-counting} with
$\calA = \calA_{j}$ and $D = c$, while controlling the endpoint
factors involving $A$ and $Y$ by \Holder{}'s inequality,
gives $|\calS_{j, \bm{\ell}}| \lesssim \zeta_{A, Y} \Big( \frac{C (j \vee c)k}{n} \Big)^{s_{c}}$, so \eqref{eq:bias-averaged-centered-product-bound} holds.
\end{proof}

\begin{proof}[Proof of Lemma~\ref{lem:bias-surviving-levels}]

Fix $\sfc_{m} \leq c \leq (m - 1)J$. For $1 \leq r \leq c$, define $\sfp_{J} (c, r) \coloneqq |\{(\ell_{1}, \cdots, \ell_{r}) \in \{1, \cdots, J\}^{r}: \sum_{u = 1}^{r} \ell_{u} = c\}|$. For fixed $j$ and $r$, there are $\binom{j}{r}$ ways to
choose the positions with positive powers of $\Delta_{n}$ and
$\sfp_{J} (c , r)$ ways to assign their powers so that the total
power is $c$. Hence,
\begin{equation*}
N_{j, c, J} \coloneqq \sum_{r = 1}^{c \wedge j} \binom{j}{r} \sfp_{J}(c, r) \leq \sum_{r = 1}^{c \wedge j} \binom{j}{r} \binom{c - 1}{r - 1}
= \binom{j + c - 1}{c}.
\end{equation*}
Applying Lemma~\ref{lem:bias-averaged-centered-product-bound}
to each term in $\calM_{c}^{(J)}$ and using the preceding
counting bound, we obtain
\begin{align*}
|\calM_{c}^{(J)}| & \lesssim \zeta_{A, Y} \sum_{j = 1}^{m - 1} \binom{m - 1}{j} N_{j, c, J} \Big( \frac{C (j \vee c) k}{n} \Big)^{s_{c}} \leq \zeta_{A, Y}
\Big( \frac{C (m \vee c) k}{n} \Big)^{s_{c}} \sum_{j = 1}^{m - 1} \binom{m - 1}{j} \binom{j + c - 1}{c}.
\end{align*}
By $c \geq \sfc_{m} = \lceil (m - 1) / 2 \rceil$, $m - 1 \leq 2 c$. Hence, $$\sum_{j = 1}^{m - 1} \binom{m - 1}{j} \binom{j + c - 1}{c} \leq \sum_{j = 1}^{m - 1} \binom{m - 1}{j} 2^{j + c - 1} \leq 2^{c} 3^{m - 1} \leq C_{1}^{c}.$$ Since $s_{c} = \lceil c / 2\rceil$, we have $c \leq 2 s_{c}$. Thus, we can always choose the constants appropriately so that $C_{1}^{c} \Big( \frac{C (m \vee c) k}{n} \Big)^{s_{c}} \leq \Big( \frac{C (m \vee c) k }{n} \Big)^{s_{c}}$. Now the proof is complete.
\end{proof}

\subsection{Proof details of Section~\ref{sec:variance}}
\label{app:variance}

\subsubsection{Proof details related to \textbf{Step ii}}
\label{app:var ii}

\begin{proof}[Proof of Lemma~\ref{lem:finite-neumann-mobius}]
By \eqref{eq:finite-neumann-centering-main}, $\sum_{i \notin \{i_{1}, i_{2}\}} H_{i}^{(J)} = - (\calR_{i_{1} i_{2}}^{(J)} + n E_{J})$. Using this identity in the proof of Lemma~\ref{lem:mobius} gives the first equality in \eqref{eq:finite-neumann-mobius-main}. By the definition of
$K_{\calB, S}^{(J)}$, 
we have $K_{\calB}^{(J), *} = \sum_{S \subseteq \calU_{\calB}} K_{\calB, S}^{(J)}$. Separating $S = \emptyset$ from $S \neq \emptyset$ gives the second equality in \eqref{eq:finite-neumann-mobius-main}.
\end{proof}

\subsubsection{Proof details related to \textbf{Step iii}}
\label{app:var iii}

\begin{proof}[Proof of Lemma~\ref{lem:mobius-block-to-chain}]
By the definitions of $\calE (\calB)$ and $d_{\calB, \varepsilon}$, the full expansion of all $H^{(J)}$- and $\calR^{(J)}$-factors in $K_{\calB}^{(J),*}$ is indexed by $\varepsilon \in \calE(\calB)$. 
At a position $l \in B_{\nu}$, the values
$\varepsilon_{l} = 0$ and $\varepsilon_{l} = \nu$ correspond, respectively, to selecting $- I$ and $X_{a_{\nu}} X_{a_{\nu}}^{\top} \hat{\Omega}_{J}$. At a position $l \notin \bigcup_{\nu = 1}^{r} B_{\nu}$, the values
$0$, $1$, $2$, and $\star$ correspond, respectively, to selecting $- 2 I$, $X_{i_{1}} X_{i_{1}}^{\top} \hat{\Omega}_{J}$,
$X_{i_{2}} X_{i_{2}}^{\top} \hat{\Omega}_{J}$, and $n E_{J}$.
For each such $\varepsilon$, the selections of $- I$ and $- 2 I$ contribute $(- 1)^{N_{H} (\varepsilon)} (- 2)^{N_{R} (\varepsilon)}$,
and the $s_{\varepsilon}$ selections of $n E_{J}$ contribute $n^{s_{\varepsilon}}$. Thus, their product is
$d_{\calB, \varepsilon}$.

Fix $\varepsilon \in \calE (\calB)$, and recall that $l_{1} < \cdots < l_{q_{\varepsilon}}$ are the positions of its selected factors of the form $X_{u} X_{u}^{\top} \hat{\Omega}_{J}$. For $\bm{h} \in \{0, \cdots, J\}^{q_{\varepsilon} + 1}$, the powers selected from the initial occurrence of $\hat{\Omega}_{J}$ immediately following $X_{i_{1}}^{\top}$ and from the factors $\hat{\Omega}_{J}$ in $X X^{\top} \hat{\Omega}_{J} $ at positions
$l_{1}, \cdots, l_{q_{\varepsilon}}$ are $h_{0}, \cdots, h_{q_{\varepsilon}}$, respectively. For each $0 \leq u \leq q_{\varepsilon}$, the $\star$-selections at positions $l_{u} < l < l_{u + 1}$ add $(J + 1) s_{\varepsilon, u}$ to $h_{u}$, giving $\ell_{\varepsilon, \bm{h}, u}
= h_{u} + (J + 1) s_{\varepsilon, u}$. Recall that $a_{\varepsilon} = 2 + b_{\varepsilon}$ and $\calA_{\varepsilon} = \{\nu_{1}, \cdots, \nu_{b_{\varepsilon}}\}$, where
$\nu_{1} < \cdots < \nu_{b_{\varepsilon}}$.
Set $i_{2 + t} \coloneqq a_{\nu_{t}}$ for
$1 \leq t \leq b_{\varepsilon}$ and write
$\bfi_{a_{\varepsilon}}
\coloneqq (i_{1}, i_{2}, a_{\nu_{1}}, \cdots, a_{\nu_{b_{\varepsilon}}})$.
With this relabeling, the term indexed by
$(\varepsilon, \bm{h})$ is
$k_{a_\varepsilon, q_{\varepsilon}, \gamma_{\varepsilon}, \bm\ell_{\varepsilon, \bm{h}}} (O_{\bfi_{a_\varepsilon}})
$ defined in \eqref{before chain}, where the the selected $X_{u} X_{u}^{\top}$ factors retain their order in the original
product.

If $\nu \notin \calA_{\varepsilon}$, then $a_{\nu}$ does not appear in this kernel. For each fixed $\bfi_{a_{\varepsilon}} \in \calI_{n, a_{\varepsilon}}$, these
$r - b_{\varepsilon}$ remaining indices can be chosen, in order, in
$\frac{(n - a_{\varepsilon}) !}{(n - 2 - r) !}$ ways. Consequently, their summation changes the original normalization according to $\frac{(n - 2 - r) !}{n !} \frac{(n -a_{\varepsilon}) !}{(n - 2 - r) !}
= \frac{(n - a_{\varepsilon}) !}{n !}$,
which is exactly the normalization of
$\bbU_{n, a_{\varepsilon}}$. The fixed pair $(\varepsilon, \bm{h})$ therefore reduces to $d_{\calB, \varepsilon} T_{a_{\varepsilon}, q_{\varepsilon}, \gamma_{\varepsilon}, \bm{\ell}_{\varepsilon, \bm{h}}}$. Summing over all $\varepsilon$ and $\bm{h}$ yields the desired \eqref{lemma 10}. It remains to prove the inequality $q_{\varepsilon} - b_{\varepsilon} \le \iota - r - s_{\varepsilon}$. For $\nu = 1, \cdots, r$, define $t_{\nu} (\varepsilon) \coloneqq |\{l \in B_{\nu}: \varepsilon_{l} = \nu\}|$. Then $\nu \in \calA_{\varepsilon}$ if and only if $t_{\nu} (\varepsilon) \ge 1$. Also define $u_{\varepsilon} \coloneqq |\{l \notin \cup_{\nu = 1}^{r} B_{\nu} : \varepsilon_{l}  \in \{1, 2\}\}|$.
Since $q_{\varepsilon} = \sum_{\nu \in \calA_{\varepsilon}} t_{\nu} (\varepsilon) + u_{\varepsilon}$ and $u_{\varepsilon} + s_{\varepsilon} \leq \iota - \sum_{\nu = 1}^{r} |B_{\nu}|$, we have $$q_{\varepsilon} - b_{\varepsilon} = \sum_{\nu \in \calA_{\varepsilon}} \{t_{\nu} (\varepsilon) - 1\} + u_{\varepsilon} \le \sum_{\nu = 1}^{r} (|B_{\nu}| - 1) + (\iota - \sum_{\nu = 1}^{r} |B_{\nu}|) - s_{\varepsilon} = \iota - r - s_{\varepsilon}.$$ Finally, after multiplied by $c_{\calB, n}$ and summing over $\calB \in \bbB_{\iota}$, $s_{\varepsilon} = 0$ and $s_{\varepsilon} \geq 1$ give $M_{j, J}$ and $L_{j, J}$, respectively.

\end{proof}

\subsubsection{Proof details related to \textbf{Step iv}}
\label{app:var iv}

We prove the variance decomposition of $T_{a, q, \gamma, \bm{\ell}}$ as stated in \eqref{var bound decomposition} after Lemma~\ref{lem:mobius-block-to-chain}. Let $$V_{\alpha} \coloneqq \frac{(n - a) !^{2}}{(n !)^{2}} \sum\limits_{\substack{\bfi, \bfi' \in \calI_{n, a} \\ \alpha (\bfi, \bfi') = \alpha}} |\cov \{k (\bfi), k (\bfi')\}|, \text{ for $\alpha = 0, 1, \cdots, a$}.$$

\begin{lemma}
\label{lem:chain-covariance-decomposition}
Let $T_{a, q, \gamma, \bm{\ell}}$ be defined as in \eqref{before chain}.
For brevity, write $k (\bfi) \coloneqq
k_{a, q, \gamma, \bm{\ell}} (O_{\bfi})$. For $\bfi, \bfi' \in \calI_{n, a}$, define the overlap number $ \alpha (\bfi, \bfi') \coloneqq |\operatorname{ind} (\bfi) \cap \operatorname{ind} (\bfi')|$. Then $\var (T_{a, q, \gamma, \bm{\ell}}) \le \sum_{\alpha = 0}^{a} V_{\alpha}$.
\end{lemma}

\begin{proof}
By definition, $T_{a, q, \gamma, \bm{\ell}} = \frac{(n - a) !}{n !} \sum_{\bfi \in \calI_{n, a}} k (\bfi)$. Therefore,
\begin{align*}
\var (T_{a, q, \gamma, \bm{\ell}}) = \frac{(n - a)!^{2}}{(n!)^{2}} \sum_{\bfi, \bfi' \in \calI_{n, a}} \cov \{k (\bfi), k (\bfi')\} \le \frac{(n - a)!^{2}}{(n!)^{2}} \sum_{\bfi, \bfi' \in \calI_{n, a}} |\cov \{k (\bfi), k (\bfi')\}|.
\end{align*}
Grouping the pairs $(\bfi, \bfi')$ according to the overlap number $\alpha (\bfi, \bfi')$ gives
$\var (T_{a, q, \gamma, \bm{\ell}}) \le \sum_{\alpha = 0}^{a} V_{\alpha}$.
\end{proof}

\subsubsection{Proof details related to \textbf{Step v}}
\label{app:var v}



\begin{proof}[Proof of Lemma~\ref{lemma:generic-chain-variance-explicit}]
By definition, $T_{a, q, \gamma, \bm{\ell}} = \bbU_{n, a} \Big\{ A_{i_{1}} X_{i_{1}}^{\top} \Delta^{\ell_{0}}_{n} \Big\{ \prod_{u = 1}^{q} X_{i_{\gamma (u)}} X_{i_{\gamma (u)}}^{\top} \Delta^{\ell_{u}}_{n} \Big\} X_{i_{2}} Y_{i_{2}}\Big\}$. Recall that $\calI_{n, a}$ denotes the collection of ordered tuples $\bfi = (i_{1}, \cdots, i_{a})$ of pairwise distinct sample indices, and that $\alpha (\bfi, \bfi')$ denotes the overlap number of the associated index sets.
Let $V_{\alpha}$ denote the contribution to the variance from pairs of ordered tuples satisfying $\alpha (\bfi, \bfi') = \alpha$, for $\alpha = 0, 1, \cdots, a$. By Lemma~\ref{lem:chain-covariance-decomposition}, $\var (T_{a, q, \gamma, \bm{\ell}}) \le \sum_{\alpha = 1}^{a} V_{\alpha} + V_{0}$. Fix $\bfi, \bfi' \in \calI_{n, a}$ with overlap $\alpha$, and write $K_{L} \coloneqq k (\bfi)$, $K_{R} \coloneqq k (\bfi')$. Let $S \coloneqq \ind (\bfi) \cup \ind (\bfi')$, $|S| = 2 a - \alpha$. There are $2 d$ $\Delta_{n}$'s in the kernel product.

\paragraph{The analysis of $V_{\alpha}$ for $\alpha \geq 1$.} The graph $G_{\alpha}$ obtained by replacing every occurrence of $\Delta_{n}$ by $I$ in the product of the two kernels satisfies $e (G_{\alpha}) = 2 (q + 1), v (G_{\alpha}) = 2a - \alpha, \kappa (G_{\alpha}) = 1, \frakr (G_{\alpha}) = 2 (q - a + 2) + \alpha - 1$.
If $d \geq 1$, apply the proof of Lemma~\ref{lem:graph-counting} to the product of the two kernels with $\calA = S$ and $D = 2 d$, retaining the factors involving $A$ and $Y$. Since $|S| \vee 2 d \leq C (a \vee d)$, the $2 d$ occurrences of $\Delta_{n}$ contribute at most $\Big\{\frac{C (a \vee d) k}{n}\Big\}^{d}$. When $d = 0$, the corresponding factor is one.

For the product of the two separate expectations,
apply the same argument to each kernel separately.
After replacing every occurrence of $\Delta_{n}$ by $I$
and relabeling the sample variables separately,
the disjoint union of the two resulting graphs has
first Betti number $2 (q - a + 2)$, which is no larger
than $\frakr (G_{\alpha})$. For $d \geq 1$, since $2 \lceil d / 2\rceil \geq d$ and $C (a \vee d) k / n \leq \eta < 1$, $\Big\{ \frac{C (a \vee d) k}{n} \Big\}^{2 \lceil d / 2\rceil} \leq \Big\{ \frac{C (a \vee d) k}{n} \Big\}^{d}$. For $d = 0$, this factor is one.
Thus, using $|\cov (K_{L}, K_{R})|
\leq |\bbE(K_{L} K_{R})| + |\bbE K_{L}|\,|\bbE K_{R}|$,
the same upper bound, up to a constant factor,
applies to the absolute covariance.

It remains to count ordered pairs of tuples according to the overlap of the sets of sample indices appearing in the two tuples. Fix the first ordered tuple $\bfi \in \calI_{n, a}$. To construct a second ordered tuple $\bfi' \in \calI_{n, a}$ with $\alpha (\bfi, \bfi') = \alpha$, we first choose the $\alpha$ indices of $\bfi$ that are shared, choose the $\alpha$ entries in $\bfi'$ occupied by these shared indices, and assign the shared indices to these entries in $\alpha!$ possible ways. The remaining $a - \alpha$ entries of $\bfi'$ are chosen as an ordered selection from the $n - a$ indices outside the first tuple. Hence, for each fixed $\bfi$, the number of such $\bfi'$ is $\binom{a}{\alpha}^{2} \alpha ! \frac{(n - a) !}{(n - 2 a + \alpha) !}$. Since the number of choices for the first ordered tuple $\bfi$ is $n ! / (n - a) !$, multiplying by the ordered $U$-statistic normalization $\{(n - a) ! / n!\}^{2}$ gives the combinatorial prefactor $\binom{a}{\alpha}^{2} \alpha ! \frac{((n - a) !)^{2}}{n ! (n - 2 a + \alpha) !}$.

Since $n \geq 2 a$, we have $((n - a) !)^{2} / n ! (n - 2 a + \alpha) ! \leq (\frac{2}{n} )^{\alpha}$. Combining $\binom{a}{\alpha} \leq a^{\alpha} / \alpha !$, we obtain
\begin{align*}
\binom{a}{\alpha}^{2} \alpha ! \frac{((n - a) !)^{2}}{n ! (n - 2 a + \alpha) !} \leq \Big(\frac{a^{\alpha}}{\alpha !} \Big)^{2} \alpha ! \Big(\frac{2}{n}\Big)^{\alpha} = \frac{1}{\alpha !} \Big(\frac{2 a^{2}}{n}\Big)^{\alpha} \leq\frac{1}{\alpha !} \Big(\frac{Ca^{2}}{n} \Big)^{\alpha}.
\end{align*}
Consequently, for $1 \leq \alpha \leq a$, $V_{\alpha} \lesssim C^{2 q} k^{2 q - 2 a + 4} \Big\{ \frac{C (a \vee d) k}{n} \Big\}^{d} \frac{k^{\alpha - 1}}{\alpha !} \Big(\frac{C a^{2}}{n} \Big)^{\alpha}$. Then:
\begin{equation}
\label{eq:finite-neumann-overlap-sum}
\sum_{\alpha = 1}^{a} V_{\alpha}
\lesssim  \frac{C^{2 q} a^{2} }{n}k^{2 q - 2 a + 4} \Big\{\frac{C (a \vee d) k}{n} \Big\}^{d}
\exp \Big( \frac{C a^{2} k}{n} \Big).
\end{equation}

\paragraph{The analysis of $V_{0}$.}
If $d = 0$, the two kernels depend on disjoint sample indices and their covariance is zero. Let $d \geq 1$, and define $S_{L} \coloneqq \ind(\bfi), S_{R} \coloneqq \ind(\bfi'), S \coloneqq S_{L} \cup S_{R}$. Thus $S_{L} \cap S_{R} = \emptyset$ and $|S_{L}| = |S_{R}| = a$. Write $\Delta_{n} = U_{S} + V_{S}$, $U_{S} \coloneqq \frac{1}{n} \sum_{u \in S} (I - X_{u} X_{u}^{\top})$, and $V_{S} \coloneqq \frac{1}{n}\sum_{u \notin S} (I - X_{u} X_{u}^{\top})$.

For a fixed choice of $U_{S}$ and $V_{S}$ in the two kernels, let $r_{A}$ and $r_{F}$ be their total multiplicities, so that $r_{A} + r_{F} = 2 d$. For the $V_{S}$-terms, let $u_{1}, \cdots, u_{r_{F}}$ denote the corresponding sample indices, and let $\pi$ be the partition of $[r_{F}]$ induced by their equality pattern. Define $d_{F} \coloneqq r_{F} - |\pi|$. Since $\bbE (I - X_{u} X_{u}^{\top}) = 0$, any term in which one of the indices $u_{1}, \cdots, u_{r_{F}}$ occurs exactly once has zero mean and therefore vanishes. Hence, in every nonzero term, each distinct index among $u_{1}, \cdots, u_{r_{F}}$ occurs at least twice. Consequently,
\begin{equation}
\label{eq:finite-neumann-two-chain-defect}
r_{A} + r_{F} = 2 d, \quad d_{F} \geq \Big \lceil \frac{r_{F}}{2} \Big \rceil, \quad r_{A} + d_{F} \geq d.
\end{equation}

Expand $K_{L}$ and $K_{R}$ in both terms of the covariance by selecting $U_{S}$ or $V_{S}$ at every occurrence of $\Delta_{n}$
and then selecting $I$ or $- X_{u} X_{u}^{\top}$ from each
resulting factor $I - X_{u} X_{u}^{\top}$. For each fixed set of selections, compare the corresponding term in $\bbE (K_{L} K_{R})$ with the product of the corresponding terms
in $\bbE K_{L}$ and $\bbE K_{R}$. 
If, in the term from $\bbE (K_{L} K_{R})$, no sample index appears in both kernels, then the two kernels depend on disjoint sets of observations. By independence, the expectation of their product equals the product of their expectations, so the two corresponding terms cancel before absolute values are taken. Then all remaining terms satisfy one of the following two cases. 

In the first case, a $U_{S}$ position in one kernel selects a factor $X_{u} X_{u}^{\top}$ with $u$ in the index set of the other kernel.
There are at most $a r_{A}$ choices of this position and index. Once they are fixed, the selected term contributes a factor $n^{- 1}$. Each of the remaining $r_{A} - 1$ occurrences of $U_{S}$ contributes another factor $n^{-1}$. The partition of the $V_{S}$ positions is unchanged:
the $r_{F}$ summation indices have $|\pi|$ distinct values, so their index count together with the factor $n^{- r_{F}}$ contributes at most $n^{|\pi| - r_{F}} = n^{- d_{F}}$. Thus, after separating the selected factor $n^{- 1}$, the remaining coefficient and index count contribute at most $n^{- \{(r_{A} - 1) + d_{F}\}}$ for each fixed assignment of the remaining $U_{S}$ indices.

In the second case, an element of $\pi$ contains positions from both kernels. There are at most $r_{F}^{2}$ choices
of one such position in each kernel.
Let $u$ and $v$ denote their summation indices. For the terms satisfying $u = v$, there are $n - |S|$ choices of the common index. Together with the factor $n^{- 2}$ from the two $V_{S}$ factors, this gives
$\frac{1}{n^{2}} \sum_{\substack{u, v \notin S\\u = v}} 1 = \frac{n - |S|}{n^{2}} \leq \frac{1}{n}$. For the partition count, take the selected pair as one position. There are then $r_{F} - 1$ positions, while the number of elements remains $|\pi|$.
The element containing the combined position may now have cardinality one, whereas every other element has
cardinality at least two. Therefore,
$r_{F} - 1 \geq 1 + 2 (|\pi| - 1)
= 2 |\pi| - 1$, which gives $|\pi| \leq r_{F} - |\pi| = d_{F}$. Moreover, $(r_{F} - 1) - |\pi| = d_{F} - 1$. For fixed $|\pi|$, the number of partitions after combining
the selected pair is one when $d_{F} = 1$ and, when $d_{F} \geq 2$, is at most
$\binom{r_{F} - 1}{d_{F} - 1} |\pi|^{d_{F} - 1} \leq \{C (r_{F} \vee 1)\}^{d_{F} - 1}$.
After separating the factor $n^{-1}$ from the selected pair, the remaining $V_{S}$ index count contributes at most $n^{- (d_{F} - 1)}$. Together with the $r_{A}$ factors $n^{- 1}$ from $U_{S}$, this gives $n^{- \{r_{A} + (d_{F} - 1)\}}$ for each fixed assignment of the $U_{S}$ indices.

For an upper bound, we may sum over both cases, even if a term is counted more than once. The total number of choices of the selected position and index, or of the selected pair of positions,
is at most $a r_{A} + r_{F}^{2}$. The exponents of $n^{-1}$ in the two cases coincide: $(r_{A} - 1) + d_{F} = r_{A} + (d_{F} - 1) = r_{A} + d_{F} - 1 \geq d - 1$, where the last inequality follows from
\eqref{eq:finite-neumann-two-chain-defect}.

We next bound the corresponding powers of $k$. Let $G_{0}$ be the graph obtained by replacing every occurrence of $\Delta_{n}$ by $I$ in the product of the two kernels. Since it is the disjoint union of the two kernel graphs, $e (G_{0}) = 2 (q + 1), v(G_{0}) = 2 a, \kappa (G_{0}) = 2, \frakr (G_{0}) = 2 (q - a + 2)$. In the first case, for the graph calculation, we first
insert $X_{u} X_{u}^{\top}$ selected at the chosen $U_{S}$ position
into $G_{0}$, before accounting for the remaining $X_{v} X_{v}^{\top}$ factors
selected from the $U_{S}$- and $V_{S}$-positions.
The selected term $X_{u} X_{u}^{\top}$ adds one edge and joins the two components of $G_{0}$. Thus the first Betti number does not increase. In the product of the two separate expectations, suppose that the chosen $U_{S}$ position is in $K_{L}$ and that its
selected index $u$ belongs to $S_{R}$; the other case is analogous. Since the two expectations are evaluated separately, $u$ is a new vertex in the graph associated with $K_{L}$. Incorporating the factor $X_{u} X_{u}^{\top}$ increases both the number of vertices and the number of edges by one, so the first Betti number of this graph does not change. The graph associated with $K_{R}$ is unchanged. Hence the first Betti number of the disjoint union of the two graphs does not increase. Hence the selected term $X_{u} X_{u}^{\top}$ contributes a factor $n^{- 1}$ without an additional factor $k$. Each of the remaining $r_{A} - 1$ positions selected from $U_{S}$ increases the first Betti number by at most one. Thus the total increase due to the $U_{S}$ positions is at most $r_{A} - 1$. The $V_{S}$ positions contribute at most $\sum_{B' \in \pi} (|B'| - 1) = d_{F}$, so the total increase in the first case is at most $(r_{A} - 1) + d_{F}$.

In the second case, let $B \in \pi$ be the element containing the selected pair. The bound $|B| - 1$ in the proof of
Lemma~\ref{lem:graph-counting} is reduced to
$|B| - 2$. Indeed, in the expectation of the product of the two kernels, if $X_{u} X_{u}^{\top}$ is selected at least once in each kernel, the first selected occurrence in one kernel increases both the number of edges and the number of vertices by one, while
the first occurrence in the other kernel joins the two components without increasing the first Betti number. Hence only the remaining at most $|B| - 2$ selected occurrences can increase the first Betti number. Otherwise, since $B$ contains positions from both kernels, at least one position in $B$ selects the identity matrix,
so at most $|B| - 1$ factors $X_{u} X_{u}^{\top}$ are selected. If at least one position in $B$ selects $X_{u} X_{u}^{\top}$, the first such occurrence increases both the number of edges and the number of vertices by one; hence the first Betti number increases by at most $|B| - 2$. If none is selected, the graph is unchanged. For the product of the two separate expectations, let
$B_{L}$ and $B_{R}$ denote the subsets of $B$ consisting of positions in $K_{L}$ and $K_{R}$, respectively. Since the
two expectations are evaluated separately, the contribution of $B$ to the first Betti number is at most $(|B_{L}| - 1) + (|B_{R}| - 1) = |B| - 2$. Since $d_{F} = \sum_{B' \in \pi} (|B'| - 1)$, the total increase due to the $V_{S}$ positions in the second case is at most $$(|B| - 2) + \sum_{\substack{B' \in \pi\\ B' \ne B}} (|B'| - 1) = d_{F} - 1.$$
Together with the contribution of at most $r_{A}$ from the $U_{S}$ positions, this gives the same upper bound $r_{A} + d_{F} - 1$ as in the first case. Thus, in both cases, the graph $G$ associated with a
resulting multiplicative kernel term satisfies
$\frakr (G) \leq 2 (q - a + 2) + r_{A} + d_{F} - 1$.

Applying the coefficient, partition, and graph calculation
in the proof of Lemma~\ref{lem:graph-counting}
to the remaining positions, together with \Holder{}'s
inequality, gives, for fixed $r_{A}$, $r_{F}$, and $d_{F}$,
the bound
\begin{equation*}
C^{2 q + 2 d} k^{2 (q - a + 2)} \frac{C (a r_{A} + r_{F}^{2})}{n} \Big\{\frac{C (a \vee d) k}{n} \Big\}^{r_{A} + d_{F} - 1}.
\end{equation*}
Multiplied by the $U$-statistic normalization,
the number of disjoint ordered pairs contributes a factor
at most one. Indeed, this factor equals
$\frac{((n - a) !)^{2}}{n ! (n - 2 a) !}
= \prod_{t = 0}^{a - 1} \frac{n - a - t}{n - t}
\leq 1$. Summing over the remaining choices of $U_{S}$ or $V_{S}$
and of $I$ or $- X_{u} X_{u}^{\top}$ at the $2 d$
positions contributes at most a factor $C^{2 d}$. By \eqref{eq:finite-neumann-two-chain-defect}, $r_{A} + d_{F} - 1 \geq d - 1$.
Since $C (a \vee d) k / n \leq \eta < 1$,
$\Big\{ \frac{C (a \vee d) k}{n} \Big\}^{r_{A} + d_{F} - 1} \leq
\Big\{ \frac{C (a \vee d) k}{n} \Big\}^{d - 1}$.
When $d = 1$, these factors are absorbed into the implicit constant. When $d \geq 2$, since $d \leq 2 (d - 1)$, they can be absorbed into $\{C (a \vee d) k / n\}^{d - 1}$ by enlarging the universal constant $C$.
Using $a r_{A} + r_{F}^{2} \leq C (a \vee d)^{2}$
and $(a \vee d)^{2} \leq a^{2} 4^{d - 1}$, we obtain
\begin{align}
V_{0} \lesssim  C^{2 q} k^{2 q - 2 a + 4}
\frac{C (a \vee d)^{2}}{n} \Big\{\frac{C (a \vee d) k}{n} \Big\}^{d - 1} \lesssim C^{2 q} k^{2 q - 2 a + 4}
\frac{C a^{2}}{n} \Big\{\frac{C (a \vee d) k}{n} \Big\}^{d - 1}.
\label{eq:finite-neumann-zero-overlap-bound}
\end{align}
Finally, combine \eqref{eq:finite-neumann-overlap-sum} and
\eqref{eq:finite-neumann-zero-overlap-bound}. For $d \geq 1$, $\{\frac{C (a \vee d) k}{n}\}^{d} \lesssim \{\frac{C (a \vee d) k}{n}\}^{d - 1}$, while $V_{0} = 0$ when $d = 0$.
\end{proof}


\subsubsection{Proof details related to \textbf{Step vi}}
\label{app:var vi}

We first introduce the notation used in this section. Fix $\iota \ge 1$ and $0 \le r \le \lfloor\iota / 2\rfloor$, and define $\bbB_{\iota, r} \coloneqq
\{ \calB = \{B_{1}, \cdots, B_{r}\}: B_{\nu} \subseteq [\iota],\
|B_{\nu}| \ge 2,\ B_{\nu} \cap B_{\nu'} = \emptyset
\text{ for } \nu \ne \nu' \}$. When $r = 0$, the collection $\bbB_{\iota, 0}$ contains only the empty collection. For $\calB \in \bbB_{\iota, r}$, define $ D (\calB) \coloneqq \sum_{\nu = 1}^{r} |B_{\nu}|$, with the convention $D (\emptyset) = 0$ and define $\calE_{0} (\calB) \coloneqq \{ \varepsilon \in \calE (\calB) :s_{\varepsilon} = 0\}$. Further, define $w_{\iota, r} \coloneqq \sum_{\calB \in \bbB_{\iota, r}} \prod_{B \in \calB}(|B| - 1) 2^{D (\calB)} 4^{\iota - D (\calB)}$. For each $\calB \in \bbB_{\iota, r}$, the part of the
decomposition in Lemma~\ref{lem:mobius-block-to-chain} corresponding to $s_{\varepsilon} = 0$ is
\begin{equation*}
c_{\calB, n} \bbU_{n, 2 + r} (K_{\calB, \emptyset}^{(J)}) = \sum_{\varepsilon \in \calE_{0} (\calB)} c_{\calB, n} d_{\calB, \varepsilon} \sum_{\bm{h} \in \{0, \cdots, J\}^{q_{\varepsilon} + 1}} T_{a_{\varepsilon}, q_{\varepsilon}, \gamma_{\varepsilon}, \bm{\ell}_{\varepsilon, \bm{h}}}.
\end{equation*}
The following lemma bounds the corresponding coefficients and the weights $w_{\iota,r}$.
\begin{lemma}
\label{lem:weighted-block-counting}
If $n \ge \iota + 2$, we have $$\sum_{\calB \in \bbB_{\iota, r}} \sum_{\varepsilon \in \calE_{0} (\calB)} |c_{\calB, n} d_{\calB, \varepsilon}| \le w_{\iota, r} \frac{(n - 2 - \iota) !}{(n - 2 - r) !}.$$ Furthermore, if $n \ge 2 (\iota + 2)$, then $$\sum_{\calB \in \bbB_{\iota, r}} \sum_{\varepsilon \in \calE_{0} (\calB)} |c_{\calB, n} d_{\calB, \varepsilon}| \le 2^{\iota - r} n^{- (\iota - r)} w_{\iota, r}.$$ Finally, $ w_{\iota, r} \le 4^{\iota} \frac{(r + 1)^{\iota}}{r !}$.
\end{lemma}

\begin{proof}
Fix $\calB \in \bbB_{\iota, r}$. By Lemma~\ref{lem:mobius}, $|c_{\calB, n}| = \prod_{B \in \calB} (|B| - 1) \frac{(n - 2 - \iota)!}{(n - 2 - r)!}$. We next account for the coefficients produced by expanding the kernel $K_{\calB, \emptyset}^{(J)}$. 
For each position covered by some $B \in \calB$, the corresponding factor is $H_{u}^{(J)} = X_{u}X_{u}^{\top} \hat{\Omega}_{J} - I$. The sum of the absolute values of the coefficients in this expansion is $2$. Hence, the $D (\calB)$ covered positions contribute the total absolute weight $2^{D (\calB)}$.

For each uncovered position, the corresponding factor is $\calR^{(J)}_{i_{1} i_{2}} = X_{i_{1}} X_{i_{1}}^{\top} \hat{\Omega}_{J} + X_{i_{2}} X_{i_{2}}^{\top} \hat{\Omega}_{J} - 2 I$. The sum of the absolute values of the scalar coefficients in this expansion is $1 + 1 + 2 = 4$. Hence, the $\iota - D (\calB)$ uncovered positions contribute a total absolute weight $4^{\iota - D (\calB)}$. The selections from $\hat{\Omega}_{J} = \sum_{h = 0}^{J} \Delta_{n}^{h}$ have coefficient one and form the inner sum over $\bm{h}$ in the statement. Therefore $\sum_{\varepsilon \in \calE_{0} (\calB)} |d_{\calB, \varepsilon}| \le 2^{D (\calB)} 4^{\iota - D (\calB)}$. Combining the M\"obius coefficient with the expansion coefficients:
\begin{align*}
\sum_{\varepsilon \in \calE_{0} (\calB)} |c_{\calB, n} d_{\calB, \varepsilon}| \le |c_{\calB, n}| \sum_{\varepsilon \in \calE_{0} (\calB)} |d_{\calB, \varepsilon}| \le \Big\{ \prod_{B \in \calB} (|B| - 1) \Big\} 2^{D (\calB)} 4^{\iota - D (\calB)} \frac{(n - 2 - \iota)!}{(n - 2 - r)!}.
\end{align*}
Summing over $\calB \in \bbB_{\iota, r}$ gives
\begin{align*}
\sum_{\calB \in \bbB_{\iota, r}} \sum_{\varepsilon \in \calE_{0} (\calB)} |c_{\calB, n} d_{\calB, \varepsilon}| \le \sum_{\calB \in \bbB_{\iota, r}} \Big\{ \prod_{B \in \calB} (|B| - 1) \Big\} 2^{D (\calB)} 4^{\iota - D (\calB)} \frac{(n - 2 - \iota)!}{(n - 2 - r)!} = w_{\iota, r} \frac{(n - 2 - \iota)!}{(n - 2 - r)!}.
\end{align*}

We next prove the simplified bound under $n \ge 2 (\iota + 2)$. Since $\frac{(n - 2 - r)!}{(n - 2 - \iota)!} = \prod_{s = 0}^{\iota - r - 1} (n - 2 - r - s)$ and $n - 2 - r - s \ge n - 1 - \iota \ge \frac{n}{2}$, for all $ s =  0, \cdots, \iota - r - 1$, we have $\frac{(n - 2 - r)!}{(n - 2 - \iota)!} \ge \Big( \frac{n}{2} \Big)^{\iota - r}$.  Equivalently, $\frac{(n - 2 - \iota)!}{(n - 2 - r)!} \le 2^{\iota - r} n^{-(\iota - r)}$. Thus $$\sum_{\calB \in \bbB_{\iota, r}} \sum_{\varepsilon \in \calE_{0} (\calB)} |c_{\calB, n} d_{\calB, \varepsilon}| \le 2^{\iota - r} n^{- (\iota - r)} w_{\iota, r}.$$ It remains to bound $w_{\iota, r}$. Since $d - 1 \le 2^{d}$ for every integer $d \ge 2$, we have $(|B| - 1) 2^{-|B|} \le 1$. Thus, for every $\calB \in \bbB_{\iota, r}$, $\prod_{B \in \calB}(|B| - 1) 2^{D (\calB)} 4^{\iota - D (\calB)} = 4^{\iota} \prod_{B \in \calB} \{(|B| - 1) 2^{-|B|}\} \le 4^{\iota}$, implying $w_{\iota, r} \le 4^{\iota} |\bbB_{\iota, r}|$.

Finally, we bound $|\bbB_{\iota, r}|$. When $r = 0$, $|\bbB_{\iota, 0}| = 1$, and hence $|\bbB_{\iota, 0}| = 1 = \frac{(0 + 1)^{\iota}}{0!}$. Now assume $r \ge 1$. Let $S_{\ge 2} (t, r)$ denote the number of partitions of a $t$-element set into $r$ nonempty subsets, each size at least two. Then 
\begin{align*}
|\bbB_{\iota, r}| = \sum_{t = 2 r}^{\iota} \binom{\iota}{t} S_{\ge 2}(t, r) \text{ and }  S_{\ge 2}(t, r) = \frac{1}{r!} \sum_{\substack{b_{1}, \cdots, b_{r} \ge 2 \\
b_{1} + \cdots + b_{r} = t}} \frac{t!}{b_{1}! \cdots b_{r}!}.
\end{align*}
Therefore,
\begin{align*}
|\bbB_{\iota, r}| = \sum_{t = 2 r}^{\iota} \binom{\iota}{t} \frac{1}{r!} \sum_{\substack{b_{1}, \cdots, b_{r} \ge 2 \\
b_{1} + \cdots + b_{r} = t}} \frac{t!}{b_{1}! \cdots b_{r}!} = \frac{\iota!}{r!} \sum_{\substack{b_{1}, \cdots, b_{r} \ge 2 \\
b_{1} + \cdots + b_{r} \le \iota}} \frac{1}{(\iota - b_{1} - \cdots - b_{r})! b_{1}! \cdots b_{r}!}.
\end{align*}
Relaxing the constraints $b_{1}, \cdots, b_{r} \ge 2$ to $b_{1}, \cdots, b_{r} \ge 0$, and writing $b_{0} = \iota - \sum_{j = 1}^{r} b_{j}$ for the uncovered positions, we obtain
\begin{align*}
|\bbB_{\iota, r}| & \le \frac{\iota!}{r!} \sum_{\substack{b_{0}, b_{1}, \cdots, b_{r} \ge 0 \\
b_{0} + b_{1} + \cdots + b_{r} = \iota}} \frac{1}{b_{0} ! b_{1} ! \cdots b_{r} !}  
= \frac{1}{r !} \sum_{\substack{b_{0}, b_{1}, \cdots, b_{r} \ge 0 \\
b_{0} + b_{1} + \cdots + b_{r} = \iota}}\frac{\iota!}{b_{0}! b_{1}! \cdots b_{r}!} = \frac{(r + 1)^{\iota}}{r!}.
\end{align*}
Combining this with $w_{\iota, r} \le 4^{\iota} |\bbB_{\iota, r}|$ yields $w_{\iota, r} \le 4^{\iota} \frac{(r + 1)^{\iota}}{r!}$.
\end{proof}

\begin{lemma}
\label{lem:block-level-weighted-tail}
Let $\iota \ge 1, \ r_{\iota}^{*} = \lfloor \iota / 2 \rfloor$, and $h_{\iota} = \iota - r_{\iota}^{*}$.  Let $\rho = k / n$, and let $w_{\iota, r}$ be the weighted combinatorial quantity defined in
Lemma~\ref{lem:weighted-block-counting}. There exists a universal constant $C > 0$ such that, if $C \iota \rho < 1$, then
\begin{equation*}
\sum_{r = 0}^{r_{\iota}^{*}} 2^{\iota - r} w_{\iota, r} \rho^{r_{\iota}^{*} - r}
\lesssim \frac{(C \iota)^{h_{\iota}}}{1 - C \iota \rho}.
\end{equation*}
Consequently, with $\iota = j - 2$ and
$h_{j} = \iota - r_{\iota}^{*} = \lfloor (j - 1) / 2 \rfloor$,
\begin{equation*}
\sum_{r = 0}^{r_{\iota}^{*}} 2^{\iota - r} w_{\iota, r}
\Big(\frac{k}{n}\Big)^{r_{\iota}^{*} - r}
\lesssim \frac{(C j)^{h_{j}}}{1 - C j k/n}.
\end{equation*}
\end{lemma}

\begin{proof}
By Lemma~\ref{lem:weighted-block-counting}, $w_{\iota, r} \le 4^{\iota} \frac{(r + 1)^{\iota}}{r !}$. Hence $2^{\iota - r} w_{\iota, r} \le 8^{\iota} \frac{(r + 1)^{\iota}}{r !} \eqqcolon \tilde{w}_{\iota, r}$. For $0 \le r \le r_{\iota}^{*}$, set $A_{r} \coloneqq \tilde{w}_{\iota, r} \rho^{r_{\iota}^{*} - r}$.
Then $$\frac{A_{r - 1}}{A_{r}} = \rho \frac{\tilde{w}_{\iota, r - 1}}{\tilde{w}_{\iota, r}} = \rho r \Big(\frac{r}{r + 1}\Big)^{\iota} \le \rho r \le \iota \rho.$$ Repeatedly applying $A_{r - 1} \le \iota \rho A_{r}$ for $1 \le r \le r_{\iota}^{*}$ gives $ A_{r_{\iota}^{*} - s} \le (\iota \rho)^{s} A_{r_{\iota}^{*}}, \ 0 \le s \le r_{\iota}^{*}$. Thus, under the assumption $C \iota \rho < 1$ with $C$ sufficiently large,
\begin{equation*}
\sum_{r = 0}^{r_{\iota}^{*}} A_{r} \le A_{r_{\iota}^{*}} \sum_{s = 0}^{r_{\iota}^{*}} (\iota \rho)^{s}
\le \frac{A_{r_{\iota}^{*}}}{1 - \iota \rho}\lesssim \frac{\tilde{w}_{\iota, r_{\iota}^{*}}}{1 - C \iota \rho}.
\end{equation*}

It remains to bound the terminal weight $\tilde{w}_{\iota, r_{\iota}^{*}}$ at the maximal level. Let $r_{*}\coloneqq r_{\iota}^{*}$ and $h_{\iota} = \iota - r_{*}$. If $r_{*} = 0$, then $\iota = 1$ and the desired bound follows after increasing $C$. We therefore assume $r_{*} \ge 1$. By Stirling's lower bound, $r_{*} ! \ge c \Big(\frac{r_{*}}{e}\Big)^{r_{*}}$, and hence $\tilde{w}_{\iota, r_{*}} = 8^{\iota} \frac{(r_{*} + 1)^{\iota}}{r_{*} !}\le C 8^{\iota} e^{r_{*}} \frac{(r_{*} + 1)^{\iota}}{r_{*}^{r_{*}}} = C 8^{\iota} e^{r_{*}} (r_{*} + 1)^{\iota - r_{*}} \Big(\frac{r_{*} + 1}{r_{*}}\Big)^{r_{*}}$. Since $\Big( 1 +\frac1{r_{*}}\Big)^{r_{*}} \le e,
\ r_{*} + 1 \le \iota + 1 \le 2 \iota$, we get
$ \tilde{w}_{\iota, r_{*}} \le C 8^{\iota} e^{r_{*}} (2 \iota)^{h_{\iota}}$. Moreover, since $r_{*} = \lfloor \iota / 2 \rfloor$ and $h_{\iota} = \iota - r_{*}$, we have $ \iota \le 2 h_{\iota}, \ r_{*} \le h_{\iota}$. Therefore $ 8^{\iota} e^{r_{*}} \le 8^{2 h_{\iota}} e^{h_{\iota}} = (64 e)^{h_{\iota}}$. Consequently, $\tilde{w}_{\iota, r_{\iota}^{*}}
\le (C \iota)^{h_{\iota}}$. Combining this with the preceding geometric-tail estimate gives $\sum_{r = 0}^{r_{\iota}^{*}} 2^{\iota - r} w_{\iota, r} \rho^{r_{\iota}^{*} - r} \lesssim \frac{(C \iota)^{h_{\iota}}}{1 - C \iota \rho}$. The stated result follows by taking $\iota = j - 2$ and using $\iota \le j$ and $h_{\iota} = h_{j} = \lfloor (j - 1) / 2\rfloor$.
\end{proof}

Armed with Lemma~\ref{lem:weighted-block-counting} and Lemma~\ref{lem:block-level-weighted-tail}, we are ready to prove Lemma~\ref{lem:mobius-r-variance}.

\begin{proof}[Proof of Lemma~\ref{lem:mobius-r-variance}]
Fix $j \ge 3$, set $\iota = j - 2$, and fix $0 \le r \le r_{\iota}^{*}$. By Lemma~\ref{lem:mobius-block-to-chain}, for each $\calB \in \bbB_{\iota,r}$, after expanding $H^{(J)}_{i} = X_{i} X_{i}^{\top} \hat{\Omega}_{J} - I$ and $\calR^{(J)}_{i_{1} i_{2}} = X_{i_{1}} X_{i_{1}}^{\top} \hat{\Omega}_{J} + X_{i_{2}} X_{i_{2}}^{\top} \hat{\Omega}_{J} - 2 I$, the term $\bbU_{n, 2 + r} (K_{\calB, \emptyset}^{(J)})$ can be written as a finite linear combination of chain statistics $T_{a_{\varepsilon}, q_{\varepsilon}, \gamma_{\varepsilon}, \bm{\ell}_{\varepsilon, \bm{h}}}$. Thus we may write
\begin{equation*}
c_{\calB, n} \bbU_{n, 2 + r} (K_{\calB, \emptyset}^{(J)}) = \sum_{\varepsilon \in \calE_{0} (\calB)} c_{\calB, n} d_{\calB, \varepsilon} \sum_{\bm{h} \in \{0, \cdots, J\}^{q_{\varepsilon} + 1}} T_{a_{\varepsilon}, q_{\varepsilon}, \gamma_{\varepsilon}, \bm{\ell}_{\varepsilon, \bm{h}}},
\end{equation*}
where the coefficients $c_{\calB, n} d_{\calB, \varepsilon}$ include the M\"obius coefficient, the $U$-statistic normalization, and the coefficients arising from the expansion of the $H^{(J)}$- and $\calR^{(J)}$-factors. Moreover, Lemma~\ref{lem:mobius-block-to-chain} gives $q_{\varepsilon} - b_{\varepsilon} \le \iota - r$.
For fixed $\varepsilon \in \calE_{0} (\calB)$, write $a \coloneqq a_{\varepsilon}$, $q \coloneqq q_{\varepsilon}$,
$p \coloneqq q + 1$, and $\rho \coloneqq k / n$, and define $\bar{T}_{\varepsilon}^{(J)} \coloneqq
\sum_{\bm{h} \in \{0, \cdots, J\}^{p}}
T_{a, q, \gamma_{\varepsilon}, \bm{h}}$.
Since $s_{\varepsilon} = 0$,
$\bm{\ell}_{\varepsilon, \bm{h}} = \bm{h}$. For
$\bm{h}, \bm{h'} \in \{0, \cdots, J\}^{p}$, define
\begin{equation*}
|\bm{h}| \coloneqq \sum_{u = 0}^{q} h_{u},
\ d_{L} \coloneqq |\bm{h}|, \ d_{R} \coloneqq |\bm{h'}|,
\ D \coloneqq d_{L} + d_{R}, \ x_{D} \coloneqq C \{j \vee D\}\rho, \ \calK_{\varepsilon} \coloneqq C^{2 q} k^{2 q -2 a + 4}.
\end{equation*}
Apply the expansion of the covariance and the graph calculation
in the proof of Lemma~\ref{lemma:generic-chain-variance-explicit}. For a fixed pair $(\bm{h}, \bm{h'})$ and $0 \leq \alpha \leq a$, define
\begin{align*}
V_{\alpha} (\bm{h}, \bm {h'}) \coloneqq \frac{\{(n - a) !\}^{2}}{(n !)^{2}} \sum_{\substack{ \bfi, \bfi' \in \calI_{n, a}\\ \alpha(\bfi, \bfi') = \alpha}}
| \cov \{k_{a, q, \gamma_{\varepsilon}, \bm{h}} (\bfi), k_{a, q, \gamma_{\varepsilon}, \bm{h'}} (\bfi') \} |.
\end{align*}

For $\alpha \geq 1$, the argument in the proof of
Lemma~\ref{lemma:generic-chain-variance-explicit},
applied with total degree $D = d_{L} + d_{R}$,
gives the factor $x_{D}^{\lceil D / 2\rceil}$. For $\alpha = 0$, the proof of Lemma~\ref{lemma:generic-chain-variance-explicit}
distinguishes two cases. Recall that
$r_{A} + r_{F} = D$ and $d_{F} \geq \lceil r_{F} / 2\rceil$. In the first, choose a $U_{S}$ position in one kernel and a displayed index in the other; there are at most $a d_{L} + a d_{R} = a D$ pairs.
After separating the selected factor $n^{- 1}$,
the remaining exponent is $(r_{A} - 1) + d_{F}$,
which is at least $\lceil (D - 1) / 2\rceil$
because $r_{A} \geq 1$. In the second, choose one $V_{S}$ position in each kernel and require the corresponding summation indices to be equal; there are at most $d_{L} d_{R}$ pairs.
After separating the resulting factor $n^{-1}$,
the remaining exponent is $r_{A} + d_{F} - 1$,
which is at least $\lceil (D - 2) / 2\rceil$. Combining these bounds and keeping the two contributions for $\alpha = 0$ separate, we obtain, for each fixed pair $(\bm{h}, \bm{h'})$,
\begin{align}
V_{\alpha} (\bm{h}, \bm{h'}) & \lesssim \calK_{\varepsilon} \frac{k^{\alpha - 1}}{\alpha !}
\Big(\frac{C a^{2}}{n} \Big)^{\alpha}
x_{D}^{\lceil D / 2\rceil}, & 1 \leq \alpha \leq a, \label{eq:positive-overlap}\\
V_{0} (\bm{h}, \bm{h'}) & \lesssim
\calK_{\varepsilon} \Big\{ \frac{C a D}{n} x_{D}^{\lceil (D - 1) / 2\rceil} + \frac{C d_{L} d_{R}}{n} x_{D}^{\lceil (D - 2) / 2\rceil} \Big\}, & D \geq 1. \label{eq:zero-overlap}
\end{align}

For $D = 0$, the two kernels with disjoint displayed
index sets depend on disjoint sets of observations,
and hence $V_{0} (\bm{h}, \bm{h'}) = 0$.
We next sum the bounds for $V_{\alpha} (\bm{h}, \bm{h'})$, $0 \leq \alpha \leq a$, over $\bm{h}$ and $\bm{h'}$.
Recall that, in the expansion defining
$\bar{T}_{\varepsilon}^{(J)}$, each summand is obtained
by selecting a power of $\Delta_{n}$ from each of the
$p = q + 1$ occurrences of $\hat{\Omega}_{J} = \sum_{h = 0}^{J} \Delta_{n}^{h}$. These choices are indexed by
$\bm{h} = (h_{0}, \cdots, h_{q})$, where $h_{u}$ records
the power selected from the $(u + 1)$st occurrence. We group the pairs $(\bm{h}, \bm{h'})$ by their total degree $D = d_{L} + d_{R}$ and sum the preceding bounds over all pairs with the same total degree.
For $d \geq 0$, let
$c_{p, J} (d) \coloneqq | \{ \bm{h} \in \{0, \cdots, J\}^{p}: |\bm{h}| = d \} |$. For $D \geq 0$, define
\begin{align*}
\calN_{0} (D) \coloneqq \sum_{d_{L} + d_{R} = D} c_{p, J} (d_{L}) c_{p,J} (d_{R}), \ \calN_{1} (D) \coloneqq D \calN_{0} (D), \ \calN_{2} (D) \coloneqq \sum_{d_{L} + d_{R} = D} d_{L} d_{R} c_{p, J} (d_{L}) c_{p, J} (d_{R}).
\end{align*}

For fixed $d_{L}$ and $d_{R}$, there are exactly
$c_{p, J} (d_{L}) c_{p, J}(d_{R})$ pairs $(\bm{h},\bm{h'})$ satisfying $|\bm{h}| = d_{L}$ and $|\bm{h'}| = d_{R}$. Thus, for $D \geq 1$, summing the bound for $V_{0}$ over all such pairs with $d_{L} + d_{R} = D$ gives
\begin{align*}
\sum_{\substack{ \bm{h},\bm{h'} \in \{0, \cdots, J\}^{p}\\ |\bm{h}| + |\bm{h'}| = D}} V_{0} (\bm{h}, \bm{h'}) \lesssim \calK_{\varepsilon} \Big\{\frac{C a}{n} \calN_{1} (D) x_{D}^{\lceil (D - 1) / 2\rceil} + \frac{C}{n}\calN_{2} (D) x_{D}^{\lceil(D - 2) / 2\rceil} \Big\}.
\end{align*}
To bound these quantities, let
$c_{p} (d) \coloneqq \binom{d + p - 1}{d}$.
Since $c_{p, J} (d) \leq c_{p} (d)$,
$\sum_{d = 0}^{\infty} c_{p} (d) z^{d} = (1 - z)^{-p}$ and the coefficient of $z^{D}$ in the square of the series on the left is $\sum_{d_{L} + d_{R} = D} c_{p} (d_{L}) c_{p} (d_{R})$, whereas the corresponding coefficient in $(1 - z)^{- 2 p}$ is $\binom{D + 2 p - 1}{D}$, therefore,
\begin{align*}
\calN_{0} (D) & \leq \sum_{d_{L} + d_{R} = D} c_{p} (d_{L}) c_{p} (d_{R}) = \binom{D + 2 p - 1}{D}, \\
\calN_{1} (D) & = D \calN_{0} (D) \leq
D \binom{D + 2 p - 1}{D} = 2 p \binom{D + 2 p - 1}{D - 1}.
\end{align*}
Moreover, $$\sum_{d = 0}^{\infty} d \, c_{p} (d) z^{d} = z \frac{d}{d z}(1 - z)^{- p} = p z (1 - z)^{- p - 1}.$$
The coefficient of $z^{D}$ in the square of this series is $p^{2} \binom{D + 2 p - 1}{D - 2}$. Since $c_{p, J} (d) \leq c_{p}(d)$, it follows that $\calN_{2} (D) \leq p^{2} \binom{D + 2 p - 1}{D - 2}$. For $\nu = 0 , 1 , 2$, $\calN_{\nu} (D) \leq C p^{\nu} \binom{D + 2 p - 1}{D - \nu}$, where the binomial coefficient is zero when $D < \nu$. Let $$\calS_{\nu} \coloneqq \sum_{D = \nu}^{2 p J} \calN_{\nu} (D) x_{D}^{\lceil (D - \nu) / 2\rceil}$$ for $\nu = 0, 1, 2$. Writing $t = D - \nu$, we split the sum at $t = 2 j$. Recall that $p \leq j - 1$ and $\nu \leq 2$. If $t \leq 2 j$, then $D + 2 p - 1 \leq C j$, and hence $$\binom{D + 2 p - 1}{D - \nu} = \binom{t + \nu + 2 p - 1}{t} \leq \frac{(C j)^{t}}{t !}.$$ Moreover, $D \leq 2 j + 2$, so $x_{D} \leq C j \rho$. If $t > 2 j$, then $\nu + 2 p - 1 \leq 2 j - 1 < t$, and therefore $\binom{t + \nu + 2 p - 1}{t} \leq 2^{t + \nu + 2 p - 1} \leq 4^{t}$. Also, $D \leq 2 p J \leq 2 j J$, so $x_{D} \leq C j J \rho$, hence $x_{D}^{\lceil t / 2\rceil} \leq
x_{D}^{t / 2} \leq (C j J \rho)^{t / 2}$.
Consequently, after enlarging the universal constant,
\begin{equation*}
\calN_{\nu} (D) x_{D}^{\lceil t / 2\rceil}
\leq C p^{\nu}
\begin{cases}
\displaystyle
\frac{(C j)^{t}}{t !} (C j \rho)^{\lceil t / 2\rceil}, & t \leq 2 j,\\[1mm]
\{C \sqrt{j J \rho}\}^{t}, & t > 2 j.
\end{cases}
\end{equation*}
Separating the even and odd terms in the sum over $0 \leq t \leq 2 j$ gives
\begin{equation*}
\sum_{0 \leq t \leq 2 j} \frac{(C j)^{t}}{t !}
(C j \rho)^{\lceil t / 2\rceil} \lesssim
\sum_{t = 0}^{\infty} \frac{\{C j \sqrt{j \rho}\}^{t}}{t !} = \exp\{C j \sqrt{j \rho}\}.
\end{equation*}
For $t > 2 j$, the sum is bounded by the geometric series $\sum_{t > 2 j}\{C \sqrt{j J \rho}\}^{t} \lesssim 1$. Therefore, $\calS_{\nu} \lesssim p^{\nu} \exp \{ C j \sqrt{j \rho} \}$ for $\nu = 0, 1, 2$. Summing \eqref{eq:positive-overlap} over $1 \leq \alpha \leq a$ and all $(\bm{h}, \bm{h'})$, and summing the two terms in \eqref{eq:zero-overlap} over all such pairs, yields
\begin{align*}
\var\{ \bar{T}_{\varepsilon}^{(J)}\} & = \sum_{\bm{h}, \bm{h'} \in \{0, \cdots, J\}^{p}} \operatorname{Cov} ( T_{a, q, \gamma_{\varepsilon}, \bm{h}}, T_{a, q, \gamma_{\varepsilon}, \bm{h'}} ) \leq \sum_{\bm{h}, \bm{h'} \in \{0, \cdots, J\}^{p}} \sum_{\alpha = 0}^{a} V_{\alpha} (\bm{h}, \bm{h'}) \\
& \lesssim \calK_{\varepsilon} \Big( \frac{C a^{2}}{n} \exp (C a^{2}\rho)\calS_{0} + \frac{C a}{n}\calS_{1} + \frac{C}{n}\calS_{2} \Big)  \lesssim \frac{C j^{2}}{n} C^{2 q} k^{2 q - 2 a + 4} \exp \{ C j \sqrt{j \rho} \}.
\end{align*}
In the last line we used $a, p \leq j$ and $a^{2} \rho \leq j \sqrt{j \rho}$, which follows from $j \rho < 1$. 

We now use the preceding variance bound for
$\bar{T}_{\varepsilon}^{(J)}$ to bound
$\var (Z_{\iota, r})$. Recall that
$$Z_{\iota, r} = \sum_{\calB \in \bbB_{\iota, r}} \sum_{\varepsilon \in \calE_{0} (\calB)}
c_{\calB, n} d_{\calB, \varepsilon}
\bar{T}_{\varepsilon}^{(J)}.$$
Since $a = 2 + b$ and $q - b \leq \iota - r$,
we have $q - a + 2 \leq \iota - r$.
Moreover, $q \leq \iota \leq 2 (\iota - r)$.
Thus the preceding bound gives, uniformly over
$\calB \in \bbB_{\iota, r}$ and
$\varepsilon \in \calE_{0} (\calB)$,
\begin{equation*}
\var^{1 / 2} \{\bar{T}_{\varepsilon}^{(J)}
\} \lesssim \frac{j}{\sqrt{n}}
(C k)^{\iota - r} \exp \Big\{ C j \sqrt{\frac{j k}{n}} \Big\}.
\end{equation*}
By Lemma~\ref{lem:weighted-block-counting}, since $n \ge 2 (\iota + 2)$,
\begin{equation*}
\sum_{\calB \in \bbB_{\iota, r}}
\sum_{\varepsilon \in \calE_{0} (\calB)}
|c_{\calB, n} d_{\calB, \varepsilon}| \le 2^{\iota - r} w_{\iota, r} n^{- (\iota - r)}.
\end{equation*}
Applying Minkowski's inequality and combining these
two bounds, we obtain
\begin{equation*}
\var^{1/2} (Z_{\iota,r}) \lesssim \frac{ j}{\sqrt{n}} 2^{\iota - r} w_{\iota, r} \Big( \frac{C k}{n} \Big)^{\iota - r} \exp \Big\{ C j \sqrt{\frac{j k}{n}} \Big\}.
\end{equation*}
Equivalently, $$\var (Z_{\iota, r}) \lesssim 2^{2 (\iota - r)} w_{\iota, r}^{2} \exp \Big\{C j \sqrt{\frac{j k}{n}} \Big\}\frac{ j^{2} k^{2 (\iota - r)}}{n^{2 (\iota - r) + 1}}.$$
\end{proof}

Applying Minkowski's inequality to
$M_{j, J} = \sum_{r = 0}^{\lfloor \iota / 2\rfloor} Z_{\iota, r}$ with Lemma~\ref{lem:mobius-r-variance} and Lemma~\ref{lem:block-level-weighted-tail} gives
\begin{equation}
\label{eq:finite-neumann-main-part-simplified}
\var (M_{j, J}) \lesssim \frac{C j^{2}}{n} \exp \Big\{ C j \sqrt{\frac{j k}{n}} \Big\} \Big( \frac{C j k}{n} \Big)^{2 h_{j}}.
\end{equation}
It remains to control $L_{j, J}$. 
Fix $j \geq 3$ and $J \geq 1$, and define
$D_{j, J} \coloneqq (j - 1) J + j - 2$, $N_{j} (d) \coloneqq
\binom{d + 2 j - 4}{2 j - 4}$. For $J + 1 \leq d \leq D_{j, J}$ and each $1 \leq s \leq j - 2$, define
\begin{equation*}
\mathfrak{I}_{j, d, s} \coloneqq \{ (\calB, \varepsilon, \bm{h}): \calB \in \bbB_{j - 2},\ \varepsilon \in \calE (\calB),\ \bm{h} \in \{0, \cdots, J\}^{q_{\varepsilon} + 1},\ s_{\varepsilon} = s,\ d (\bm{\ell}_{\varepsilon, \bm{h}}) = d \}.
\end{equation*}
For $\xi = (\calB, \varepsilon, \bm{h})
\in \mathfrak{I}_{j, d, s}$, set $\beta_{\xi, n} \coloneqq c_{\calB, n} d_{\calB, \varepsilon}$, $a_{\xi} \coloneqq a_{\varepsilon}$, $q_{\xi} \coloneqq q_{\varepsilon}$, $T_{\xi} \coloneqq T_{a_{\varepsilon}, q_{\varepsilon}, \gamma_{\varepsilon}, \bm{\ell}_{\varepsilon, \bm{h}}}$. By Lemma~\ref{lem:mobius-block-to-chain}, 
\begin{equation}
\label{eq:finite-neumann-defect-decomposition} 
L_{j, J} = \sum_{d = J + 1}^{D_{j, J}}
\sum_{s = 1}^{j - 2} \sum_{\xi \in \mathfrak{I}_{j, d, s}}
\beta_{\xi, n} T_{\xi}.
\end{equation}
The following lemma bounds the sum of the coefficients of the terms in $L_{j, J}$.

\begin{lemma}
\label{lem:finite-neumann-defect-coefficient-counting}
Suppose that $n \geq 2 j$ and $k \leq n$. 
Then, for every $J + 1 \leq d \leq D_{j, J}$,
\begin{equation}
\label{eq:finite-neumann-defect-coefficient-sum}
\sum_{s = 1}^{j - 2} \sum_{\xi \in \frakT_{j, d, s}^{E}}
|\beta_{\xi,n}| C^{q_{\xi}} k^{q_{\xi} - a_{\xi} + 2}
\leq (C j)^{j} N_{j} (d).
\end{equation}
\end{lemma}

\begin{proof}
Fix $\calB \in \bbB_{j - 2}$ and write 
$D_{\calB} \coloneqq \sum_{B \in \calB} |B|$,
$u_{\calB} \coloneqq j - 2 - D_{\calB}$, $\delta_{\calB}
\coloneqq j - 2 - |\calB|$.
Thus, $u_{\calB}$ is the number of positions not covered by $\calB$. If $s$ of these positions contribute $n E_{J}$, then, for $n \geq 2 j$, $$|c_{\calB, n}| n^{s} \leq 2^{\delta_{\calB}}
n^{- (\delta_{\calB} - s)} \prod_{B \in \calB} (|B| - 1).$$ Summing over the choices of the $s$ positions contributing $n E_{J}$ and expanding the remaining matrix factors,
with a factor $C$ for each selected term
$X_{i} X_{i}^{\top} \hat{\Omega}_{J}$,
gives an absolute coefficient sum of at most
$$\binom{u_{\calB}}{s} (2 C)^{D_{\calB}} (4 C)^{u_{\calB} - s}.$$ Here $\binom{u_{\calB}}{s}$ counts the choices of the
$s$ positions contributing $n E_{J}$. The factors $2 C$ and $4 C$ bound the absolute coefficient sums $1 + C$ and $2 + 2 C$ in the expansions of $H_{i}^{(J)}$ and $\calR_{i_{1} i_{2}}^{(J)}$,
respectively, with $n^{s}$ already accounted for.

By Lemma~\ref{lem:mobius-block-to-chain}, every resulting $U$-statistic satisfies $q_{\varepsilon} - a_{\varepsilon} + 2 \le \delta_{\calB} - s$. Since $k \leq n$ for all sufficiently large $n$, $n^{- (\delta_{\calB} - s)} k^{q_{\xi} - a_{\xi} + 2} \leq 1$. Once $\calB$, the $s$ positions contributing $n E_{J}$, and the matrix selections are fixed, these $s$ positions contribute total degree $s (J + 1)$. The remaining degree $d - s (J + 1)$ is the sum of at most $j - 1$ nonnegative powers selected from the expansions of $\hat{\Omega}_{J}$. Hence the number of possible choices is at most $$\binom{d - s (J + 1) + j - 2}{j - 2} \leq \binom{d + j - 2}{j - 2} \leq N_{j} (d).$$ Finally, $\prod_{B \in \calB} (|B| - 1) \leq 2^{D_{\calB}}$, the number of possible collections $\calB$ is at most $(j + 1)^{j - 2}$, and $$\sum_{s = 1}^{u_{\calB}} \binom{u_{\calB}}{s} (4 C)^{u_{\calB} - s} \leq (1 + 4 C)^{u_{\calB}}.$$ Summing the preceding bounds over $\calB$, $s$, the matrix selections, and $\bm{h}$ proves \eqref{eq:finite-neumann-defect-coefficient-sum}.
\end{proof}

We now combine \eqref{eq:finite-neumann-defect-decomposition} with Lemma~\ref{lem:finite-neumann-defect-coefficient-counting} and the variance bound in Lemma~\ref{lemma:generic-chain-variance-explicit}. Every term in $L_{j, J}$ has $d \geq s (J + 1) \geq J + 1$. Since $a_{\xi} \leq j$ and $d \leq D_{j, J} \leq 2 j J$,
the conditions $n \geq 2 j$ and
$C j J k / n \leq \eta < 1$, with $C$ sufficiently large,
ensure that Lemma~\ref{lemma:generic-chain-variance-explicit}
applies uniformly to these terms. Under $J \geq C_{2} j$, the exponential factor in Lemma~\ref{lemma:generic-chain-variance-explicit}
is uniformly bounded, since
$a_{\xi}^{2} k / n \leq (j / J)(j J k / n)$
and $C j J k / n \leq \eta < 1$.
Taking square roots of its variance bound and using
$a_{\xi} \leq j$, we obtain $\var^{1 / 2} (T_{\xi}) \lesssim \frac{C j}{\sqrt{n}} C^{q_{\xi}} k^{q_{\xi} - a_{\xi} + 2} \Big\{ \frac{C (j \vee d) k}{n} \Big\}^{(d - 1) / 2}$. Applying Minkowski's inequality and combining the preceding bound with Lemma~\ref{lem:finite-neumann-defect-coefficient-counting}, we obtain
\begin{align*}
& \var^{1 / 2} (L_{j, J}) \leq \sum_{d = J + 1}^{D_{j, J}} \sum_{s = 1}^{j - 2} \sum_{\xi \in \frakT_{j, d, s}^{E}} |\beta_{\xi, n}| \var^{1 / 2} (T_{\xi}) \\
& \lesssim \frac{C j}{\sqrt{n}} \sum_{d = J + 1}^{D_{j, J}} \Big\{ \frac{C (j \vee d) k}{n} \Big\}^{(d - 1) / 2} \sum_{s = 1}^{j - 2} \sum_{\xi \in \frakT_{j, d, s}^{E}} |\beta_{\xi, n}| C^{q_{\xi}} k^{q_{\xi} - a_{\xi} + 2} \\
& \lesssim \frac{C j}{\sqrt{n}} (C j)^{j} \sum_{d = J + 1}^{D_{j, J}} N_{j} (d) \Big\{ \frac{C (j \vee d) k}{n} \Big\}^{(d - 1) / 2}.
\end{align*}
Squaring both sides yields
\begin{align}
\label{eq:finite-neumann-defect-part-bound}
\var (L_{j, J}) \lesssim \frac{C j^{2}}{n} (C j)^{2 j} \Big[ \sum_{d = J + 1}^{D_{j, J}} N_{j} (d) \Big\{ \frac{C (j \vee d) k}{n} \Big\}^{(d - 1) / 2} \Big]^{2}.
\end{align}
We now complete the proof of Lemma~\ref{lem:fixed-j-ifhat-variance}.

\begin{proof}[Proof of Lemma~\ref{lem:fixed-j-ifhat-variance}]
Fix $j \ge 3$ and set $\iota = j - 2$. Recall that $r_{\iota}^{*} = \lfloor \frac{\iota}{2} \rfloor$ and $h_{j} = \iota - r_{\iota}^{*} = \lfloor \frac{j - 1}{2} \rfloor$. By the definition of $Z_{\iota, r}$, the M\"obius-inversion expansion becomes $ \hat{\IIFF}^{(J)}_{j, j, k} (\Omega) = \sum_{r = 0}^{r_{\iota}^{*}} Z_{\iota, r} + L_{j, J}$.

Applying Minkowski's inequality to
$\hat{\IIFF}_{j , j , k}^{(J)} (\Omega)
= M_{j , J} + L_{j , J}$,
and using \eqref{eq:finite-neumann-main-part-simplified} and
\eqref{eq:finite-neumann-defect-part-bound}, we obtain
\begin{align*}
\var^{1 / 2} \{\hat{\IIFF}_{j, j, k}^{(J)} (\Omega)\} \lesssim \frac{C j}{\sqrt{n}}
\Big\{ \exp \Big\{ C j \sqrt{\frac{j k}{n}} \Big\} (C j \rho)^{h_{j}}  + (C j)^{j} \sum_{d = J + 1}^{D_{j, J}}
N_{j} (d) \{ C (j \vee d) \rho \}^{(d - 1) / 2} \Big\}.
\end{align*}
Recall that $\rho = k / n$, for $J + 1 \leq d \leq D_{j, J}$, let $t_{d} \coloneqq 4^{d} \{C (j \vee d) \rho\}^{(d - 1) / 2}$. Since $J \geq C_{2} j$ with $C_{2} \geq 2$, we have $J + 1 > 2 j$, and hence
$j \vee d = d$ and $N_{j} (d) \leq 4^{d}$ throughout the sum. Moreover, $D_{j, J} \leq 2 j J$, and $t_{d + 1} / t_{d} \leq C \sqrt{D_{j, J} \rho} \leq \sqrt{\eta}$, provided that the constant in the condition $C j J \rho \leq \eta < 1$ is sufficiently large. Then $\sum_{d = J + 1}^{D_{j, J}} t_{d} \lesssim 4^{J + 1}\{C (J + 1) \rho\}^{J/2}$. Since $h_{j} \leq j / 2$ and $J \geq C_{2} j$, one has $J / 2 > h_{j}$. The preceding bound yields
\begin{align*}
(C j)^{j} \sum_{d = J + 1}^{D_{j, J}} N_{j} (d) \{ C (j \vee d) \rho \}^{(d - 1) / 2} 
\lesssim (C j \rho)^{h_{j}} j^{- (J - 3 j ) / 2} \exp\{- \tilde c_{\eta} J + C_{\eta} j\}.
\end{align*}

For clarity, write the condition $C j J \rho \leq \eta$
as $C_{1} j J \rho \leq \eta$, where $C_{1}$ is
sufficiently large. The first inequality uses
$\rho^{J / 2 - h_{j}} \leq \{\eta / (C_{1} j J)\}^{J / 2 - h_{j}}$ and $(1 + 1 / J)^{J / 2} \leq e^{1 / 2}$.
For the second inequality, we use $h_{j} \leq j / 2$ and
$J^{h_{j}} \leq j^{h_{j}} \exp(J / 2)$. Choosing $C_{1}$ sufficiently large ensures that $\tilde c_{\eta} > 0$. For sufficiently large $C_{2}$, the condition $J \geq C_{2} j$ ensures that both
$j^{- (J - 3 j) / 2}$ and $\exp\{- \tilde c_{\eta} J + C_{\eta} j\}$
are bounded by one. Hence, $(C j)^{j} \sum_{d = J +1}^{D_{j, J}} N_{j} (d) \Big\{ \frac{C (j \vee d) k}{n} \Big\}^{(d - 1) / 2} \lesssim (C j \rho)^{h_{j}}$. By Lemma~\ref{lem:block-level-weighted-tail},
$\sum_{r = 0}^{r_{\iota}^{*}} 2^{\iota - r} w_{\iota, r} (C \rho)^{\iota - r} \lesssim (C j \rho)^{h_{j}}$.
Together with the preceding bound for the $L_{j, J}$ term, this yields $\var^{1/2} \{\hat{\IIFF}_{j , j , k}^{(J)} (\Omega)\} \lesssim \frac{C j}{\sqrt{n}} \exp \Big( C j \sqrt{\frac{j k}{n}} \Big) (C j \rho)^{h_{j}}$.
\end{proof}

\begin{lemma}
\label{lem:second-order-correction-variance}
Assume that $C J \rho \le \eta < 1$ for some constant $C$. Then $\var \{\hat{\IIFF}^{(J)}_{2, 2, k} (\Omega)\} \lesssim \frac{1}{n} + \frac{k}{n^{2}}$.
\end{lemma}

\begin{proof}
Recall that $$\hat{\IIFF}^{(J)}_{2, 2, k} (\Omega) = \bbU_{n, 2} (A_{1} X_{1}^{\top} \hat{\Omega}_{J} X_{2} Y_{2}) = \sum_{d = 0}^{J} \bbU_{n, 2} (A_{1} X_{1}^{\top} \Delta_{n}^{d} X_{2} Y_{2}).$$ For $0 \leq d \leq J$, let $T_{d} \coloneqq \bbU_{n, 2} (A_{1} X_{1}^{\top} \Delta_{n}^{d} X_{2} Y_{2})$. When $d = 0$, by Lemma~\ref{lem:chain-covariance-decomposition} and the calculation in the proof of Lemma~\ref{lemma:generic-chain-variance-explicit}, applied with $a = 2$, $q = 0$, and $d = 0$, $V_{0} = 0$, $V_{1} \lesssim \frac{1}{n}$, $V_{2} \lesssim \frac{k}{n^{2}}$, and $\var (T_{0}) \leq V_{0} + V_{1} + V_{2} \lesssim  \frac{1}{n} + \frac{k}{n^{2}}$. For $1 \leq d \leq J$, applying Lemma~\ref{lemma:generic-chain-variance-explicit} with $a = 2$, $q = 0$,
and $\bm{\ell} = (d)$ yields $$\var^{1 / 2} (T_{d}) \lesssim \frac{1}{\sqrt{n}} \Big\{ \frac{C (2 \vee d) k}{n} \Big\}^{(d - 1) / 2}$$ for $1 \leq d \leq J$. By Minkowski's inequality, $$\var^{1 / 2} \Big( \sum_{d = 1}^{J} T_{d} \Big) \leq \sum_{d = 1}^{J} \var^{1 / 2} (T_{d}) \lesssim \frac{1}{\sqrt{n}} \Big\{ 1 + \sum_{d = 2}^{J} \Big( \frac{C d k}{n} \Big)^{(d - 1) / 2} \Big\} \lesssim \frac{1}{\sqrt{n}}.$$ Combining this with the bound for $T_{0}$ and Minkowski's inequality completes the proof.
\end{proof}

\section{Technical Lemma}
\label{app:lemma}

\subsection{Technical lemma related to concentration inequalities}
\label{app:concentration}



A central technical tool in our proof is matrix Bernstein's inequality \citep{tropp2015introduction}.

\begin{lemma}[Matrix Bernstein's inequality]
\label{lem:matrix Bernstein}
Given a sequence $\{W_{i}\}_{i = 1}^{n}$ of independent and symmetric random matrices with dimension $k$. Assume that each matrix satisfies: $\bbE W_{i} = 0$ and $\lambda_{\max} (W_{i}) \lesssim k$ almost surely. Let $S_{n} \coloneqq \sum_{i = 1}^{n} W_{i}$. Then for all $t \geq 0$, $\bbP (\lambda_{\max} (S_{n}) \geq t) \leq k \cdot \exp \Big( - \frac{t^{2} / 2}{\nu^{2} + k t / 3} \Big)$, where $\nu^{2} = \Big\Vert \sum_{i = 1}^{n} \bbE [W_{i}^{2}] \Big\Vert_{\op}$. In particular, the following also hold: $\Big\Vert \frac{1}{n} S_{n} \Big\Vert_{\op} = O_{\bbP} \Big( \sqrt{\frac{k \log k}{n}} \Big) \quad \text{and} \quad \bbE \Big( \Big\Vert \frac{1}{n} S_{n} \Big\Vert_{\op} \Big) = O \Big( \sqrt{\frac{k \log k}{n}} \Big)$.
\end{lemma}

\subsection{Technical lemma related to matrix expansions}
\label{app:matrix lemma}

In this section, we present the Neumann series expansion of the inverse of a square matrix. The proof is elementary and is hence omitted.

\begin{lemma}
\label{lem:Neumann series}
Given two square, symmetric, and invertible matrices $A$ and $B$, the following holds 
\begin{align}
\label{resolvent identity}
A^{-1} - B^{-1} = \sum_{j = 1}^{J} B^{-1} \{(B - A) B^{-1}\}^{j} + \{B^{-1} (B - A)\}^{J + 1} A^{-1}.
\end{align}
Furthermore, suppose that
$\|A^{-1}\|_{\op} \vee \|B^{-1}\|_{\op} \lesssim 1$ and that there exists a diminishing sequence $\{r_{n}\}$ as $n$ increases such that $\Vert A - B \Vert_{\op} \lesssim r_{n}$. If $J$ is chosen such that $r_{n}^{J} = o (n^{-1 / 2})$, then
\begin{align}
\label{resolvent inequality}
A^{-1} - B^{-1} = \sum_{j = 1}^{J} B^{-1} \{(B - A) B^{-1}\}^{j} + o (n^{-1 / 2}).
\end{align}
\end{lemma}

\subsection{Technical lemma related to enumerative combinatorics}
\label{app:combinatorics}

Various useful results on enumerative combinatorics \citep{stanley2011enumerative} and combinatorial identities will be frequently invoked in the proofs of our theoretical results and are collected in this section. 

\begin{lemma}[M\"obius inversion on partition lattices]
\label{lem:mobius-inversion expansion}
Let $\iota \ge 1$ and let $\Pi_{\iota}$ denote the lattice of partitions of $\{1, \cdots, \iota\}$, ordered by refinement. 
Let $\calC$ be a finite index set, and let
$\{Z_{c} : c \in \calC\}$ be square matrices of the same dimension. More generally, the argument only requires that the $Z_{c}$'s belong to an associative algebra, so that the ordered products below are well-defined. The matrices $Z_{c}$ are not assumed to commute. The product below is always ordered in increasing $l$. For $\frakm \in \Pi_{\iota}$, write $B_{\frakm}(l)$ for the element of $\frakm$ containing $l$, and define $\mu_{\iota} (\frakm) \coloneqq (-1)^{\iota - |\frakm|} \prod_{B \in \frakm} (|B| - 1)!$. Then
\begin{equation}
\label{eq:mobius-partition-distinct-sum}
\sum_{\substack{(\ell_{1}, \cdots, \ell_{\iota}) \in \calC^{\iota}\\
\ell_{1} \neq \cdots \neq \ell_{\iota}}}
Z_{\ell_{1}} \cdots Z_{\ell_{\iota}}
= \sum_{\frakm \in \Pi_{\iota}} \mu_{\iota} (\frakm)
\sum_{\{b_{B}\}_{B \in \frakm} \in \calC^{|\frakm|}}
\prod_{l = 1}^{\iota} Z_{b_{B_{\frakm} (l)}}.
\end{equation}
In the inner summation on the RHS of \eqref{eq:mobius-partition-distinct-sum}, $b_{B}$ and $b_{B'}$ are allowed to be equal even when $B \ne B'$.
\end{lemma}

\begin{proof}
Let $\pi^{\dag}$ denote the discrete partition of $[\iota]$. For a tuple $\bm{\ell} = (\ell_{1}, \cdots, \ell_{\iota}) \in \calC^{\iota}$, let $\pi_{\bm{\ell}} \in \Pi_{\iota}$ denote its equality partition: two positions $s$ and $t$ belong to the same element of $\pi_{\bm{\ell}}$ if and only if $\ell_{s} = \ell_{t}$. For $\sigma \in \Pi_{\iota}$, define $$G (\sigma) \coloneqq \sum\limits_{\substack{\bm{\ell} \in \calC^{\iota} \\ 
\pi_{\bm{\ell}} = \sigma}} Z_{\ell_{1}} \cdots Z_{\ell_{\iota}}.$$ Thus the left-hand side of \eqref{eq:mobius-partition-distinct-sum} is $G (\pi^{\dag})$. For $\frakm \in \Pi_{\iota}$, define $$F (\frakm) \coloneqq \sum_{\{b_{B}\}_{B \in \frakm} \in \calC^{|\frakm|}} \prod_{l = 1}^{\iota} Z_{b_{B_{\frakm} (l)}}.$$ This is the sum over all tuples that are constant on every element of $\frakm$. Equivalently, their equality partition is coarser than $\frakm$. Hence, with the refinement order, $F (\frakm) = \sum\limits_{\sigma \in \Pi_{\iota}: \frakm \preceq \sigma} G (\sigma)$. Using the preceding relation between $F$ and $G$:
\begin{align*}
\sum_{\frakm \in \Pi_{\iota}} \mu_{\iota} (\frakm) F (\frakm) = \sum_{\frakm \in \Pi_{\iota}} \mu_{\iota} (\frakm) \sum_{\sigma \in \Pi_{\iota} : \frakm \preceq \sigma} G (\sigma) = \sum_{\sigma \in \Pi_{\iota}} \Big\{ \sum_{\frakm \in \Pi_{\iota} : \frakm \preceq \sigma} \mu_{\iota} (\frakm) \Big\} G (\sigma).
\end{align*}
We evaluate the inner coefficient. Fix $\sigma \in \Pi_{\iota}$. Every 
refinement $\frakm \preceq \sigma$ is obtained by partitioning each element $D \in \sigma$ independently. Therefore, by the definition of $\mu_{\iota}$, $$\sum_{\frakm \preceq \sigma} \mu_{\iota} (\frakm) = \prod_{D \in \sigma} \Big\{ \sum_{\pi \in \Pi (D)} (-1)^{|D| - |\pi|} \prod_{A \in \pi} (|A| - 1)! \Big\}.$$ For a finite set $D$ with $|D| = d$, define $a_{d} \coloneqq \sum_{\pi \in \Pi (D)} (-1)^{d - |\pi|} \prod_{A \in \pi} (|A| - 1)!$. The value of $a_{d}$ depends only on $d$. By the Taylor expansion, 
\begin{align*}
\sum_{d = 0}^{\infty} a_{d} \frac{z^{d}}{d !}
& = \exp \Big\{ \sum_{s = 1}^{\infty} (-1)^{s - 1} (s -1) !\frac{z^{s}}{s!} \Big\} = \exp \Big\{ \sum_{s = 1}^{\infty} (-1)^{s - 1} \frac{z^{s}}{s} \Big\} = \exp \{\log (1 + z)\} = 1 + z.
\end{align*}
Thus $a_{0} = 1, a_{1} = 1, a_{d} = 0$ for all $d \geq 2$. Consequently,
\begin{equation*}
\sum_{\frakm \preceq \sigma} \mu_{\iota} (\frakm) = \prod_{D \in \sigma} a_{|D|} = \begin{cases} 
1, & \sigma = \pi^{\dag},\\
0, & \sigma \neq \pi^{\dag}.
\end{cases}
\end{equation*}
Indeed, the product equals one if and only if every element of $\sigma$ is a singleton, namely $\sigma = \pi^{\dag}$; otherwise at least one element $D \in \sigma$ has cardinality at least two, and the corresponding factor $a_{|D|}$ is zero. Substituting this coefficient identity into the previous expansion gives $\sum_{\frakm \in \Pi_{\iota}} \mu_{\iota} (\frakm) F (\frakm) = G (\pi^{\dag})$. Recalling the definitions of $G (\pi^{\dag})$ and $F (\frakm)$, this is exactly \eqref{eq:mobius-partition-distinct-sum}.
\end{proof}

The next three lemmas are elementary and therefore stated without proof.
\begin{lemma}
\label{lem:comb_pre}
For any $u \in \bbR$ and any integers $y, \ell \geq 0$ with $y \geq \ell$, $\sum_{k = 0}^{y} \binom{y}{k} \binom{k}{\ell} u^{k - \ell} = \binom{y}{\ell} (1 + u)^{y - \ell}$.
\end{lemma}

\begin{lemma}
\label{lem:hockey-stick}
For any integers $0 \leq r \leq N$, $\sum_{k = r}^{N} \binom{k}{r}
= \binom{N + 1}{r + 1}$.
\end{lemma}

\begin{lemma}
\label{lem:combinatorial}
For any non-negative integers $r, x, y \geq 0$, $\sum_{k = 0}^{y} (-1)^{k} \binom{y}{k} \binom{x + k}{r} \equiv (-1)^{y} \binom{x}{r - y}$.
\end{lemma}

\end{document}